\documentclass{article}[11pt]
\usepackage[utf8]{inputenc}
\usepackage[main=english]{babel}
\usepackage{mathtools,amssymb,amsthm,amsfonts,upgreek}
\usepackage{geometry}
\usepackage{enumitem}
\usepackage{amsmath}
\usepackage{url}
\usepackage{fancybox}
\usepackage{fullpage}
\usepackage[colorlinks, linktocpage, citecolor=blue, linkcolor=blue]{hyperref}
\usepackage{multicol} 
\usepackage{array}
\usepackage{xargs}
\usepackage[prependcaption]{todonotes}
\usepackage[capitalise]{cleveref}
\newcommandx{\fab}[2][1=]{\todo[inline, author={Fabien}, linecolor=blue,backgroundcolor=blue!25,bordercolor=blue,#1]{#2}}
\newcommandx{\Max}[2][1=]{\todo[inline, author={Maxime}, linecolor=green,backgroundcolor=green!25,bordercolor=green,#1]{#2}}

\newtheorem{theoreme}{Theorem}[section]

\newtheorem{prop}{Proposition}[section]
\newtheorem{lem}{Lemma}[section]

\theoremstyle{remark}
\newtheorem{rem}{\bf Remark}[section]
\newcommand{\eps}{\varepsilon}
\renewcommand{\l}{\left}
\renewcommand{\r}{\right}
\newcommand{\un}{\underline}
\newcommand{\X}{X}
\newcommand{\Xge}{\bar{X}}
\newcommand{\B}{B}
\newcommand{\pas}{\gamma}
\renewcommand{\Pr}{\mathbb{P}}
\newcommand{\blip}{L}
\newcommand{\ER}{\mathbb{R}}
\newcommand{\ES}{\mathbb{E}}
\renewcommand{\a}{\alpha}
\newcommand{\at}{\tilde{\alpha}}
\renewcommand{\b}{\nabla\pot}
\newcommand{\com}[1]{}
\newcommand{\pot}{U}
\newcommand{\vpinf}{\un{\lambda}}

\newcommand{\Tge}{\bar{T}}

\newcommand{\ind}{j}
\newcommand{\lev}{J}

\newcommand{\Hrun}{(\mathbf{H_{r}^1})}
\newcommand{\Hrdeux}{(\mathbf{H_{r}^2})}
\newcommand{\Psiun}{\Psi}
\newcommand{\Psideux}{\bar{\Psi}}
\newcommand{\gamzero}{\gamma^\star}
\newcommand{\FFF}{\mathfrak{F}}
\title{(Non)-penalized Multilevel methods for non-uniformly log-concave distributions }
\author{Maxime \textsc{Eg\'ea}~\thanks{LAREMA, Facult\'e des Sciences, 2 Boulevard Lavoisier, Universit\'e d'Angers, 49045 Angers, France. E-mail: \texttt{maxime.egea@univ-angers.fr}}}

\begin{document}
\maketitle
\begin{abstract}
We study and develop multilevel methods for the numerical approximation of a log-concave probability $\pi$ on $\ER^d$, based on (over-damped) Langevin diffusion. In the continuity of \cite{art:egeapanloup2021multilevel} concentrated on the uniformly log-concave setting, we here study the procedure in the absence of the uniformity assumption. More precisely, we first adapt an idea of \cite{art:DalalyanRiouKaragulyan} by adding a penalization term to the potential to recover the uniformly convex setting. Such approach leads to an  \textit{$\varepsilon$-complexity} of the order $\eps^{-5} \pi(|.|^2)^{3} d$ (up to logarithmic terms). Then,  in the spirit of \cite{art:gadat2020cost}, we propose to explore the robustness of the method in a weakly convex parametric setting where the lowest eigenvalue of the Hessian of the potential $U$ is controlled by the function $\pot(x)^{-r}$ for $r \in (0,1)$. In this intermediary framework between the strongly convex setting ($r=0$) and the ``Laplace case'' ($r=1$), we show that with the help of the control of exponential moments of the Euler scheme, we can adapt some fundamental properties for the efficiency of the method. In the ``best'' setting where $U$ is ${\cal C}^3$ and $\pot(x)^{-r}$ control the largest eigenvalue of the Hessian, we obtain an $\varepsilon$-complexity of the order $c_{\rho,\delta}\varepsilon^{-2-\rho} d^{1+\frac{\rho}{2}+(4-\rho+\delta) r}$ for any $\rho>0$ (but with a constant $c_{\rho,\delta}$ which increases when $\rho$ and $\delta$ go to $0$).
\end{abstract}

\section{Introduction}
In this paper, we are interested in the sampling of probability distribution named Gibbs measure whose density is $\pi(\mathrm{d}x)= \frac{1}{Z}e^{-\frac{\pot(x)}{2\sigma^2}} \lambda(\mathrm{d}x)$ where $\lambda $ is the Lebesgue measure, $Z= \int_{\ER^d} e^{-\frac{2 \pot(x)}{\sigma^2}}\lambda(\mathrm{d}x)$ and $\pot:\ER^d\rightarrow \ER$ is a coercive function. Many applications require the computation of these measures in high dimension state space including for example machine learning, Bayesian estimation or statistical physics. Methods that are studied in this paper are based on the discretization of over-damped Langevin stochastic differential equation (\textbf{SDE}) 
\begin{equation}\label{art2:eq:SDE}
    \mathrm{d} X_t =- \nabla \pot(X_t) \mathrm{d}t + \sigma \mathrm{d}B_t ,
\end{equation}
where $ \left(B_t \right)_{t \ge 0} $ is a $d$-dimensional Brownian motion and $\sigma \in \ER^+_*$. These methods received a lot of attention in the last few years, in particular when $\pot$ is strongly convex (in the sense where, in the whole space, the smallest eigenvalue of its Hessian is lower bounded by a positive $\alpha$). This assumption may be certainly constraining in view of applications. It is the reason why, in this paper, we suppose that $\pot$ is not strongly convex but only weakly convex\footnote{We use in the sequel the usual terminology where ``strongly convex'' and ``weakly convex'' respectively means ``uniformly strongly convex'' and ``non-uniformly strongly convex''.}.
\newline More precisely, we will assume that the potential $\pot$ is a convex twice differentiable function with Lipschitz gradient. Under these assumptions, strong existence and uniqueness of a solution $\left(\X_t\right)_{t\ge0}$ classically hold and the solution to \eqref{art2:eq:SDE} is an ergodic Markov process whose invariant distribution is exactly the Gibbs distribution $\pi \propto e^{-U} \mathrm{d}\lambda$ (for background, see \emph{e.g.} \cite{book:meyn2012markov}, \cite{book:KaratzasShreve}, \cite{book:khasminskiiStab},\cite{note:hairerconvergence}).
\bigskip \newline We respectively denote by $(P_t)_{t\ge 0}$  and  $\mathcal{L}$ the related semi-group and infinitesimal generator. We recall that for a twice differentiable function $f : \ER^d \mapsto \ER$ by
\begin{equation*}
    \mathcal{L}f= -\left\langle \nabla \pot, \nabla f\right\rangle + \frac{\sigma^2}{2}\Delta f.
\end{equation*}
It is also well-known that in this log-concave setting,  the distribution $\pi$ satisfies the Poincar\'e inequality (see \emph{e.g.} \cite{art:bakry2008simple}) and that convergence holds to equilibrium in distribution and in ``pathwise average'': for any starting point $x\in\ER^d$, the occupation measure converges to $\pi$ in the following sense: for all continuous function $f \in L^2(\pi)$,
\begin{equation}\label{art2:eq:ergod}
    \frac{1}{t}\int_0^t f \left(\X_s^x\right) \mathrm{d}s \underset{t \rightarrow + \infty}{\longrightarrow} \pi(f)\quad a.s.
\end{equation}
In the continuity of \cite{art:egeapanloup2021multilevel}, our multilevel methods will be based on discretized adaptations of  \eqref{art2:eq:ergod}. More precisely, we first choose to approximate the stochastic process $(\X_t)_{t\ge 0}$ by the classical Euler-Maruyama scheme. When the related step size $\gamma$ is constant, this discretization scheme is defined by $\Xge_{0}=x\in \ER^d$ and:
\begin{equation*}
    \forall n\ge0,\quad    \Xge_{(n+1)\pas}=\Xge_{n\pas} - \pas \nabla \pot \left(\Xge_{n\pas}\right) + \sigma \sqrt{\pas}Z_{n+1},
\end{equation*}
where $\left(Z_n \right)_{n\in \mathbb{N}}$ is an $i.i.d$ sequence of $d$-dimensional standard Gaussian random variables. In the long-time setting, these schemes and there convergence properties to equilibrium were first studied in the nineties by \cite{art:talay1990second} and \cite{art:TweedieRoberts}. Then, some decreasing step Euler schemes were investigated by \cite{art:LambertonPages} (see also \cite{these:lemaire}) in order to manage, in the same time, the discretization and long-time errors. Here, we choose to keep the constant step size point of view in order to avoid some additional technicalities but our ideas could be probably adapted to this setting.
\bigskip \newline We also introduced the continuous-time extension of $(\bar{X}_{n\gamma})_{n\ge0}$ given by: for all $n\in\mathbb{N}$ and for all $t \in [ n\pas, (n+1)\pas )$,
\begin{equation*}
    \Xge_t^{\pas}: = \Xge_{n\pas} - (t-n\pas)\nabla \pot \left( \Xge_{n\pas}\right) + \sigma \left( \B_t - \B_{(n+1)\pas} \right). 
\end{equation*}
If we denote by $\un{t}_\gamma$ the discretization time related to a positive number $t$, \emph{i.e.},
\begin{equation}\label{art2eq:tbargamma}
    \un{t}_\gamma=\gamma \sup\{n\ge 1, n\gamma\le t\},
\end{equation}
we remark that 
\begin{equation*}
    \forall t\ge0, \quad \Xge_t = x - \int_0^t \nabla \pot(\Xge_{\underline{s}_\pas}) \mathrm{d}s+ \sigma \int_0^t  \mathrm{d}\B_s.
\end{equation*}
This  ``pseudo-diffusion'' form is usually convenient for proofs but it is worth noting that the procedure is only based on the discrete-time Euler scheme. If no confusion arises, we will sometimes write $\underline{t}$ instead of $\underline{t}_\pas$, and $\Xge_t$ or $\Xge_t^{\pas}$ instead of $\Xge_t^{\pas,x_0}$ to alleviate the notations. We now mimic \eqref{art2:eq:ergod} with the Euler scheme to approximate the target measure $\pi$. Thus, consider the following occupation measure (for background see \cite{art:talay1990second}), for $N \in \mathbb{N}$
\begin{equation*}
    \nu_N^\pas (f) := \frac{1}{N } \sum_{i=0}^{N-1} f \left( \Xge_{i\pas}\right). 
\end{equation*}

\subsection{Multilevel methods}
Multilevel methods introduced by M. Giles in \cite{art:giles2008multilevel}. These methods, initially used for the approximation of $\ES[f(X_T)]$, are now widely exploited in many settings. The rough idea is the following: assume that the target $\ES[\X]$ is the expectation of a random variable that cannot be sampled (with a reasonable cost) and consider a family of random variables $(\X_\ind)_{\ind}$ approximating $\X$, with a cost of simulation and a precision which typically increases with $j$. The principle of the multilevel is to stack correcting layers with a low variance to a coarse approximation $X_0$ of the target. More precisely, writing
\begin{equation}\label{art2:eq:telescopHeuri}
    \ES [\X_\lev] = \underbrace{\ES [\X_0]}_{\text{Coarse}} + \sum_{\ind=1}^\lev \underbrace{\ES [\X_\ind-\X_{\ind-1}]}_{\text{Correcting layer}},
\end{equation}
the multilevel method consists in building a procedure based on the addition of Monte-Carlo approximations of $\ES[X_0]$ and of $\ES[X_j-X_{j-1}]$, $j=1,\ldots,J$. Then, if the random variables $X_j-X_{j-1}$ have low variance, the approximation of $\ES[X_j-X_{j-1}]$ requires few simulations and, in view of \eqref{art2:eq:telescopHeuri}, we can obtain a procedure which has the bias related to $X_J$ but with a cost which may be much less than the one generated by a standard Monte-Carlo method applied to estimate $ \ES [\X_\lev]$.
\bigskip \newline In the discretization setting, the family of random variables $(\X_\ind)_{\ind}$ is a sequence of Euler schemes $(\Xge^{\pas_\ind})_{\ind}$ where $(\pas_\ind)_\ind$ is a family of decreasing time steps\footnote{Decreasing according the levels but constant for each layer.}. Following the heuristic \eqref{art2:eq:telescopHeuri}, the (independent) correcting layers are built by coupling of Euler schemes with steps $\gamma_{j-1}$ and $\gamma_j$. Note that in view of the simulation of the (synchronous) coupling, we need $\gamma_{j-1}$ to be a multiple of $\gamma_j$ (in this paper, we will assume that $\gamma_j=\gamma_0 2^{-j}$).

 \bigskip Multilevel methods have been already studied in the literature for the approximation of the invariant distribution of the Langevin diffusion. In \cite{art:gilesmajkameteuszSzpruch}, the authors take advantage of the convergence in distribution to equilibrium. Thus, the classical Monte-Carlo point of view is adopted: the approximation of $\pi(f)$ is obtained by sampling a large number of Euler schemes for each level. In  \cite{art:pagespanloup2018weighted}\footnote{This paper is written in the multiplicative setting with a so-called Multilevel-Romberg point of view.} and \cite{art:egeapanloup2021multilevel}, the point of view is to take advantage of the convergence of the occupation measure. Thus, each level is based on only one path of the Euler scheme or of the couple of Euler schemes whose length decreases (since the variance of the correcting layers decreases) and discretization step increases. All these papers show that, in the uniformly strongly convex setting, the invariant distribution can be approximated (along Lipschitz continuous functions) with a precision $\varepsilon$ (in a $L^2$-sense) using a Multilevel procedure whose complexity is of order $\varepsilon^{-2}$ or $\varepsilon^{-2}\log^p(\varepsilon)$ with $p\in[1,3]$. Moreover, in \cite{art:egeapanloup2021multilevel}, a particular attention is paid to the dependency in the dimension. In this case, it is shown that one can build a multilevel procedure that produces an $\eps$-approximation of the target for a complexity cost proportional to $d \eps^{-2}$ (with an explicit expression of the dependence in the Lipschitz constant $L$ and the contraction parameter $\alpha$).

\bigskip The more involved weakly convex case seems to be less explored in the multilevel paradigm but, in view of applications (for instance for Bayesian Lasso), it is natural to ask about the robustness of these methods when one relaxes the contraction assumption.

\subsection{Contributions and plan of the paper}
The main goal of this paper is thus to extend the multilevel Langevin algorithm for the Gibbs sampling to the weakly convex setting, and if  possible to obtain some quantitative bounds for the complexity related to the computation of an $\eps$-approximation of the target (see Section \ref{art2:notations} for a definition of $\varepsilon$-approximation). 

\bigskip We first investigate the \textit{penalized multilevel} method: in the continuity of \cite{art:DalalyanKaragulyan} and \cite{art:DalalyanRiouKaragulyan}, we build a multilevel procedure based on the following observation: consider a new equation with another potential $\pot_\a(x):= \pot(x)+\frac{\a}{2}|x|^2$, this new equation has an invariant distribution named $\pi_\a$ which converges to $\pi$ when $\a$ tends to $0$ in Wasserstein metric. The idea is that this new invariant distribution is easiest to sample because of the uniform convexity of the potential $\pot_\a$. In Section \ref{art2:sec:penalised}, Theorem \ref{art2:theo:PenaMain} combines the benefits of the penalized approach and of the multilevel methods. For a Lipschitz-continuous function $f:\ER^d \rightarrow \ER$ and a ${\cal C}^2$-potential $U$, the multilevel procedure performs an $\eps$-approximation of $\pi(f)$ with a complexity cost proportional to $\pi(|.|^2)^{3}d \eps^{-5}$. As in \cite{art:DalalyanRiouKaragulyan}, our result depends on the generally unknown constant $\pi(|.|^2)$ which is at least proportional to $d$ (see Remark \ref{rem:dalriou} for details and comparisons with \cite{art:DalalyanRiouKaragulyan}).

\bigskip Because of the above remarks, we chose in a second part to try to develop some tools which tackle the weakly convex setting from a dynamic point of view and which can improve the complexity in terms of $\varepsilon$. More precisely, in the spirit of \cite{art:gadat2020cost}, we study an intermediary framework (called weakly parametric convex setting in the sequel). We assume that the eigenvalues of the Hessian matrix of $U$ vanish when  $|x|$ goes to $+\infty$, but with a rate controlled by the function $x\mapsto \pot^{-r}(x)$ with $r \in [0,1)$ (see Assumption $\Hrun$). The parameter $r$ characterizes the ``lack of strong convexity'', the case $r=1$ referring to the ``Laplace case''\footnote{The ``Laplace case'' will refer to the setting where the potential has a flat gradient. The simplest example is $U(x)=|x|$. In this case, the invariant distribution is a Laplace distribution. That is the reason why we use this terminology.} whereas $r=0$ corresponds to the uniformly convex setting. When one assumes such an assumption, one can get some bounds for the exponential moments of the Euler scheme (at the price of technicalities). One is also able to preserve some \textit{confluence} properties, \emph{i.e.} two paths of the Euler scheme have a tendency to stick at infinity. Finally, in this setting, it is also possible to control the distance between diffusion paths and the related Euler schemes. These three ingredients (obtained with a lower quality than in the strongly convex setting) allow us to tackle the multilevel procedure in this framework.
\bigskip \newline The related main contribution is Theorem \ref{art2:theo:maintheo2}. In this result, we provide a series of statements under different sets of assumptions: when $U$ is only ${\cal C}^2$ or when $U$ is ${\cal C}^3$. Under $\Hrun$ only or under $\Hrun$ and $\Hrdeux$, where $\Hrdeux$ denotes an additional assumption which requires the highest eigenvalue to be also controlled by the function $x\mapsto \pot^{-r}(x)$ (we could roughly say that under $\Hrun$ and $\Hrdeux$, the potential is \textit{uniformly} weakly convex in the sense that ``the decrease of the contraction is uniform''). In each statement, we provide a multilevel procedure adapted to the assumptions. The related complexity is exhibited in terms of $d$, $\varepsilon$ but also in terms of the contraction parameter and the Lipschitz constant. Without going too much into the details, when $U$ is only ${\cal C}^2$, the complexity is of the order $\varepsilon^{-3}$ whereas when $U$ is ${\cal C}^3$, we can obtain a rate of the order $\varepsilon^{-2-\rho}$ for any $\rho>0$ and thus approach the ``optimal'' complexity $\varepsilon^{-2}$. Now, in terms of the dimension, when only  $\Hrun$ holds the dependence in the dimension of the complexity is bounded by
$d^{\frac{\frac{3}{2}+(\frac{9}{2}+\delta) r}{1-r}}$ whereas when $\Hrun$ and $\Hrdeux$ hold, we obtain in $d^{{\frac{3}{2}+(\frac{9}{2}+\delta) r}}$ when $U$ is ${\cal C}^2$
and  $d^{1+\frac{\rho}{2}+4r}$ for any $\rho>0$ when $U$ is ${\cal C}^3$. This means that when $U$ is ${\cal C}^3$ and $\Hrun$ and $\Hrdeux$ hold, the complexity is of the order $\varepsilon^{-2-\rho} d^{1+\frac{\rho}{2}+4r}$ for any $\rho>0$. With respect to the paper \cite{art:gadat2020cost}, our multilevel procedure improves the dependence in $\varepsilon$ and is most comparable in terms of the dimension. Note that when only $\Hrun$ holds, the dependency in the dimension dramatically increases on $r$. Whereas, when the potential is \textit{uniformly weakly} convex, the dependence in the dimension does not explode when $r\rightarrow 1$ (see Theorem \ref{rem:theoparam} for more details).
\bigskip \newline \textbf{Plan of the paper.} As detailed in the previous paragraphs, Sections \ref{art2:sec:penalised} and \ref{art2:sec:weaklyconvex} are respectively devoted to the statement of the main theorems for the penalized multilevel and in the parametric weakly convex case. Then, Section \ref{art2:sec:proofs} is dedicated to the proof of the first main theorem (Theorem \ref{art2:theo:PenaMain}). In Proposition \ref{art2:prop:alphaBias}, we obtain a Wasserstein bound related to the bias induced by the penalization on the invariant distribution. The proof of Theorem \ref{art2:theo:PenaMain} is then an adaptation of \cite[Theo 2.2]{art:egeapanloup2021multilevel}. From Section \ref{art2:sec:prelimtool}, we focus on the proof of Theorem \ref{art2:theo:maintheo2}. In Section \ref{art2:sec:prelimtool}, we prove some preliminary results on the diffusion and its Euler scheme under $\Hrun$: we begin with some controls of the exponential moments (Proposition \ref{art2:prop:exponentcontrol} and Proposition \ref{art2:prop:Eulexponentcontrol}) which in turn lead to some bounds of the polynomial moments (Proposition \ref{art2:prop:moment}). In this section, we also show that the discretization error can be controlled in long time (Proposition \ref{art2:prop:HypH2}) and finally obtain an integrable rate of convergence to equilibrium for the Euler scheme (Proposition \ref{art2:prop:HypH1}). With the help of these fundamental bricks, in Section \ref{art2:sec:maintheo2}, we obtain some bounds on the bias (Proposition \ref{art2:prop:HypH3}) and of the variance of the procedure (Proposition \ref{art2:varcontr}) which in turn allow us to finally provide the proof of Theorem \ref{art2:theo:PenaMain} in Theorem \ref{art2:theo:maintheo2}.

\subsection{Design of the algorithm}
We now build the multilevel procedure. Let $x\in \ER^d$ be the initialization of the procedure, $\lev \in \mathbb{N}$ be a number of levels, $\left(\pas_\ind\right)_{0 \le \ind \le \lev}$ be a sequence of time steps, $\left(T_\ind\right)_{0 \le \ind \le \lev}$ be a sequence of final times. Define by $ \mathcal{Y}(\lev,\left(\pas_\ind\right)_\ind,\tau,\left(T_\ind\right)_\ind,x,.)$ the \textit{multilevel occupation measure} : for all $f : \ER^d \rightarrow \ER$,
\begin{equation}\label{art2:eq:estimat}
    \begin{split}
        \mathcal{Y}(\lev,\l(\pas_\ind\r)_\ind,\tau,\l(T_\ind\r)_\ind,x,\a,f) &:= \frac{1}{T_0-\tau}\int_{\tau}^{T_0} f(\bar{X}_{\un{s}_{\pas_0}}^{\pas_0,x}) \mathrm{d}s
        \\ &+ \sum_{\ind=1}^\lev \frac{1}{T_\ind-\tau}\int_\tau^{T_\ind} f(\bar{X}_{\un{s}_{\pas_{\ind-1}}}^{\pas_{\ind},x}) -f(\bar{X}_{\un{s}_{\pas_{\ind-1}}}^{\pas_{\ind-1},x}) \mathrm{d}s,
    \end{split}
\end{equation}
where  trajectories on each level are coupled with the same Brownian motion. To ease notation, we simplify $\mathcal{Y}(\lev,\left(\pas_\ind\right)_\ind,\tau,\left(T_\ind\right)_\ind,x,f)$ by $\mathcal{Y}(f)$. The parameter $\tau \ge 0$ is the time where we begin the average. Indeed, this ``warm-start trick'' may improve the precision of the estimation in some cases (we refer to \cite{art:egeapanloup2021multilevel} for more details). But, it could be equal to $0$ when the gain is non-efficient, for example in the second part of our main result.

\subsection{Notations}\label{art2:notations}
The usual scalar product on $\ER^d$ and the induced Euclidean norm are respectively denoted by $\langle \cdot, \cdot\rangle$ and $|\cdot|$. The set $\mathcal{M}_{d,d}$ refers to the set of $d \times d$-real matrices, we denote by $\|\cdot\|$ the operator norm associated with the Euclidean norm.  For a symmetric matrix $A$, we denote respectively by $\un{\lambda}_A$ and $\bar{\lambda}_A$ its lowest and highest eigenvalues. The Frobenius norm for $A \in {\cal M}_{d,d}$ is denoted by $\|A\|_F$.\\

\noindent The Lipschitz constant of a given (Lipschitz) function $f:\ER^d\rightarrow\ER$ is denoted by $[f]_1$.
A function $f:\ER^d\rightarrow\ER$ is ${\cal C}^k$, $k\in\mathbb{N}$, if all its partial derivatives are well-defined and continuous up to order $k$. The gradient and Hessian matrix of $f$ are respectively denoted by $\nabla f$ and $D^2f$. The probability space is denoted by $(\Omega,{\mathcal{ F}},\mathbb{P})$. The Laplace operator is denoted by $\Delta$: $\Delta f=\sum_{i=1}^d \partial^2_{i,i} f$. The $L^p$-norm on $(\Omega,{\cal F},\mathbb{P})$ is denoted by $\|\cdot\|_p$. For two probability measures $\mu$ and $\nu$, we define the Wasserstein distance of order $p$ by 
\begin{equation*}
    \mathcal{W}_p(\mu,\nu) = \inf_{\zeta \in \Pi(\mu,\nu)} \left(\int_{\ER^d} |x-y|^p \mathrm{d}\zeta(x,y) \right)^{\frac{1}{p}},
\end{equation*}
where $\Pi(\mu,\nu)$ is the set of couplings of $\mu$ and $\nu$.\\
\begin{itemize}
\item{} \textbf{$\lesssim_{P}$ and $\lesssim_{uc}$}: For two positive real numbers $a$ and $b$ and a set of parameters $P$, one writes $a\lesssim_{P} b$ if $a\le C_P b$ where $C_P$ is a positive constant which only depends on the parameters $P$. When $a\le C b$ where $C$ is a universal constant, we write $a\lesssim_{uc} b$.
\item{}  \textbf{$\varepsilon$-approximation}: We say that ${\cal Y}$ is an  $\varepsilon$-approximation (or more precisely an $\varepsilon$-approximation of $a$ for the $L^2$-norm) if $\|{\cal Y}-a\|_2=\ES[|{\cal Y}-a|^2]^{\frac{1}{2}}\le \varepsilon$. Equivalently, ${\cal Y}$ is said to be an  $\varepsilon$-approximation of $a$ if the related Mean-Squared Error (MSE) is lower than $\varepsilon^2$.
\item{}  \textbf{Complexity/$\varepsilon$-complexity}: For a random variable ${\cal Y}$ built with some iterations of a standard Euler scheme, we denote by ${\cal C}(\cal Y)$, the number of iterations of the Euler scheme which is needed to compute ${\cal Y}$. For instance, ${\cal C}(\bar{X}_{n\pas}^{\pas})=n$. We sometimes call $\varepsilon$-complexity of the algorithm, the complexity related to the algorithm which produces an $\varepsilon$-approximation. 
\end{itemize}
\section{Main results}

\subsection{The penalized approach}\label{art2:sec:penalised} 
In this section, we develop a \textit{penalized multilevel method} to sample a non-strongly log-concave probability distribution $\pi$. The idea is based on \cite{art:DalalyanRiouKaragulyan} and \cite{art:karagulyan2020penalized} where the authors consider the potential $\pot^\a(x) := \pot(x)+\frac{\a}{2} |x|^2$  with $\a>0$ which is called \textit{penalized} version of $\pot$. We here assume that $\pot$ satisfies the following assumption:
 \bigskip \newline $\mathbf{WC_\blip}$: $\pot$ is a non-zero ${\cal C}^2$-function  and there exists $\blip > 0$ such that
\begin{equation}\label{art2:hyp:hessianpot}
    \forall x\in \ER^d, \qquad 0 \preccurlyeq  D^2\pot(x) \preccurlyeq \blip \mathrm{Id}_{\ER^d},
\end{equation}
the inequalities being  taken in the sense of symmetric matrices. Denote by $\pi_{\a}$ the invariant measure of the diffusion process $\left(\X_t^\a\right)_{t \ge 0}$ solution of the stochastic differential equation 
\begin{equation}\label{art2:eq:SDEr}
\mathrm{d} X_t^{\a} =- \nabla \pot^\a(X_t^\a) \mathrm{d}t + \sigma \mathrm{d}B_t.
\end{equation}
It appears that $\pi_\a$ satisfies the Bakry-\'Emery criterion thus we can apply our multilevel method that requires strong convexity to approximate $\pi_\a$. But our target is $\pi$ then we have to control the distance (in a Wasserstein sense) between $\pi$ and $\pi_\a$. To this end, the results of \cite{art:bolley2005weighted} and \cite{art:DalalyanRiouKaragulyan} ensure the convergence of $\pi_\a$ when $\a$ goes to $0$ with a bound of the Kullback-Leibler divergence. This leads to the following theorem: 
\begin{theoreme}\label{art2:theo:PenaMain} Assume that $\mathbf{WC_\blip}$ hold. Let $\varepsilon>0$, let $f : \ER^d \rightarrow \ER$ be a Lipschitz function and $x \in \ER^d$. For $\varepsilon \ge 0$ let
\begin{equation} \label{art2:eq:parameterPenawithoual}
    \begin{split}
        \lev_\varepsilon&= \lceil 2\log_2(\sigma^2 d m_4^{1/2} \varepsilon^{-2})\rceil,\quad T_\ind= \sigma^2 d\log(\pas_0^{-1})m_4 \varepsilon^{-5} \lev_\varepsilon^2 2^{-\lev}, 
        \\& \qquad \ind\in\{0,\ldots,\lev_\varepsilon\}, \quad  \gamma_0=\varepsilon m_4^{-1/2} L^{-2},
\end{split}
\end{equation}
with $m_4=\ES_{\pi}[|\cdot|^4]$. Then we have
\begin{equation}
 \| \mathcal{Y}(f) - \pi(f) \|_2^2 \le \varepsilon,
\end{equation}
with a complexity satisfying
\begin{equation}\label{art2:eq:complexitePenawithoutal}
{{\cal C}({\cal Y})\le \frac{1}{3}\log\l(\pas_0\r)m_4^{3/2} L^2\sigma^2 d\varepsilon^{-5}\left\lceil\log_2\l(\frac{1}{2}\sigma^2d m_4^{1/2} \varepsilon^{-3}\r)\right\rceil^3.}
\end{equation}
\end{theoreme}
\begin{rem}\label{rem:dalriou} By Borell's inequality (see \emph{e.g.} \cite{alonso-bastero}), $m_4\lesssim_{uc} \pi(|.|^2)^2$ for any log-concave probability. This means that, up to logarithmic terms, the complexity is controlled by $\pi(|.|^2)^3d \varepsilon^{-5}$. When $\pi$ is an \textit{isotropic probability}, $\pi(|.|^2)\lesssim_{uc} d$ (see \emph{e.g.} \cite{alonso-bastero}\footnote{Note that the control of $\pi(|.|^2)$ is strongly linked to the so-called \textit{KLS-conjecture}.}) and this implies that the complexity is of the order $d^4 \eps^{-5}$.
\smallskip \newline Let us now compare with  \cite{art:DalalyanRiouKaragulyan}: note that in this paper, the cost is not explicitly written. As usual in the Monte-Carlo literature, the authors control the number of iterations of the Euler scheme which is necessary to draw a random variable whose distance to the target is lower than $\varepsilon$ instead of giving the real cost. In the Langevin Monte-Carlo case, when $U$ is ${\cal C}^2$, they then obtain a number of iterations which is, up to logarithmic terms, of order $\pi(|.|^2)d\varepsilon^{-4}$. Normalizing $\varepsilon$ (\emph{i.e.} replacing $\varepsilon$ by $\varepsilon/\sqrt{\pi(|.|^2)}$) leads to a number of iterations of order $\pi(|.|^2)^{3} d\varepsilon^{-4}$. But to compare with our work, we need to include the Monte-Carlo cost, \emph{i.e.} to the number of simulations which is necessary to make the variance lower than $\varepsilon^{2}$. Then, we have to multiply the previous number of iterations by ${\rm Var}_\pi(f) \varepsilon^{-2}$ which can be reasonably bounded by $\pi(|.|^2)\varepsilon^{-2}$(when $f$ is Lipschitz). This means that the complexity of the penalized Langevin Monte-Carlo in \cite{art:DalalyanRiouKaragulyan} is of the order $\pi(|.|^2)^{4} d\varepsilon^{-6}$. In consequence, the multilevel method allows us to improve the result of \cite{art:DalalyanRiouKaragulyan}. Note that the authors also provide other algorithms such as the Kinetic-LMC where  the bound in $\varepsilon$ is improved (it seems that our result meets the complexity given for this algorithm).
\end{rem}
\noindent \textbf{About decreasing penalization.}In the above result, we propose a Multilevel strategy based on a fixed penalization. A natural question arises: could we take advantage of the Multilevel strategy by keeping the same penalization for the highest level and progressively reducing it on the lower layers? Indeed, this is precisely what we do with the discretization bias,  thus we can wonder about the effect of such a strategy for the penalization: to this end, let us introduce a decreasing sequence $(\a_\ind)_{0\le \ind \le \lev }$  of penalization levels such that the bias induced by the distance between $\pi_{{\a}_\lev}$ and $\pi$  is small with respect to the required precision $\varepsilon$. More specifically, we want to replace \eqref{art2:eq:telescopHeuri} with the following telescoping series
\begin{equation*}
    \pi^{\pas_\lev}_{{\a}_\lev}= \pi^{\pas_0}_{{\a}_0} + \sum_{\ind=1}^\lev \pi^{\pas_{\ind}}_{{\a}_\ind} - \pi^{\pas_{\ind-1}}_{\a_{\ind-1}}.
\end{equation*}
However, the above decomposition requires several longtime bounds on the underlying dynamics to be an efficient multilevel procedure. In particular, for the control of the variance generated by each level, we need to control the pathwise distance between two paths of the dynamics related to penalization levels ${\a}_\ind$ and ${\a}_{\ind-1}$. But, oppositely to the constant penalization case where we can obtain some \emph{confluence} properties,  we can observe that for two trajectories computed with two different degrees of penalization $\a$ and $\at$, there is a ``lack of confluence" quantified by the following inequality,
\begin{equation*}
        \ES \left[ \left|\X_t^{x, \a}-\X_t^{y,\at}\right|^2 \right] \le e^{-2 \a t}|x-y|^2 + \frac{ (\a-\at)d\sigma^2}{\at},
\end{equation*}
where the bound dramatically depends on $\a$. In this inequality we voluntary treated the continuous case to ease the readability but a discrete analogous result can be shown. We refer to the section \ref{art2:sec:proofs} to get the proof of this result. This inequality is certainly related to the shape of the penalization sequence $(\a_\ind)_{0\le \ind \le \lev }$. In fact, this result must be considered in addition of the error induced by the difference between two Euler schemes with different time steps. Up to an universal constant we have (see \cite[Prop 5.1]{art:egeapanloup2021multilevel})
\begin{equation*}
    \sup_{t \ge 0} \ES \left[ \left|\Xge_{\un{t}}^{x, \pas}-\Xge_{\un{t}}^{x,\pas/2}\right|^2 \right] \le \frac{\pas d \sigma^2}{\alpha^2}.
\end{equation*}
Then, the additional variance generated by the decrease of the penalization does not have an impact on the results, we have to impose that  
$ \displaystyle{\frac{\pas d \sigma^2}{\alpha_\ind^2}}\lesssim \displaystyle{ \frac{ (\a_{\ind}-\a_{\ind-1})d\sigma^2}{\a_{\ind-1}}}$. In particular, the sequence $(\a_\ind)_{0\le \ind \le \lev }$ cannot be too decreasing. Going further into the computations, it seems that  we cannot expect a significant gain  with this approach.

\subsection{Parametric weakly convex setting}\label{art2:sec:weaklyconvex}
The purpose of this section is to study the non-penalized multilevel procedure in the weakly convex setting. Instead of penalizing the dynamics, it is actually natural to ask about the robustness of the ``standard'' multilevel method in this case. 
 To answer this question, we have to prove a series of properties in the spirit of the assumptions  $\mathbf{H_i}$ ($i \in \{1,2,3,4\}$) in \cite{art:egeapanloup2021multilevel}. These assumptions include the convergence to equilibrium of the Euler scheme with a quantitative rate, the long-time control of the $L^2$-distance between the Euler scheme and the diffusion, the control of the Wasserstein distance between $\pi^\gamma$ and $\pi$ and the control of the moments. Some of these properties (especially the long-time control of the $L^2$-distance) seem  hard to check in a general convex setting. We thus propose to work in the parametric weakly convex setting used in  \cite{art:egeapanloup2021multilevel} by introducing $\Hrun$ (see below) where we assume that the contraction vanishes at $\infty$ but with a rate controlled by $U^{-r}$. 
\\

\noindent Let us now introduce our assumptions depending on a parameter $r\in[0,1)$ :\\

\noindent $\Hrun:$ The potential $\pot$ is a positive $C^2$-convex function with a unique minimum $x^\star$ such that $U(x^\star)=1$.  $\nabla U$ is $L$-Lipschitz {with $L\ge1$\footnote{The fact that $L$ is greater than $1$ is clearly not fundamental but allows to simplify the usually technical expressions which appear in the sequel.}}.   The function $x \mapsto \un{\lambda}_{D^2 \pot (x)}$ is positive and there exist $L$ and $\un{c} > 0$ such that, 
\begin{equation*}
    \forall x \in \ER^d, \quad \un{\lambda}_{D^2 \pot (x)} \ge \un{c} \pot^{-r}(x).  
\end{equation*}
The lower-bound can be seen as the ``lack of uniform strong convexity'' for the potential. Indeed, if $r=0$ we recover the strong convexity and $r=1$ corresponds to the weakest convexity case where the gradient is flat at infinity. \\

\noindent The fact that $\nabla U$ is $L$-Lipschitz implies that $x \mapsto \bar{\lambda}_{D^2 \pot (x)}$ is upper-bounded by $L$.  In order to improve the dependence in the dimension, we also introduce an additional assumption that deals with the case where the largest eigenvalue decreases at infinity with an intensity that is of the same order as the lowest eigenvalue:\\

\noindent $\Hrdeux:$ There exists a positive $\bar{c}$ such that  for all $x\in\ER^d$, $ \bar{\lambda}_{D^2 \pot (x)} \le     \bar{c}  \pot^{-r}(x).$\\

\noindent For instance, it can be checked that the function $x\mapsto (1+|x|^2)^p$ with $p\in(1/2,1]$ satisfies $\Hrun$ and $\Hrdeux$ with $r=\frac{1-p}{p}$, $\un{c}=2p(2p-1)$ and $\bar{c}=2p$. \\

\noindent Let us finally define the couple $(\gamzero,\Psideux)$ by:
\begin{equation}\label{art2:def:gamzero}
(\gamzero,\Psideux)=\begin{cases}
\left(\frac{1}{4L},(1+\sigma^2)\left(dL+\left(1+\frac{dL}{\un{c}}\right)^{\frac{1}{1-r}}\right)\right)&\textnormal{if only $\Hrun$ holds}\\
\left(\frac{1-r}{4 (\bar{c}\vee L)}, c_r\frac{ d(1+\sigma^2)(\bar{c}\vee L)}{\un{c}}\right)& \textnormal{if  $\Hrun$ and $\Hrdeux$ hold true},
\end{cases}
\end{equation}
where $c_r$ is a constant which only depends on $r$. $\gamzero$ will denote the largest value for $\gamma_0$ whereas, $\Psideux$ controls the moments of $U(\bar{X}_{n\gamma})$ (see Proposition \ref{art2:prop:moment} for details). It is worth noting that on the one hand, $\gamzero$ does not depend on $d$ and on the other hand, that 
$\Psideux \propto d^{\frac{1}{1-r}}$ when only $\Hrun$ holds and $\Psideux \propto d$ when $\Hrun$ and $\Hrdeux$ hold true. This means that in the first case, the dependence in $d$ dramatically increases with $r$ whereas in the second case, it does not depend on $r$. In the next result, the reader will have to keep in mind that 
the definition of these parameters depends on the assumptions. In particular, even  if $\Hrdeux$ does not appear in the statement, it is hidden in the value of the parameters $\gamzero$ and $\Psideux$.\\

We are now ready to state our main theorem in this setting:
\begin{theoreme}\label{art2:theo:maintheo2}
Assume $\Hrun$ and let $x \in \ER^d$ such that $U(x)\lesssim_{r} \Psideux$, $\pas_0 \in (0,\gamzero]$, $\delta \in (0,1/4]$ and let $f$ be a Lipschitz-continuous function. 
For an integer $J\ge 1$, set $\forall \ind \in \{1,\dots,\lev\}$,  $\pas_\ind= \pas_0 2^{-\ind}$ and $T_\ind= T_0 2^{-(1-\rho)\ind}$ with $\rho\in[1/2,1)$.
Let $\varepsilon>0$.
\begin{itemize}
\item[(i)] 
 Set $\rho=1/2$,
\begin{equation*}
            \lev  =\lceil  \log_2(\frac{L}{\un{c}^{\frac{2}{1-\delta}}\wedge\un{c}}\Psideux^{1+(3+\delta)r} \pas_0 \varepsilon^{-2})\rceil, \quad \textnormal{and}\quad T_0\coloneqq (\un{c}^{-\frac{3}{4}}\vee \un{c}^{-\frac{5}{2}-\delta}) \Psideux^{\frac{3}{2}+(\frac{9}{2}+\delta) r}\varepsilon^{-2}.
   \end{equation*}
Then, for $\delta$ small enough,
\begin{equation}\label{art2:eq:pluspetitqueepsilon}
    \| \mathcal{Y}(\lev,\l(\pas_\ind\r)_\ind,\tau,\l(T_\ind\r)_\ind,f)- \pi (f) \|_2 \lesssim_{r,\delta} \varepsilon,
\end{equation}
with a complexity cost,
\begin{equation}\label{art2:eq:complex2}
    \mathcal{C}(\mathcal{Y})\le\sqrt{L}(\un{c}^{-\frac{5}{4}}\wedge \un{c}^{-{\frac{7}{2}}-\delta})\Psideux^{\frac{3}{2}+(\frac{9}{2}+\delta)r} \varepsilon^{-3}.
\end{equation} 
\item[(ii)] Assume that $U$ is a ${\cal C}^3$-function with $\sup_{x\in\ER^d} \sum_{i=1}^d |\Delta (\partial_i U)|^2\lesssim_{r} \sigma^{-4} L^3 \Psideux$. Set
\begin{equation*}
J=\lceil  \log_2\left(\un{c}^{-\frac{2}{1-\delta}} L^{{3}}\Psideux^{1+\frac{2r}{1-\delta}} \pas_0 \varepsilon^{-1}\right)\rceil\quad \textnormal{and}\quad T_0=
L^{\frac{\rho}{2}} \left(\un{c}^{-(\frac{5}{4}-\rho)\wedge(3\rho)+\delta}\vee \un{c}^{-\frac{5}{2}-\delta}\right)\Psideux^{1+(4-2\rho+\delta) r} \varepsilon^{-2}.
\end{equation*}
Then, for $\delta$ small enough,     $\| \mathcal{Y}(\lev,\l(\pas_\ind\r)_\ind,\tau,\l(T_\ind\r)_\ind,f)- \pi (f) \|_2 \lesssim_{r,\delta,\rho} \varepsilon,$
with 
$$\mathcal{C}(\mathcal{Y})\le \gamma_0^{-\rho} L^{2\rho}  \left(\un{c}^{\frac{5}{4}\wedge(2\rho)+\delta}\vee \un{c}^{-\frac{5}{2}-\delta}\right)\Psideux^{1+\frac{\rho}{2}+(4-\rho+\delta) r} \varepsilon^{-2-\rho}.$$
\end{itemize}
\end{theoreme}

\begin{rem} \label{rem:theoparam}
This technical result deserves several comments:\\ 

$\rhd$ \textit{Complexity in terms of $\varepsilon$.} If we only consider the dependence in $\varepsilon$, we obtain $\varepsilon^{-3}$ when $U$ is only ${\cal C}^2$ and $\varepsilon^{-2-\rho}$ for any $\rho>0$ when $U$ is ${\cal C}^3$ and an additional (but reasonable\footnote{See Remark \ref{art2:rem:laplacienU} for details on this assumption}) assumption on $\Delta(\nabla U)$ is satisfied. We can thus theoretically approach the complexity in $\varepsilon^{-2}$. However, it is worth noting that the non-explicit constants depending on $\rho$ and $\delta$ go to $\infty$ (independently of the other parameters) when $\rho$ and $\delta$ go to $0$. The fact that we ``only'' obtain a complexity in $\varepsilon^{-3}$ when $U$ is ${\cal C}^2$ is due to the fact that in this case, our bound of the $1$-Wasserstein distance between $\pi^\gamma$ and $\pi$ is of the order $\sqrt{\gamma}$.  When $U$ is ${\cal C}^3$, the bound on the $1$-Wasserstein distance between $\pi^\gamma$ and $\pi$ is of order $\delta$. This allows us to clearly improve the complexity but it can be noted that we do not retrieve the $\varepsilon^{-2}$-bound of the uniformly convex case. This is due to the rate of convergence to equilibrium. Actually, our rate is polynomial and not exponential, which in turn, implies a ``slight cost'' on the dependence in $\varepsilon$.  In fact, we could get some (sub)-exponential rates but without controlling the dependence in the dimension which is of first importance for applications. \\

$\rhd$ \textit{Complexity in terms of the dimension.} The dependence in the dimension strongly varies with the assumptions. In the ``worst'' case where $U$ is only ${\cal C}^2$ and only $\Hrun$ holds, the complexity is of the order $d^{\frac{\frac{3}{2}+(\frac{9}{2}+\delta) r}{1-r}}$. Unfortunately, when $r$ is close to $1$, this means that this dependence seriously worsens. We retrieve this same phenomenon when $U$ is ${\cal C}^3$ and only $\Hrun$ holds but with a better bound of the order $d^{\frac{{1+\frac{\rho}{2}+(4-\rho+\delta) r}}{1-r}}$ for any positive $\rho$ and $\delta$. This bad behavior when $r$ goes to $1$ is due to the fact that when only $\Hrun$ holds, the bounds on the exponential moments of $\sup_{t\ge 0}\ES[e^{U(\bar{X}_{\un{t})}}]$ are of the order $\exp({d^{\frac{1}{1-r}}})$. Introducing $\Hrdeux$ dramatically improves this exponential bound since in this case, we are able to prove that this is of the order $e^d$ (this implies that $\sup_{t\ge 0}\ES[U^p(\bar{X}_{\un{t}})]$ is of the order $d^p$, see Propositions \ref{art2:prop:moment} and \ref{art2:prop:Eulexponentcontrol} for details). It is worth noting that in this case, the dependence in the dimension does not explode when $r$ goes to $1$ being of the order $d^{{1+\frac{\rho}{2}+(4-\rho+\delta) r}}$ for any $\rho>0$. Remark that when $r=0$, we formally approach the rate of the uniformly convex case in $d\varepsilon^{-2}$ obtained in \cite{art:egeapanloup2021multilevel}. \\

$\rhd$ \textit{Comparison with the literature:}  In this setting, the only paper which we may reasonably compare with is \cite{art:gadat2020cost} since we use similar assumptions. Compared with this paper, our multilevel procedure certainly improves the dependence in $\varepsilon$, replacing $\varepsilon^{-4}$ by $\varepsilon^{-3}$ when $U$ is ${\cal C}^2$ and $\varepsilon^{-3}$ by $\varepsilon^{-2-\rho}$ when $U$ is ${\cal C}^3$. In terms of the dimension, our approach slightly increases the dependence on the dimension. For instance, when $\Hrun$ and $\Hrdeux$ hold, \cite{art:gadat2020cost} obtain a bound in $d^{1+4r}$ when $U$ is ${\cal C}^2$ or ${\cal C}^3$. We here retrieve a dependence which is somewhat similar when $U$ is ${\cal C}^3$ but when $U$ is ${\cal C}^2$, our bound in $d^{\frac{3}{2}+(\frac{9}{2}+\delta) r}$ is clearly worse.\\

$\rhd$ \textit{About the parameters.} In applications, the dependence in the parameters, $L$, $\un{c}$, and $\bar{c}$ may be of importance (think for instance to applications to Bayesian estimation where these parameters can strongly depend on the number of observations). This is why here, we chose to keep all these dependencies in the main result even if it sometimes adds many technicalities in the proof.

\end{rem}

\section{Proof of Theorem \ref{art2:theo:PenaMain}} \label{art2:sec:proofs}
This section is devoted to the proof of the first main result. We first quantify the bias induced by the approximation of $\pi$ by $\pi_\a$. To this end, we use the Talagrand concentration inequality that estimates the Wasserstein distance between these two measures by their Kullback-Leibler divergence.  
\begin{prop} \label{art2:prop:Pinsker}
    Assume that $\mathbf{WC_\blip}$ hold. Then for all $\a \ge 0$, there is a universal constant $C$ such that
    \begin{equation*}
        \mathcal{W}_1 (\pi,\pi_\a) \le  C \sqrt{D_{\mathrm{KL}} (\pi| \pi_\a )} .
    \end{equation*}
\end{prop}
We refer to \cite[Cor 2.4]{art:bolley2005weighted} to find a proof of this result. In addition, in \cite{art:DalalyanRiouKaragulyan} the authors show that $C \le 2\ES_{\pi}[|X|^2]$ (page 24).
It remains to compute the Kullback Leibler divergence of $\pi$ from $\pi_\a$, to bound the bias induced by the penalization.
\begin{prop}\label{art2:prop:alphaBias}
    Assume that $\mathbf{WC_\blip}$ hold. Then for all $\a \ge 0$,
    \begin{equation*}
        \mathcal{W}_1 (\pi, \pi_\alpha) \le  \frac{\alpha}{2\sqrt{2}} \ES_{\pi} [|.|^4]^{1/2}.
    \end{equation*}
\end{prop}
\begin{proof}
The Kullback-Leibler divergence is defined by
\begin{equation*}
D_{\mathrm{KL}} (\pi| \pi_\a ) = \int_{\ER^d} \mathrm{log} \left( \frac{\mathrm{d}\pi}{\mathrm{d}\pi_\a}(x)\right) \pi(\mathrm{d}x).
\end{equation*}
By definition of $\pi$ and $\pi_\a$,
\begin{equation*}
    \begin{split}
        D_{\mathrm{KL}} (\pi| \pi_\a ) &= \int_{\ER^d}  \mathrm{log} \left( \frac{Z_\a}{Z} \right) + \frac{\a}{2} |x|^2 \pi(\mathrm{d}x)
        \\ & = \mathrm{log} \left( \ES_{\pi} \left[ e^{-\frac{\alpha}{2}|.|^2} \right] \right) + \frac{\alpha}{2}  \ES_{\pi} \left[|.|^2\right].
    \end{split}
\end{equation*}
Using the inequality $e^{-x} \le 1-x+\frac{x^2}{2}$ for $x \ge 0$, this leads to,
\begin{equation*}
    D_{\mathrm{KL}} (\pi| \pi_\a ) \le \mathrm{log} \left( 1 - \ES_{\pi} \left[ \frac{\a}{2} |.|^2 - \frac{\alpha^2}{8}|.|^4 \right] \right) + \frac{\a}{2}  \ES_{\pi} \left[|.|^2\right],
\end{equation*}
and by the inequality $\mathrm{log}(1-x) \le -x$ for $x \le 1$ we get,
\begin{equation*}
        D_{\mathrm{KL}}(\pi|\pi_\a) \le \frac{\a^2}{8} \ES_{\pi} [|.|^4].
\end{equation*}
Proposition \ref{art2:prop:Pinsker} for $\pi_\a$ implies the result.
\end{proof}

Now we switch to the proof of the main theorem. With the two previous propositions, we control the bias induced by the penalization, then it remains to compute the error and the complexity of a multilevel procedure in a uniformly convex setting. To this end, we use \cite[Theorem 2.2]{art:egeapanloup2021multilevel} which gives parameters to perform an $\varepsilon$-approximation of the invariant distribution with an explicit complexity in terms of the parameters, especially in terms of the \textit{contraction parameter}. Here, this is exactly our penalization parameter  $\a$ and we will thus optimize its choice in the proof.
\begin{proof}\textit{of Theorem \ref{art2:theo:PenaMain}}
Let $\varepsilon$ be a positive number and $f:\ER^d\rightarrow\ER$ be a Lipschitz continuous function. By the bias/variance decomposition, triangular inequality and the Monge-Kantorovich duality, we have
\begin{equation}\label{art2:eq:decomperr}
    \begin{split}
        \| \mathcal{Y}(f) - \pi(f) \|_2^2 & \le \ES \left[ \left|\mathcal{Y}(f) - \pi(f)\right|\right]^2 + \mathrm{Var}( \mathcal{Y}(f))
        \\ & \le  2|\pi_{\a}(f)- \pi(f)|^2 + 2 \ES [\mathcal{Y}(f) - \pi_{\a}(f)]^2 + \mathrm{Var}( \mathcal{Y}(f)))
        \\ &  \lesssim_{uc} \underbrace{\mathcal{W}_1^2(\pi,\pi_\a)}_{P_1} + \underbrace{ \ES [\mathcal{Y}(f) - \pi_{\a}(f)]^2  +  \mathrm{Var}( \mathcal{Y}(f))}_{P_2}.
    \end{split}
\end{equation}
The second term denoted by $P_2$ is the mean squared error of a Multilevel procedure for the approximation of $\pi_\a$. This penalized measure is invariant for the diffusion process defined with the potential $\pot_\a$. By assumption $\mathbf{WC_L}$, $\pot_\a$ satisfies the following property
\noindent \\ $\mathbf{\mathbf{C_\a}}$: For all $x,y \in \ER^d$,
\begin{equation*}
    \langle \nabla \pot^{\a}(x)-\nabla\pot^{\a}(y),x-y\rangle \ge \a |x-y|^2.
\end{equation*}
\cite[Theorem 2.2]{art:egeapanloup2021multilevel} ensures that with $\alpha/L^2\le 1$, $\sigma^2 \alpha^{-1} d\ge 1$, $|\pot^{\a} (x_0)|^2\lesssim_{uc} \sigma^2 \alpha d$ and the following parameters\footnote{To ease notation we have voluntarily omitted an assumption about $\varepsilon$: we have to consider $\varepsilon$ small enough, we refer to \cite{art:egeapanloup2021multilevel} to get more precision. }:
\begin{equation} \label{art2:eq:parameterPena}
    \begin{split}
        \lev_\varepsilon&= \lceil 2\log_2(\sigma^2 d \a^{-1} \varepsilon^{-1})\rceil,\quad T_\ind= \frac{\sigma^2 d \log(\pas_0^{-1})}{\a^2} \varepsilon^{-2} \lev_\varepsilon^2 2^{-\ind}, 
        \\& \qquad r\in\{0,\ldots,\lev_\varepsilon\}, \quad  \gamma_0=\a/(2L^2),
\end{split}
\end{equation}
we have
\begin{equation*}
    P_2 \le \varepsilon,
\end{equation*}
with a complexity satisfying
\begin{equation}\label{art2:eq:complexitePena}
    {\cal C}({\cal Y})\le 5\log\left(\frac{2 L^2}{\a}\right)\frac{L^2}{2\a^3}\sigma^2 d\varepsilon^{-2}\left\lceil\log_2\left( \frac{\sigma^2d}{\a} \varepsilon^{-2}\right)\right\rceil^3.
\end{equation}
It remains to calibrate the penalization parameter $\a$. Proposition \ref{art2:prop:alphaBias} implies  
\begin{equation*}
    P_1  \lesssim_{uc}  \frac{\a^2}{4} \ES_{\pi} [|.|^4],
\end{equation*}
so that $P_1 \le  \varepsilon^2$ for 
\begin{equation*}
    \a = \frac{2 \varepsilon}{ \ES_{\pi} [|.|^4]^{1/2}}.
\end{equation*}
Putting $\a$ in \eqref{art2:eq:parameterPena} and \eqref{art2:eq:complexitePena},
\begin{equation*}
    {\cal C}({\cal Y})\le \frac{5d\varepsilon^{-5}}{16}\log\left( \ES_{\pi} [|.|^4]^{1/2} \blip^2 \varepsilon^{-1}\right)\ES_{\pi} [|.|^4]^{3/2} \blip^2\sigma^2 \left\lceil\log_2\left(\frac{1}{2}  \ES_{\pi} [|.|^4]^{1/2}\sigma^2d \varepsilon^{-3}\right)\right\rceil^3.
\end{equation*}
\end{proof}
\subsection*{Precisions about the ``decreasing penalization":}
For $\a>\at>0$ and  $x,y \in \ER^d$ consider the couple $(\X_t^{x,\a},\X_t^{x,\at})_{t\ge0}$ defined by
\begin{equation}
    \begin{cases}
    &\X_t^{x,\a}=x+\int_0^t \nabla \pot_{\a}(\X_s^{x,\a}) \mathrm{d}s+\sigma \B_t\\
    &\X_t^{y, \at}=y+\int_0^t \nabla \pot_{\at}(\X_s^{y,\at}) \mathrm{d}s+\sigma \B_t.
\end{cases}
\end{equation}

\begin{prop}
For all $t>0$ we have 
\begin{equation*}
        \ES \left[ \left|\X_t^{x, \a}-\X_t^{y,\at}\right|^2 \right] \le e^{-2 \a t}|x-y|^2 + \frac{ (\a-\at)d\sigma^2}{\at}.
\end{equation*}
\end{prop}
\begin{proof}
By the Itô formula we have,
\begin{equation*}
    \begin{split}
        e^{2 \a t}\left|\X_t^{x, \a}-\X_t^{y,\at}\right|^2=|x-y|^2 & + \int_0^t 2 e^{2 \a s} \left\langle \b_{\a} \left(\X_s^{x, \a}\right)-\b_{\at} \left(\X_s^{y,\at}\right),\X_s^{y,\at} -\X_s^{x, \a}\right\rangle \mathrm{d}s 
        \\ & +  \int_0^t 2 \a e^{2 \a s}\left|\X_s^{x, \a}-\X_s^{y,\at}\right|^2  \mathrm{d}s,
    \end{split}
\end{equation*}
We now use the fact that for all $ x\in \ER$ : $\b_{\at}(x)=\b_{\a}(x) + \left(\at - \a \right)|x|$ which yields
\begin{equation*}
    \begin{split}
        e^{2 \a t}\left|\X_t^{x, \a}-\X_t^{y,\at}\right|^2=|x-y|^2 & +\int_0^t 2 e^{2 \a s} \left\langle \b_\a\left(\X_s^{x, \a}\right)-\b_\a\left(\X_s^{y,\at}\right),\X_s^{y,\at} -\X_s^{x, \a}\right\rangle \mathrm{d}s 
        \\ & + \int_0^t 2(\a-\at) e^{2 \a s} \left\langle \X_s^{y,\at},\X_s^{y,\at} -\X_s^{x, \a}\right\rangle \mathrm{d}s 
        \\ & +  \int_0^t 2 \a e^{2 \a s}\left|\X_s^{x, \a}-\X_s^{y,\at}\right|^2  \mathrm{d}s.
    \end{split}
\end{equation*}
The strong convexity property of $\pot_\a$: $\langle \b_\a(x)-\b_\a(y),x-y\rangle \le \a |x-y|^2$ implies
\begin{equation*}
    \begin{split}
        e^{2 \a t}\ES \left[ \left|\X_t^{x, \a}-\X_t^{y,\at}\right|^2 \right]& \le |x-y|^2  +\int_0^t 2(\a-\at) e^{2 \a s} \left\langle \X_s^{y,\at},\X_s^{y,\at} -\X_s^{x, \a}\right\rangle \mathrm{d}s
        \\ & \le |x-y|^2 +\int_0^t 2(\a-\at) e^{2 \a s} \ES \left[ \left| \X_s^{y,\at}\right|^2 \right]\mathrm{d}s.
    \end{split}
\end{equation*}
Up to an universal constant, the moment of order two of the diffusion process under the strong convexity hypothesis are bounded by $\frac{\sigma^2 d }{\at}$ (see \cite[Lem 5.1]{art:egeapanloup2021multilevel}), we get
\begin{equation*}
        \ES \left[ \left|\X_t^{x, \a}-\X_t^{y,\at}\right|^2 \right] \le e^{-2 \a t}|x-y|^2 + \frac{ 2(\a-\at)d\sigma^2}{\at}\int_0^t e^{2\a (s-t)}\mathrm{d}s.
\end{equation*}
The result follows.
\end{proof}


\section{Preliminary bounds under \texorpdfstring{$\Hrun $ and $\Hrdeux$}{Hr}:}\label{art2:sec:prelimtool}
From now, we switch to the proof of the second part of the main results \textit{i.e.} we consider the weakly convex case under the parametric assumptions $\Hrun$ and $\Hrdeux$. As mentioned before, these hypotheses deal with the behavior of the lowest and highest eigenvalues of the Hessian matrix of $\pot$. In some sense, $\Hrun$ quantifies the strict convexity of the potential which in turn implies the contraction of the dynamics. Note that such an assumption also appears in \cite{cattiaux:hal-02486264} where the authors obtain exponential rates to equilibrium under this parametric assumption.
\bigskip \newline In this preliminary section we state a series of results related to the diffusion and its Euler scheme under Assumption $\Hrun$. For the upper-bounds of the eigenvalues of $D^2 U$, we distinguish two cases: the first one where we assume that we have a uniform upper-bound by $L$ (in others words that $\nabla U$ is $L$-Lipschitz) and the second one, we add Assumption $\Hrdeux$ where the largest eigenvalues also decrease at infinity with a rate which is comparable to the one of the lowest eigenvalues.
In fact, in the second case, we will see that we are able to preserve a dependency of the moments in the dimension which is linear, whereas, without this assumption, the dependency is $O(d^{\frac{1}{1-r}})$.
\bigskip \newline In the second part we state a result about the longtime pathwise control of the distance between the diffusion and its Euler scheme. Third, we study the convergence to equilibrium for the Euler scheme. Finally, we quantify the bias induced by the discretization with some results on the $1$-Wasserstein distance between $\pi$ and $\pi^\pas$ (the invariant measure of the Euler scheme).

\subsection{Bounds on the exponential moment}
In order to study the confluence between the continuous time process and its Euler scheme, let us start this section by a control of the moment of the continuous time process and the discrete time when the potential $\pot$ is supposed convex. We first state a result on the control of the exponential moment of the continuous time process.
\begin{prop} \label{art2:prop:exponentcontrol}
    For all $x \in \ER^d$ and $t>0$,
    \begin{equation*}
        \sup_{t\ge0} \ES_x \left[e^{ \frac{\pot\left(\X_t\right)}{\sigma^2}} \right] \le e^{\frac{U(x)}{\sigma^2}}+ \begin{cases} e^{\frac{1}{\sigma^2} \left(1+\frac{ dL}{2\un{c}}\right)^{\frac{1}{1-r}}}&\textnormal{under $\Hrun$}\\
        \frac{d L}{2\un{c}} e^{\frac{1}{\sigma^2} (4^{\frac{1}{1-r}}\vee \frac{d \bar{c}}{\un{c}})}
        &\textnormal{under $\Hrun$ and $\Hrdeux$.}
        \end{cases}
    \end{equation*}
\end{prop}
We preface the proof by a technical lemma.
\begin{lem} \label{art2:lem:compactset}
    Let $\theta \in (0,\frac{2}{\sigma^2})$ and $M>0$, then
    \begin{equation*}
        {\cal C}_M := \left\{x \in \ER^d ; (1-\theta)\left|\nabla \pot(x)\right|^2- \frac{\sigma^2}{2} \theta  \Delta \pot(x) \le M \right\}\subset\left\{x\in\ER^d, U(x)\le K_M\right\},
    \end{equation*}
    where, 
\begin{equation*}
    K_M=\begin{cases}\left(1+\frac{ (2M+\theta \sigma^2 dL)}{2\un{c}(2-\theta\sigma^2)}\right)^{\frac{1}{1-r}} &\textnormal{under $\Hrun$}\\
    \max\left(\left(\frac{M(1-r)}{(4-2\theta\sigma^2)\un{c}}\vee 4\right)^{\frac{1}{1-r}},\frac{2\sigma^2 d\theta \bar{c}}{\un{c}(2-\theta\sigma^2)}\right) &\textnormal{under $\Hrun$ and $\Hrdeux$.}
    \end{cases}
\end{equation*}
In particular, ${\cal C}_M$ is a compact set (since it is included in a level set of a coercive function).
\end{lem}
\begin{proof}
Denote by $y$ the solution of the ordinary differential equation $y'(t)=-\nabla \pot (y(t))$ starting from $y(0)=x$. Define the function $f: t \mapsto |\nabla \pot(y(t))|^2 $ we have by chain rule for all $t\in \ER^+$. By $\Hrun$,
\begin{equation*}
    \begin{split}
        \frac{\mathrm{d}}{\mathrm{d}t}f(t) &= 2 \left\langle y''(t), y'(t) \right\rangle
        \\ & =  2\left\langle D^2 \pot(y(t))y'(t), y'(t) \right\rangle
        \\ & \ge 2\un{c} \left\langle \pot^{-r}(y(t))y'(t), y'(t) \right\rangle = \frac{2 \un{c}}{1-r} \frac{\mathrm{d}}{\mathrm{d}t} \big( \pot^{1-r}(y(t)) \big),
    \end{split}
\end{equation*}
Since $\displaystyle{\lim_{t \rightarrow + \infty} y(t)=x^\star}$, we get by integration
\begin{equation}\label{art2:eq:minogradU}
    |\nabla \pot(x)|^2 \ge \frac{2\un{c}}{1-r}\left( \pot^{1-r}(x) - \pot^{1-r}(x^\star) \right).
\end{equation}
Therefore,
\begin{equation*}
     \left(1-\frac{\theta\sigma^2}{2}\right)\left|\nabla \pot(x)\right|^2- \frac{\sigma^2\theta}{2}  \Delta \pot(x) \ge  \frac{2(1-\theta)\un{c}}{1-r}\left( \pot^{1-r}(x) - \pot^{1-r}(x^\star) \right) - \frac{\sigma^2\theta}{2}   \Delta \pot(x).
\end{equation*}
Since $ \Delta \pot={\rm Tr}(D^2 U)\le d\bar{\lambda}_{D^2 U}\le d L$ (where $\bar{\lambda}_A$ stands for the largest eigenvalue of symmetric matrix $A$) and $\pot(x^\star)=1$, it follows that
\begin{equation*}
    \left(1-\frac{\theta\sigma^2}{2}\right)\left|\nabla \pot(x)\right|^2- \frac{\sigma^2}{2} \theta  \Delta \pot(x)\ge \frac{2\left(1-\frac{\theta\sigma^2}{2}\right){\un{c}}}{1-r}(U^{1-r}(x)-U^{1-r}(x^\star))-\frac{\theta \sigma^2 dL}{2},
\end{equation*}
so that 
\begin{equation}
    {\cal C}_M\subset \left\{x\in\ER^d; U(x)\le  \left(1+\frac{2M+\theta \sigma^2 dL}{2\un{c}(2-{\theta\sigma^2}) }\right)^{\frac{1}{1-r}}\right\}.
\end{equation}
If we now consider the case where $\Hrdeux$ also holds, we use that $ \Delta \pot\le\un{c} U^{-r}(x)$ to obtain:
\begin{equation*}
    \begin{split}
        \left(1-\frac{\theta\sigma^2}{2}\right)\left|\nabla \pot(x)\right|^2- \frac{\sigma^2}{2} \theta  \Delta \pot(x)
        \  \ge\frac{(2-\theta\sigma^2)\un{c}}{1-r} \pot^{1-r}(x) \left(  1-\pot^{r-1}(x)- \frac{(1-r) \sigma^2 \theta \bar{c}d}{2\un{c}(2-\theta\sigma^2})U(x)\right).
    \end{split}
\end{equation*}
To ensure that the right-hand member is lower-bounded by $M$, it is enough to ensure that
\begin{equation*}
    \frac{(2-\theta\sigma^2)\un{c}}{1-r} \pot^{1-r}(x)>\frac{M}{2},\quad U^{r-1}(x)\le \frac{1}{4}\quad \textnormal{and} \quad \frac{ \sigma^2 \theta \bar{c} d}{\un{c}(4-2\theta\sigma^2})U(x) \le \frac{1}{4}.
\end{equation*}
This concludes the proof.
\end{proof}
\begin{proof}\textit{(of Proposition \ref{art2:prop:exponentcontrol})}
Let $\theta \in (0,1)$, (to be choosen latter) and for all $x \in \ER$ define, 
\begin{equation} \label{art2:eq:LyapFunction}
    f_{\theta}(x):=e^{\theta \pot(x)},
\end{equation}
show that $f_{\theta}$ is a Lyapunov function for the dynamic $\mathcal{L}$ :
\begin{equation*}
    \begin{split}
        \mathcal{L}f_{\theta}(x) & = -\theta f_{\theta}(x)\left(\left(1-\frac{\theta\sigma^2}{2}\right)\left|\nabla \pot(x)\right|^2- \frac{ \theta  \sigma^2}{2} \Delta \pot(x) \right)
        \\ & \le - \theta M f_\theta(x)   -\theta f_{\theta}(x)\left(\left(1-\frac{\theta\sigma^2}{2}\right)\left|\nabla \pot(x)\right|^2- \frac{ \theta  \sigma^2}{2} \Delta \pot(x) \right) 1_{{\cal C}_M}(x),
    \end{split}
\end{equation*}
where ${\cal C}_M$ is defined in Lemma \ref{art2:lem:compactset}. In the proof of this lemma we showed that ${\cal C}_M$ is included in a level set $L_{K_M} = \{x \in \ER^d \mid \pot(x) \le K_M  \}$. Thus,
\begin{equation*}
    \begin{split}
        -\theta f_{\theta}(x)\left(\left(1-\frac{\theta\sigma^2}{2}\right)\left|\nabla \pot(x)\right|^2- \frac{\sigma^2}{2} \theta  \Delta \pot(x) \right) 1_{{\cal C}_M}(x) & \le   \frac{ \theta^2 \sigma^2}{2} f_{\theta}(x)  \Delta \pot(x) 1_{L_{K_M}}
        \\ & \le  \frac{ \theta^2 \sigma^2 d L }{2} e^{\theta K_M}1_{L_{K_M}}.
    \end{split}
\end{equation*}
Finally $f_\theta$ is a Lyapunov function for the dynamics: \textit{i.e.} for all $x\in \ER^d$,
\begin{equation*}
    \mathcal{L}f_{\theta}(x)  \le - \theta M f_\theta(x)   + \frac{ \theta^2 \sigma^2 d L }{2} e^{\theta K_M}   1_{L_{K_M}}(x),
\end{equation*}
Hence by a Gronwall argument we get,
\begin{equation*}
    P_t f_\theta(x) \le e^{-\theta M t} f_\theta(x) + \frac{ \theta \sigma^2 d L}{2M} e^{\theta K_M}.
\end{equation*}
 Choosing $\theta=\frac{1}{\sigma^2}$, $M=\frac{\theta\sigma^2 dL}{2} $ under $\Hrun$ and $M=\frac{(4-2\theta\sigma^2)\un{c}}{1-r} $ under $\Hrun$ and $\Hrdeux$ leads to the result.
\end{proof}
\noindent We now state an analogous result for the  Euler scheme.
\begin{prop} \label{art2:prop:Eulexponentcontrol}
\begin{itemize}
\item[$(i)$] Assume $\Hrun$. Then, if $\gamma \in (0,\frac{1}{4L}]$ and $\theta \in [0,\frac{1}{8\sigma^2}]$,
\begin{equation*}
   \sup_{n\ge0} \ES_x \left[ e^{ \theta \pot \left(\Xge_{n\gamma}\right)} \right]  \le e^{\theta\pot(x)} + e^{\theta\left(1+\frac{8dL}{\un{c}}\right)^{\frac{1}{1-r}} +\theta dL}{\cosh}(\theta\gamma d L).
\end{equation*}
\item[$(ii)$] Assume $\Hrun$ and $\Hrdeux$. If $\gamma \in (0,\frac{1-r}{4\bar{c}\vee L}]$ and $\theta \in [0,\frac{1}{8\sigma^2}\wedge 1]$, then for all $x \in \ER^d$ and $ n \in \mathbb{N}$ we have 
\begin{equation*}
    \sup_{n\ge0} \ES_x \left[ e^{ \theta \pot \left(\Xge_{n\gamma}\right)} \right]  \le e^{\theta\pot(x)} + c e^{c_r\frac{\theta d(1+\sigma^2)\bar{c}}{\un{c}}},
\end{equation*}
where $c$ denotes a constant independent of the parameters and $c_r$ a constant which only depends on $r$.
\end{itemize}
\end{prop}
The proof of this proposition is postponed in \cref{art2:appen:appendixA}.
\begin{rem} The reader will find more explicit (but more technical) bounds in the proof of the second case. It is worth noting that we can preserve a condition on $\gamma$ does not depend on $d$ (as in the strongly convex setting). This is of first importance in our multilevel setting where it is much more efficient if the rough layers of the method can be implemented with step sequences with large sizes. The proof is very close to  \cite{art:gadat2020cost}  but the bounds are refined. In particular, compared to this paper, we precisely do not require that the step size decreases with $d$.
\end{rem}
Thanks to the two previous results we are now able to control the moment of the continuous time and the discrete time processes.
\begin{prop}\label{art2:prop:moment}
Assume $\Hrun$.
\newline (i)  For all $x \in \ER^d$, $p\ge 0 $,
\begin{equation*}
   \sup_{t\ge0} \ES_x \left[\pot^p\left( \X_t\right)\right] \lesssim_{p}  \left(U(x)+ \Psiun\right)^p.
\end{equation*}
where
\begin{equation*}
    \Psiun=\begin{cases}
    (1+\sigma^2)\left(1+\frac{dL}{\un{c}}\right)^{\frac{1}{1-r}}&\textnormal{under $\Hrun$ only}\\
    4^{\frac{1}{1-r}}\vee \frac{(1+\sigma^2)  (\bar{c}\vee L)}{\un{c}}d& \textnormal{under $\Hrun$ and $\Hrdeux$},
    \end{cases}
\end{equation*}
\newline (ii) Let $\pas \in (0,\gamzero]$ with $\gamzero$ defined by \eqref{art2:def:gamzero}. Then,
\begin{equation*}
    \sup_{n\ge0} \ES_x \left[\pot^p\left( \Xge_{n\pas}\right)\right] \lesssim_{p}  (U(x)+ \Psideux)^p
\end{equation*}
where $\Psideux$ is defined by \eqref{art2:def:gamzero}
\newline (iii) In particular, 
\begin{equation}\label{art2:eq:boundfacileatransporter}
    \max\left(\sup_{t\ge0} \ES_x \left[\pot^p\left( \X_t\right)\right],\sup_{t\ge0,\gamma\in(0,\gamzero]}\ES_x \left[\pot^p\left( \Xge_{\un{t}}\right)\right]\right)\lesssim_{r,p} (U(x)+ \Psideux)^p.
\end{equation}
and 
\begin{equation}\label{art2:eq:boundfacileinvariant}
    \max\left(\int U^p(x)\pi(dx),\sup_{\gamma\in(0,\gamzero]}\int U^p(x)\pi^\gamma(dx)\right)\lesssim_{r,p}  \Psideux^p.
\end{equation}
\end{prop}
\begin{proof}
The proof similar to \cite[Prop B.4]{art:gadat2020cost} is postponed in \cref{art2:appen:appendixB}.
\end{proof}
\begin{rem} In order to avoid the distinction between cases in all the proofs, we choose to adopt only one notation for $\Psiun$ and $\Psideux$ but the reader has to keep in mind that the definition of these quantities depends on the fact that $\Hrdeux$ is satisfied or not.  Let us also recall that 
the notation $\lesssim_{r,p}$ means that the constant only depends on $r$ and $p$. These constants are certainly locally bounded: for any compact subset  $K$ of $[0,1)\times [1,+\infty)$, there exists a universal constant $c$ such that for any $(r,p)\in K$, the underlying constant $c_{r,p}$  related to $\lesssim_{r,p}$ is bounded by $c$. Finally, note that we chose to keep all the dependencies in the other parameters. 
\end{rem}
\subsection{Longtime strong discretization error under \texorpdfstring{$\Hrun $}{Hr}:}
The following proposition studies the $L^2$-error induced by the discretization of the SDE under $\Hrun $. The notations $\gamzero$ and $\Psideux$ come from Proposition
\ref{art2:prop:moment}.
\begin{prop}\label{art2:prop:HypH2} Assume $\Hrun$ and let $x\in \ER^d$,  $\pas \in (0,\gamzero]$ and $\delta \in (0,1)$. Then, 
\begin{equation*}
    \sup_{t\ge0}\ES_x\left[|\X_t - \Xge_t|^2\right]\lesssim_{r} 
     \frac{L\gamma}{\un{c}\delta} \left(\pot(x)+\Psideux\right)^{1+2r}  \left(\Gamma \left( \frac{1}{\delta}\right) +  \frac{(U(x)+\Psideux)^{r+\frac{2\delta r}{1-\delta}}}{\un{c}^{1+\frac{2\delta}{1-\delta}}}  \right)
    \end{equation*}
where $\Gamma$ is the gamma function.
\end{prop}
\begin{rem}The control of the $L^2$-distance between the diffusion and its Euler scheme is a fundamental property for the efficiency of the multilevel method. Actually, it allows us to control the variance of each level. The fact that we are able to obtain such a property in this (semi)-weakly convex setting is new.
\end{rem}
We start with two technical lemmas.
\begin{lem}\label{art2:lem:momenttoU}
    Assume $\Hrun$ then for all $x\in \ER^d$ we have
    \begin{equation*}
        |x-x^\star|^2 \le \frac{\pot^{1+r}(x)- \pot^{1+r}(x^\star)}{\un{c} (1+r)}.
    \end{equation*}
\end{lem}
\begin{proof}
First, one can check that for all $x\in \ER^d$ and for all eigenvalue of the Hessian we have
\begin{equation}\label{art2:eq:eigenvalue}
    \lambda_{D^2\pot^{1+r}(x)} \ge (1+r)\pot^r(x) \un{\lambda}_{D^2\pot(x)} \ge (1+r)\un{c},
\end{equation}
where in the last inequality we have used assumption $\Hrun$. By the Taylor formula, 
\begin{equation*}
    \pot^{1+r}(x) -\pot^{1+r}(x^\star) = \langle \nabla\pot^{1+r} (x^\star),x-x^\star \rangle + \int_0^1 \langle D^2 \pot^{1+r}(\xi_{\theta})(x-x^\star),x-x^\star \rangle \mathrm{d}\theta,
\end{equation*}
where $\xi_\theta = \lambda x + (1-\lambda)x^{*} $. By \eqref{art2:eq:eigenvalue} and the fact that $\nabla\pot^{1+r} (x^\star)=0$, we get
\begin{equation*}
    \pot^{1+r}(x) -\pot^{1+r}(x^\star) \ge (1+r)\un{c} \int_0^1 |x-x^\star|^2 \mathrm{d}\theta.
\end{equation*}
This concludes the proof.
\end{proof}
The next lemma is a bound on the moments of the increment of the Euler scheme (with the notations $\gamzero$ and $\Psideux$ introduced in Proposition \ref{art2:prop:moment}).
\begin{lem}\label{art2:lem:EulIncr}Assume $\Hrun$ with $r\in[0,1)$. Let $\pas \in (0,\gamzero]$. Then for all $t>0$ and $k\in \mathbb{N}^*$, 
\begin{equation*}
    \ES_x \left[\left| \Xge_t - \Xge_{\un{t}}\right|^{2k} \right]^{1/k}  \lesssim_{r,k} \blip^2 (t-\un{t})^2\left(\pot(x)+\Psideux\right)^{1+r} +(t-\un{t})d\sigma^2.
\end{equation*}
\end{lem}
\begin{proof}
By the definition of the Euler scheme we have
\begin{equation*}
    \begin{split}
        \ES_x \left[\left| \Xge_t - \Xge_{\un{t}}\right|^{2k} \right]^{1/k} &= \ES_x \left[\left(\left|- (t-\un{t})\nabla\pot \left(\Xge_{\un{t}}\right)+\sigma(\B_t-\B_{\un{t}})\right|^{2}\right)^k \right]^{1/k} 
        \\ & \le \ES_x \left[\left(2 (t-\un{t})^2 \left|\nabla\pot \left(\Xge_{\un{t}}\right)\right|^{2}+2\sigma^2\left|\B_t-\B_{\un{t}}\right|^{2}\right)^k \right]^{1/k}.
    \end{split}
\end{equation*}
Then, by Minkowski inequality,
\begin{equation*}
    \begin{split}
        \ES_x \left[\left| \Xge_t - \Xge_{\un{t}}\right|^{2k} \right]^{1/k} & \le 2 (t-\un{t})^2 \ES_x \left[\left|\nabla\pot \left(\Xge_{\un{t}}\right)\right|^{2k}\right]^{1/k}+2\sigma^2\ES_x \left[\left|\B_t-\B_{\un{t}}\right|^{2k} \right]^{1/k}
        \\ & \le 2 \blip^2 (t-\un{t})^2 \ES_x \left[\left|\Xge_{\un{t}}-x^\star\right|^{2k}\right]^{1/k}+2  (t-\un{t})\sigma^2\ES_x \left[|Z|^{2k} \right]^{1/k},
    \end{split}
\end{equation*}
where in the last line we used the $\blip$-Lipschitz continuous property of $\nabla \pot$ and the fact that $\B_t-\B_{\un{t}} \sim \sqrt{t-\un{t}}Z$ with $Z \sim \mathcal{N}(0,\mathrm{Id})$. Finally, by Lemma \ref{art2:lem:momenttoU} and Proposition \ref{art2:prop:moment}, we obtain 
\begin{align*}
    \ES_x \left[\left| \Xge_t - \Xge_{\un{t}}\right|^{2k} \right]^{1/k}& \le \frac{2\blip^2 (t-\un{t})^2}{1+r}  \sup_{n\ge0} \ES[\left(\pot^{1+r}(\bar{X}_{n\gamma})\right)^{2k}]^{\frac{1}{2k}} 
    + 2 d \sigma^2 \left(\frac{(2k)!}{2^k k!}\right)^{1/k} (t-\un{t})\\
    & \lesssim_{r,k} \blip^2 (t-\un{t})^2\left(\pot(x)+\Psideux\right)^{1+r} +(t-\un{t})d\sigma^2.
\end{align*}
\end{proof}
\noindent We are now ready to prove Proposition \ref{art2:prop:HypH2}.
\begin{proof}\textit{(of Proposition \ref{art2:prop:HypH2})}
Let $x \in \ER^d$ and consider the following process in $\ER^d\times\ER^d$,
\begin{equation*}
    \begin{cases}
        &\X_t=x-\int_0^t \nabla\pot(\X_s) \mathrm{d}s+\sigma \B_t\\
        &\Xge_t=\Xge_{\un{t}}-(t-\un{t}) \nabla\pot(\Xge_{\un{t}})+\sigma \left(\B_t-\B_{\un{t}}\right) ,
    \end{cases}
\end{equation*}
with $\Xge_0=x.$ Denote by $F_x(t):=|\X_t - \Xge_t|^2$. By the Lebesgue differentiability theorem,
\begin{equation*}
    \begin{split}
        F'_x(t)&=2\left\langle \X_t-\Xge_t, \nabla \pot(\Xge_{\un{t}})-\nabla\pot(\X_t)\right\rangle \\
        &=2\underbrace{ \left\langle \X_t-\Xge_t, \nabla\pot(\Xge_t)-\nabla \pot(\X_t)\right\rangle}_{E_1}+2\underbrace{ \left\langle \X_t-\Xge_t, \nabla\pot(\Xge_{\un{t}})-\nabla\pot(\Xge_t)\right\rangle}_{E_2}.
    \end{split}
\end{equation*}
To control  $E_1$ we use a Taylor expansion and obtain, 
\begin{equation*}
    \begin{split}
        E_1 &=  -2\left\langle \X_t-\Xge_t, \int_0^1 D^2\pot\left(\lambda \X_t + (1-\lambda)\Xge_t\right)(\X_t-\Xge_t) \mathrm{d}\lambda\right\rangle \\
        & = -2\int_0^1 \left\langle \X_t-\Xge_t,  D^2\pot\left(\lambda \X_t + (1-\lambda)\Xge_t\right)(\X_t-\Xge_t)\right\rangle \mathrm{d}\lambda.
    \end{split}
\end{equation*}
Since $D^2\pot$ is a symmetric matrix,
\begin{equation*}
    E_1 \le  -2 \xi_t \left| \X_t-\Xge_t \right|^2,
\end{equation*}
where $\xi_t = \int_0^1 \un{\lambda}_{D^2\pot\left(\lambda \X_t + (1-\lambda)\Xge_t\right)}  \mathrm{d}\lambda$. For the second term $E_2$, using the inequality $\langle a,b \rangle \le \frac{\xi_t}{2}|a|^2+\frac{1}{2 \xi_t}|b|^2$ and the fact that $\nabla \pot$ is $\blip$-Lipschitz, we get
\begin{equation*}
    E_2 \le  \xi_t \left| \X_t-\Xge_t \right|^2  + \frac{\blip^2}{\xi_t} \left| \Xge_t -\Xge_{\un{t}} \right|^2.
\end{equation*}
Thus,
\begin{equation*}
    F_x'(t) \le - \xi_t F_x(t) +  \frac{\blip^2}{\xi_t} \left| \Xge_t -\Xge_{\un{t}} \right|^2,
\end{equation*}
and a Gronwall argument leads to
\begin{equation*}
    F_x(t) \le \int_0^t \frac{\blip^2}{\xi_s} \left| \Xge_s -\Xge_{\un{s}} \right|^2 e^{-\int_s^t \xi_u \mathrm{d}u}\mathrm{d}s .
\end{equation*}
Taking the expectation, the Fubini's theorem implies
\begin{equation*}
    \ES_x\left[|\X_t - \Xge_t|^2\right] \le \int_0^t \underbrace{\ES_x \left[  \frac{1}{\xi_s}  \left| \Xge_s -\Xge_{\un{s}} \right|^2 e^{-\int_s^t \xi_u \mathrm{d}u}\right]}_{A_s} \mathrm{d}s.
\end{equation*}
Now let $\phi$ be a real non negative function,  $\phi: t \mapsto \phi(t)$ 
\begin{equation} \label{art2:eq:integ}
    \begin{split}
        A_s & \le  \ES_x   \left[ \frac{1}{\xi_s}  \left| \Xge_s -\Xge_{\un{s}} \right|^2e^{- \int_s^t \xi_u\mathrm{d}u} 1_{\{ \int_s^t \xi_u \mathrm{d}u > \phi(s) \}}\right] 
        \\ &+  \ES_x   \left[ \frac{1}{\xi_s}  \left| \Xge_s -\Xge_{\un{s}} \right|^2e^{- \int_s^t \xi_u\mathrm{d}u} 1_{\{  \int_s^t \xi_u \mathrm{d}u \le \phi(s) \}}\right]
        \\ & \le \underbrace{ \ES_x   \left[ \frac{1}{\xi_s}  \left| \Xge_s -\Xge_{\un{s}} \right|^2 \right]e^{- \phi(s)}}_{A_s^{(1)}}+  \underbrace{  \ES_x   \left[\frac{1}{\xi_s}  \left| \Xge_s -\Xge_{\un{s}} \right|^2 1_{\{  \int_s^t \xi_u \mathrm{d}u \le \phi(s) \}}\right]}_{A_s^{(2)}},
    \end{split}
\end{equation}
Using $\Hrun$, the convexity of $x \mapsto \pot(x)$ and $t \mapsto t^{-r}$ and the Jensen inequality, we have
\begin{equation*}
    \begin{split}
        \xi_t=\int_0^1 \un{\lambda}_{D^2\pot\left(\lambda \X_t + (1-\lambda)\Xge_t\right)}  \mathrm{d}\lambda & \ge \un{c}\int_0^1 \pot^{-r} \left(\lambda \X_t + (1-\lambda)\Xge_t\right) \mathrm{d}\lambda
        \\ & \ge \un{c} \int_0^1 \left(  \lambda \pot \left( \X_t\right)  +(1-\lambda) \pot \left(\Xge_t\right)\right)^{-r}  \mathrm{d}\lambda 
        \\ & \ge \un{c}\left(   \int_0^1 \lambda \pot \left( \X_t\right)  +(1-\lambda)  \pot \left(\Xge_t\right) \mathrm{d}\lambda \right)^{-r}
        \\ & \ge 2^{r}\un{c}\left( \pot \left( \X_t\right) +\pot \left(\Xge_t\right) \right)^{-r}.
    \end{split}
\end{equation*}
Thus by the  inequality $(a+b)^{p} \le a^p + b^p$ for $a,b \ge 0$ and $p \in [0,1]$,
\begin{equation}\label{art2:ineq:invxi}
    \begin{split}
        \xi_t^{-1}  \le \frac{1}{2^{r}\un{c}}\left(\pot \left(\X_t\right) + \pot\left(\Xge_t\right) \right)^{ r}   \le \frac{1}{2^{r}\un{c}}\left(\pot^r\left(\X_t\right) + \pot^r \left(\Xge_t\right)\right).
           \end{split}
\end{equation}
By Cauchy-Schwarz inequality, Proposition \ref{art2:prop:moment}\emph{(iii)} and Lemma \ref{art2:lem:EulIncr} we get
\begin{align}
          A_s^{(1)} & \le \ES[\xi_s^{-2}]^{\frac{1}{2}} \ES[|\bar{X}_s-\bar{X}_{\un{s}}|^4]^{\frac{1}{2}}  e^{- \phi(s)}\nonumber    \\
          & \lesssim_{r} e^{- \phi(s)} \un{c}^{-1}\left(\pot(x)+\Psideux\right)^{r} \left[ \blip^2 (s-\un{s})^2\left(\pot(x)+\Psideux\right)^{1+r} +(s-\un{s})d\sigma^2\right]\nonumber\\
          &\lesssim_{r} e^{- \phi(s)} \frac{L\gamma}{\un{c}} \left(\pot(x)+\Psideux\right)^{1+2r},\label{art2:eq:aunes}
\end{align}
where in the last line, we used that $d\sigma^2\le \Psideux$ and $\gamma L\le 1$.

For the second term use Cauchy-Schwarz inequality,
\begin{equation*}
    \begin{split}
        A_s^{(2)} & \le \ES_x   \left[\left(\frac{1}{\xi_s}  \left| \Xge_s -\Xge_{\un{s}} \right|^2 \right)^2\right]^{1/2} \Pr_x  \left( \int_s^t \xi_u \mathrm{d}u \le \phi(s) \right)^{1/2}
        \\ & \le   \ES_x   \left[\frac{1}{\xi_s^4} \right]^{1/4}\ES_x   \left[\left| \Xge_s -\Xge_{\un{s}} \right|^8\right]^{1/4} \Pr_x  \left( \int_s^t \xi_u \mathrm{d}u \le \phi(s) \right)^{1/2}.
    \end{split}
\end{equation*}
With the help of Inequality \eqref{art2:ineq:invxi} and Proposition \ref{art2:prop:moment}, we have
\begin{equation*}
    A_s^{(2)} \lesssim_{r}    \un{c}^{-1}\left(\pot(x)+\Psideux\right)^{r} \left[ \blip^2 (s-\un{s})^2\left(\pot(x)+\Psideux\right)^{1+r} +(s-\un{s})d\sigma^2\right]  \Pr_x  \left( \int_s^t \xi_u \mathrm{d}u \le \phi(s) \right)^{1/2}.
\end{equation*}
For the third term of this product, let $\kappa$ be a positive number and use Markov inequality:
\begin{equation*}
    \begin{split}
        \Pr_x  \left( \int_s^t \xi_u \mathrm{d}u  \le \phi(s) \right)^{1/2} & \le \Pr_x \left(  \left( \int_s^t \xi_u  \mathrm{d}u\right)^{-\kappa} \ge \phi^{\kappa}(s)\right)^{1/2}
        \\ & \le \phi^{\kappa/2}(s) \ES  \left[  \left( \int_s^t \xi_u \mathrm{d}u\right)^{-\kappa} \right]^{1/2}.
    \end{split}
\end{equation*}
The function $x \mapsto x^{-\kappa}$ being convex on $(0,+\infty)$ it follow from Jensen inequality that
\begin{equation*}
    \Pr_x  \left( \int_s^t \xi_u \mathrm{d}u  \le \phi(s) \right)^{1/2} \le  \phi^{\kappa/2}(s) (t-s)^{-\kappa/2} \ES  \left[  \frac{1}{t-s} \int_s^t \xi_u^{-\kappa} \mathrm{d}u \right]^{1/2}.
\end{equation*}
Using again inequality \eqref{art2:ineq:invxi} and  Proposition \ref{art2:prop:moment},
\begin{equation*}
    \begin{split}
        \Pr_x  \left( \int_s^t \xi_u \mathrm{d}u  \le \phi(s) \right)^{1/2} & \lesssim_r \un{c}^{-\frac{\kappa}{2}}\left(\frac{ \phi(s)}{t-s} \right)^{\kappa/2} \sup_{u \in [s,t]} \ES  \left[  \pot^{ \kappa r}\left( \X_t\right) +  \pot^{ \kappa r}\left(\Xge_t\right)\right]^{1/2}
        \\ & \lesssim_r  \un{c}^{-\frac{\kappa}{2}}\left(\frac{ \phi(s)}{t-s} \right)^{\kappa/2} \left(\pot(x)+\Psideux\right)^{\frac{\kappa r}{2}} .
    \end{split}
\end{equation*}
Finally, we get
\begin{equation*}
    A_s^{(2)} \lesssim_r \un{c}^{-1-\frac{\kappa}{2}}\left(\pot(x)+\Psideux\right)^{r+\frac{\kappa r}{2}} \left[ \blip^2 (s-\un{s})^2  \left(\pot(x)+\Psideux\right)^{1+r} +(s-\un{s})d\sigma^2\right] \left(\frac{ \phi(s)}{t-s} \right)^{\kappa/2} .
\end{equation*}
Now let $\phi(s)=\frac{t-s}{(t+1-s)^{1-\delta}}$ with $a>0$, $\delta \in (0,1)$ and $\kappa=\frac{2(1+\delta)}{1-\delta}$ we have 
\begin{equation*}
    A_s^{(2)}  \lesssim_{r} \un{c}^{-2-\frac{2\delta}{1-\delta}}\left[\blip^2 (s-\un{s})^2 (U(x)+\Psideux)^{1+3r+\frac{2\delta r}{1-\delta}}+(s-\un{s}){d\sigma^2}(U(x)+\Psideux)^{2r+\frac{2\delta r}{1-\delta}}\right] \left(\frac{1}{t+1-s} \right)^{1+\delta}.
\end{equation*}
As a consequence, since $\gamma L\le 1$ and $d\sigma^2 \le {\Psideux}$, we obtain
\begin{equation*}
    A_s^{(2)} \lesssim_{r} \frac{\blip \gamma}{\un{c}^{2+\frac{2\delta}{1-\delta}}} (U(x)+\Psideux)^{1+3r+\frac{2\delta r}{1-\delta}}\left(\frac{1}{t+1-s} \right)^{1+\delta}.
\end{equation*}
Back to \eqref{art2:eq:integ}, we deduce from \eqref{art2:eq:aunes} and from  the above inequality that
\begin{align*}
   \ES_x\Big[|\X_t &- \Xge_t|^2\Big] \le \int_0^t  A_s^{(1)}   \mathrm{d}s + \int_0^t  A_s^{(2)} \mathrm{d}s \\
    & \lesssim_{r} \frac{L\gamma}{\un{c}} \left(\pot(x)+\Psideux\right)^{1+2r}  \left(\int_0^t   e^{-\frac{ u}{(u+1)^{1-\delta}}}  \mathrm{d}u + 
    \frac{(U(x)+\Psideux)^{r+\frac{2\delta r}{1-\delta}}}{\un{c}^{1+\frac{2\delta}{1-\delta}}}\int_0^t \left( \frac{1}{1+u}\right)^{1+ \delta} \mathrm{d}u \right) \\
    & \lesssim_{r}  \frac{L\gamma}{\un{c}} \left(\pot(x)+\Psideux\right)^{1+2r}
    \left( \int_0^t   e^{-u^{\delta}}  \mathrm{d}u +  \frac{(U(x)+\Psideux)^{r+\frac{2\delta r}{1-\delta}}}{\delta \un{c}^{1+\frac{2\delta}{1-\delta}}}\right) \\
    & \lesssim_{r} \frac{L\gamma}{\un{c}\delta} \left(\pot(x)+\Psideux\right)^{1+2r}  \left( \Gamma \left( \frac{1}{\delta}\right) +  \frac{(U(x)+\Psideux)^{r+\frac{2\delta r}{1-\delta}}}{\un{c}^{1+\frac{2\delta}{1-\delta}}} \right)
\end{align*}
The result follows.
\end{proof}
\subsection{Convergence to equilibrium for the Euler Scheme under \texorpdfstring{$\Hrun $}{Hr}:}
We now proceed to establish the weak error between the discrete semi group and its invariant measure denoted by $\pi^\pas$. The proof of this result is based on the control of the so-called tangent process $\Tge_t^x:= \nabla \X^x_t$.
\begin{prop} \label{art2:prop:HypH1} Assume $\Hrun$ and let $x \in \ER^d$,  $\pas \in (0,\gamzero]$. Let $\kappa>0$ and $\phi:\ER_+\rightarrow\ER$ be a positive function. Assume that $U(x)\lesssim_r\Psideux$. Then, for any $r\in[0,1)$, there exists a constant $c_{r,\kappa}$ (depending only on $r$ and $\kappa$) such that for all $n>0$, for all Lipschitz continuous function $f:\ER^d\rightarrow\ER$,
\begin{equation*}
    \left| \ES_x \left[ f(\Xge_{n \pas}) \right] - \pi^{\pas}(f) \right| \le c_{r,\kappa} [f]_1 h_{\phi,\kappa}(n),
\end{equation*}
where
\begin{equation*}
    h_{\phi,\kappa}(n) = \un{c}^{-\frac{1}{2}}\Psideux^{\frac{1+r}{2}} e^{- \phi(n)}+\un{c}^{-\frac{\kappa+1}{2}}\Psideux^{\frac{1+r(1+\kappa)}{2}}\left( \frac{\phi(n)}{n\pas} \right)^{\frac{\kappa}{2}},
\end{equation*}
with $\gamzero$ and $\Psideux$ defined in Proposition \ref{art2:prop:moment}.
\end{prop}
\begin{rem} \label{art2:rem:summablefunction}
$\rhd$ In order to alleviate the purpose, the result is stated under the assumption that the initial condition $x$ satisfies $U(x)\lesssim_r\Psideux$ but the reader will find some bounds without this assumption in the proofs.
\bigskip \newline \noindent $\rhd$  The function $h_{\phi,\kappa}$ plays the role of convergence rate to equilibrium. In this setting where the Hessian is not lower-bounded, we adopt a strategy which consists in separating the space into two parts. In the first one, we assume that we have some good contraction properties parametrized by the function $\phi$ and in the other one, we try simply to control the probability that such a good contraction does not occur. This leads to a balance between two terms depending on $\phi$ and $\kappa$. In the following, we will choose $\phi$ and $\kappa$ in order that $h_{\phi,\kappa}$ is summable with the smallest impact on the dependence in the dimension.
\bigskip \newline \noindent Note that in \cite{cattiaux:hal-02486264}, some exponential rates are exhibited under similar assumptions in the continuous case (with the help of concentration inequalities). However, this exponential rate depends on some constants whose  control seems to be difficult to obtain (typically, when the starting distribution is absolutely continuous with respect to the invariant distribution, the constants involve the $L^2$-moment of the related density). Probably, some ideas could be adapted to the Euler scheme (starting from a deterministic point) but with technicalities that seem to carry us too far for this paper. 
\end{rem}
We preface the proof of Proposition \ref{art2:prop:HypH1} by a lemma about the shape of the \textit{first variation process} of the continuous time Euler scheme, $\Tge_t^x= \nabla_x \Xge^x_t$.
\begin{lem}\label{art2:lem:TDP}
For all $n \in \mathbb{N}$,  $x\in \ER^d$ and $\pas\in [0,1)$,
\begin{equation*}
    \Tge_{(n+1)\pas}^x = \prod_{i=0}^n \left(\mathrm{Id}_{\ER^d}- \pas D^2 \pot \left( \Xge_{i\pas}^x \right)\right).
\end{equation*}
\end{lem}
\begin{proof}\textit{(of Lemma \ref{art2:lem:TDP})}
First, observe that for all $n \in \mathbb{N}$,
\begin{equation*}
    \frac{\Tge_{(n+1)\pas}^x-\Tge_{n\pas}^x}{\pas} = \nabla_x \left(\frac{\Xge_{(n+1)\pas}^x-\Xge_{n\pas}^x}{\pas} \right),
\end{equation*}
and by the definition of the Euler scheme and the chain rule,
\begin{equation*}
    \begin{split}
        \frac{\Tge_{(n+1)\pas}^x-\Tge_{n\pas}^x}{\pas} &= \nabla_x \left(-\nabla_x \pot \left(\Xge_{n\pas}^x\right) + \sigma \left( \B_{(n+1)\pas}-  \B_{(n-1)\pas} \right) \right) \\
        & = - D^2 \pot \left(\Xge_{n\pas}^x\right) \Tge_{n\pas}^x.
    \end{split}
\end{equation*}
Then we get,
\begin{equation*}
    \Tge_{(n+1)\pas}^x=\Tge_{n\pas}^x \left(\mathrm{Id}_{\ER^d} - \pas D^2 \pot \left(\Xge_{n\pas}^x\right)  \right),
\end{equation*}
and the proof follows by a simple induction.
\end{proof}
Consider two paths defined with the same Brownian motion and different starting points: $x,y \in \ER^d$. The following proposition shows that there is a pathwise confluence, \textit{i.e.} the two trajectories get closer when $n$ goes to infinity. 
\begin{prop}\label{art2:prop:discretconfl}
Assume $\Hrun$ and let $x,y \in \ER^d$, $\pas \in (0,\gamzero]$, $\kappa>0$. Let $\phi:\ER_+\rightarrow\ER$ be a positive function. Then,
\begin{equation*}
   \sup_{n\ge0} \ES_{(x,y)} \left[ \left|\Xge_{n\pas}^x-\Xge_{n\pas}^y \right|^2 \right] \lesssim_{r,\kappa} \left| x-y \right|^2 h_{\phi,\kappa,x,y}(n),
\end{equation*}
where,
\begin{equation*}
    h_{\phi,\kappa,x,y}(n) = e^{-2\phi(n)} + \left( \frac{\phi(n)}{n\pas} \right)^{\kappa}\un{c}^{-\kappa} (U(x)+U(y)+\Psideux)^{\kappa r},\end{equation*}
with $\gamzero$ and $\Psideux$ given by Proposition \ref{art2:prop:moment}.
\end{prop}
\begin{rem} \label{art2:rem:exponenrates} In the sequel, this property is typically applied with a polynomial function $\phi$ which leads to polynomial rates to equilibrium. It is worth noting that the proof could be adapted to provide exponential rates (the idea would be to consider an exponentially decreasing convex function instead of $x\mapsto x^{-\kappa}$ in the proof below). However, with our method, such rates would lead to exponential dependence in the dimension. This is why we do not give such bounds here. 
\end{rem}
\begin{proof}\textit{(of Proposition \ref{art2:prop:discretconfl})}
For $x,y \in \ER^d$ and $n\in \mathbb{N}$, let us start by a Taylor expansion of the function $x \mapsto \Xge_{n\pas}^x $,
\begin{equation*}
    \begin{split}
        \left|\Xge_{n\pas}^x-\Xge_{n\pas}^y \right|^2 & = \left| \int_0^1 \Tge_{n\pas}^{\lambda x + (1-\lambda)y}(x-y) \mathrm{d}\lambda \right|^2
        \\ & =  \left| x-y \right|^2 \left| \int_0^1 \Tge_{n\pas}^{\lambda x + (1-\lambda)y} \mathrm{d}\lambda  \frac{x-y}{\left| x-y \right|}  \right|^2
        \\ & \le \left| x-y \right|^2  \left\| \int_0^1 \Tge_{n\pas}^{\lambda x + (1-\lambda)y} \mathrm{d}\lambda \right\|^2,
    \end{split}
\end{equation*}
where $\|.\|$ is the operator norm associated with the Euclidean norm. By Jensen inequality and Lemma \ref{art2:lem:TDP},
\begin{equation*}
    \begin{split}
        \left|\Xge_{n\pas}^x-\Xge_{n\pas}^y \right|^2 & \le \left| x-y \right|^2 \int_0^1  \left\|\prod_{i=0}^{n-1}  \mathrm{Id}_{\ER^d}- \pas D^2 \pot \left( \Xge_{i\pas}^{\lambda x + (1-\lambda)y} \right)  \right\|^2 \mathrm{d}\lambda
        \\ & \le \left| x-y \right|^2 \int_0^1  \prod_{i=0}^{n-1}  \left\| \mathrm{Id}_{\ER^d}- \pas D^2 \pot \left( \Xge_{i\pas}^{\lambda x + (1-\lambda)y} \right) \right\|^2 \mathrm{d}\lambda.
    \end{split}
\end{equation*}
The operator norm associated with the Euclidean norm of a symmetric matrix is equal to its spectral radius, so we get
\begin{equation*}
    \begin{split}
        \left|\Xge_{n\pas}^x-\Xge_{n\pas}^y \right|^2 & \le  \left| x-y \right|^2 \int_0^1  \prod_{i=0}^{n-1}\left( 1 - \pas \vpinf_{D^2 \pot\left(\Xge_{i\pas}^{\lambda x - (1-\lambda)y}\right)}\right)^2 \mathrm{d}\lambda 
        \\ & \le \left| x-y \right|^2 \int_0^1  e^{-2\pas \sum_{i=0}^{n-1} \vpinf_{D^2 \pot\left(\Xge_{i\pas}^{\lambda x - (1-\lambda)y}\right)} } \mathrm{d}\lambda,
    \end{split}
\end{equation*}
and 
\begin{equation} \label{art2:eq:confluencediscrete}
    \ES_{(x,y)} \left[ \left|\Xge_{n\pas}^x-\Xge_{n\pas}^y \right|^2 \right] \le \left| x-y \right|^2 \int_0^1 \ES_{(x,y)} \left[ e^{-2\xi_n (\lambda)} \right] \mathrm{d}\lambda,
\end{equation}
with,
\begin{equation*}
    \begin{split}
        \xi_n (\lambda)& := \pas \sum_{i=0}^{n-1} \vpinf_{D^2 \pot\left(\Xge_{i\pas}^{\lambda x - (1-\lambda)y}\right)} = \int_0^{n\pas}\vpinf_{D^2 \pot\left(\Xge_{\un{u}}^{\lambda x - (1-\lambda)y}\right)} \mathrm{d}u.
    \end{split}
\end{equation*}
For a given real non negative function $\phi: \ER \to \ER $ we have  
\begin{equation*}
    \begin{split}
      \ES_{(x,y)}  \left[ e^{-2 \xi_n (\lambda)} \right] & \le  \ES_{(x,y)}  \left[ e^{-2  \xi_n} 1_{\{  \xi_{n} > \phi(n) \}}\right] +  \ES_{(x,y)}  \left[ e^{-2  \xi_n} 1_{\{  \xi_{n}  \le \phi(n) \}}\right] 
         \\ & \le  e^{-2 \phi(n)}+  \ES_{(x,y)}  \left[  1_{\{  \xi_{n}  \le \phi(n) \}}\right]
        \\ & \le  e^{-2 \phi(n)}+  \Pr_{(x,y)}  \left( \int_0^{n\pas}\vpinf_{D^2 \pot\left(\Xge_{\un{u}}^{\lambda x - (1-\lambda)y}\right)} \mathrm{d}u \le \phi(n) \right).
    \end{split}
\end{equation*}
For a positive number $\kappa$ we have
\begin{equation*}
    \ES_{(x,y)}  \left[ e^{-2 \xi_n (\lambda)} \right] \le e^{-2 \phi(n)} + \Pr_{(x,y)} \left(  \left( \int_0^{n\pas}\vpinf_{D^2 \pot\left(\Xge_{\un{u}}^{\lambda x - (1-\lambda)y}\right)} \mathrm{d}u\right)^{-\kappa} \ge \phi^{-\kappa}(n) \right),
\end{equation*}
and using the Markov inequality,
\begin{equation*}
    \ES_{(x,y)}  \left[ e^{-2 \xi_n (\lambda)} \right]
    \le e^{-2 \phi(n)} +  \phi^{\kappa}(n) \ES_{(x,y)}  \left[  \left( \int_0^{n\pas}\vpinf_{D^2 \pot\left(\Xge_{\un{u}}^{\lambda x - (1-\lambda)y}\right)} \mathrm{d}u\right)^{-\kappa} \right].
\end{equation*}
The function $x \mapsto x^{-\kappa}$ is convex on $(0,+\infty)$ then by Jensen inequality,
\begin{equation*}
    \begin{split}
        \ES_{(x,y)}  \left[ e^{-2 \xi_n (\lambda)} \right] &\le  e^{-2 \phi(n)} + \phi^{\kappa}(n) (n\pas)^{-\kappa} \ES_{(x,y)}  \left[  \frac{1}{n\pas} \int_0^{n\pas} \vpinf_{D^2 \pot\left(\Xge_{\un{u}}^{\lambda x - (1-\lambda)y}\right)}^{-\kappa} \mathrm{d}u \right]
        \\ & \le  e^{-2\phi(n)} + \left( \frac{\phi(n)}{ n\pas } \right)^{\kappa} \sup_{k \in \{0,\dots,n-1\}} \ES  \left[ \vpinf_{D^2 \pot\left(\Xge_{k\pas}^{\lambda x - (1-\lambda)y}\right)}^{-\kappa} \right].
    \end{split}
\end{equation*}
Observe that assumption $\Hrun$ implies,
\begin{equation*}
    \ES_{(x,y)}  \left[ e^{-2 \xi_n (\lambda)} \right] \le  e^{-2\phi(n)} + \left( \frac{\phi(n)}{n\pas} \right)^{\kappa} \sup_{k \in \{0,\dots,n-1\}} \ES  \left[\un{c}^{-\kappa}\pot^{\kappa r}\left(\Xge_{k \pas}^{\lambda x + (1-\lambda)y}\right)\right].
\end{equation*}
By Proposition \ref{art2:prop:moment} and the convexity of $\pot$, this implies that for any $\gamma\in(0,\gamzero]$, 
\begin{equation*} 
    \ES_{(x,y)}  \left[ e^{-2 \xi_n (\lambda)} \right]  \lesssim_{r,\kappa}  e^{-2\phi(n)} + \left( \frac{\phi(n)}{n\pas} \right)^{\kappa}\un{c}^{-\kappa} (U(x)+U(y)+\Psideux)^{\kappa r}.
\end{equation*}
\end{proof}
Thanks to this confluence property we are now able to prove the convergence to equilibrium of the Euler scheme and to give the rate of this convergence.
\begin{proof}\textit{(of Proposition \ref{art2:prop:HypH1})}
Since $\pi^\pas$ is invariant for $\left(\Xge_{n\pas}\right)_{n \in \mathbb{N}}$ we deduce from Fubini's Theorem and Jensen inequality that
\begin{equation*}
    \begin{split}
      \left| \ES_x \left[ f(\Xge_{n \pas}) \right] - \pi^{\pas}(f) \right|^2 &= \left| \int_{\ER^d} \ES_{(x,y)} \left[f(\Xge_{n \pas}^x) - f(\Xge_{n \pas}^y) \right] \pi^{\pas}(\mathrm{d}y) \right|^2
      \\ & \le \int_{\ER^d} \ES_{(x,y)} \left[  \left| f(\Xge_{n \pas}^x) - f(\Xge_{n \pas}^y)\right|^2 \right] \pi^{\pas}(\mathrm{d}y).
    \end{split}
\end{equation*}
The Lipschitz property of $f$ implies that 
\begin{equation*}
    \left| \ES_x \left[ f(\Xge_{n \pas}) \right] - \pi^{\pas}(f) \right|^2 \le[f]_1^2 \int_{\ER^d} \ES_{(x,y)} \left[  \left| \Xge_{n \pas}^x - \Xge_{n \pas}^y\right|^2 \right] \pi^{\pas}(\mathrm{d}y),
\end{equation*}
where $[f]_1$ is the Lipschitz constant of $f$. Proposition \ref{art2:prop:discretconfl} implies 
\begin{equation}\label{art2:eq:jlskjdl}
        \left| \ES_x \left[ f(\Xge_{n \pas}) \right] - \pi^{\pas}(f) \right|^2   \lesssim_{r,\kappa}  [f]_1^2 \int_{\ER^d}  \left| x-y \right|^2 h_{\phi,\kappa,x,y}(n)\pi^{\pas}(\mathrm{d}y).
\end{equation}
By Lemma \ref{art2:lem:momenttoU},
\begin{equation*}
    |x-y|^2\le 2 (|x-x^\star|^2+|y-y^\star|^2)\le \frac{2}{\un{c}(1+r)}\left(U^{1+r}(x)+U^{1+r}(y)\right).
\end{equation*}
With the help of the Young inequality, we also have
\begin{equation*}
    |x-y|^2(U(x)+U(y))^{r\kappa}\lesssim_{r,\kappa} \un{c}^{-1} \left(U^{1+r(1+\kappa)}(x)+U^{1+r(1+\kappa)}(y)\right).
\end{equation*}
Plugging these controls into \eqref{art2:eq:jlskjdl} yields
\begin{equation}\label{art2:eq:vitessexxx}
    \begin{split}
        \left| \ES_x \left[ f(\Xge_{n \pas}) \right] - \pi^{\pas}(f) \right|^2   & \lesssim_{r,\kappa} \frac{[f]_1^2}{\un{c}} \left(U^{1+r}(x)+\pi^\gamma(U^{1+r})\right)e^{-2 \phi(n)}\\
        &+ [f]_1^2\un{c}^{-\kappa-1}\left(  U^{1+r(1+\kappa)}(x)+\pi^\gamma(U^{1+r(1+\kappa)})\right)\left( \frac{\phi(n)}{n\pas} \right)^{\kappa} .\\
       &+ [f]_1^2\un{c}^{-\kappa-1} \Psideux^{r\kappa}\left(  U^{1+r}(x)+\pi^\gamma(U^{1+r})\right)\left( \frac{\phi(n)}{n\pas} \right)^{\kappa} .
       \end{split}
\end{equation}
To conclude, we now use the bound \eqref{art2:eq:boundfacileinvariant} of Proposition \ref{art2:prop:moment}\emph{(iii)} and the assumption $U(x)\lesssim_r \Psideux$.
\end{proof}
\subsection{Bias induced by the discretization under \texorpdfstring{$\Hrun $}{Hr}:}\label{art2:sec:weakerror}
We now need to provide estimates of ${\cal W}_1(\pi,\pi^\gamma)$. We provide two results: 
Lemma \ref{art2:lem:firstresult} where we directly derive from Proposition \ref{art2:prop:HypH2} a bound in $O(\sqrt{\gamma})$ which ``only'' requires the potential $U$ to be ${\cal C}^2$. However, such a bound has a serious impact on the dependency in $\varepsilon$ of the complexity. Thus, we propose a second result when $U$ is ${\cal C}^3$ where we recover a bound in $O({\gamma})$.
\subsubsection{A first bound in  \texorpdfstring{$O(\sqrt{\gamma})$}{O(sqrt(gamma))}}
As mentioned before, a first estimate can be directly deduced from Proposition \ref{art2:prop:HypH2}. Actually, since in this result, the $L^2$-error between the process and its discretization is controlled uniformly in time, this leads to a similar bound for $ {\cal W}_1(\pi,\pi^\gamma) $  by letting $t$ go to $\infty$. More precisely,
\begin{lem}\label{art2:lem:firstresult}
    Assume $\Hrun$. Let $\gamma\in(0,\gamzero]$. Then,
    \begin{equation}
        {\cal W}_1(\pi,\pi^\gamma)^2  \lesssim_{r,\delta}\frac{L\gamma}{\un{c}^{\frac{2}{1-\delta}}\wedge\un{c}}\Psideux^{1+3r+\frac{2\delta r}{1-\delta}}.
    \end{equation}
\end{lem}
\begin{proof}
Owing to the stationarity of  $\pi$, we have for every $n\ge0$,
\begin{equation*}
    {\cal W}_1(\pi,\pi^\gamma)\le  {\cal W}_1(\pi P_{n\gamma},\pi \bar{P}_{n\gamma})+{\cal W}_1(\pi \bar{P}_{n\gamma}, \pi^\gamma),
\end{equation*} 
so that 
\begin{equation*}
    {\cal W}_1(\pi,\pi^\gamma)\le \limsup_{n\rightarrow+\infty} {\cal W}_1(\pi P_{n\gamma},\pi \bar{P}_{n\gamma}),
\end{equation*} 
since ${\cal W}_1(\pi \bar{P}_{n\gamma}, \pi^\gamma)\xrightarrow{n\rightarrow+\infty}0$ by 
 Proposition \ref{art2:prop:HypH1} (more precisely, this property can be deduced from an integration of \eqref{art2:eq:vitessexxx} with respect to $\pi$ and from the fact that $\pi (U^p)<+\infty$ for any $p$ by Proposition \ref{art2:prop:moment}).
\newline Now, integrating with respect to $\pi$ the bound of Proposition \ref{art2:prop:HypH2} and using that $\pi(U^p)\lesssim_{r} \Psideux^p$ by Proposition \ref{art2:prop:moment} leads to the result.
\end{proof}
\subsubsection{A second bound in  \texorpdfstring{$O(\gamma)$}{O(gamma)}}
Even if the above bound is quite explicit in terms of its dependency with respect to $L$, $\un{c}$, $\bar{c}$ and $d$, the fact that it is in $O(\sqrt{\gamma})$ dramatically impacts the complexity in terms of $\varepsilon$ (at least).
\newline In fact, it is possible to get a $1$-Wasserstein  error of the order $\gamma$ by using a combination of the control of the rate of convergence to equilibrium of the continuous process and of the finite-time weak error (between the process and its discretization). Such a strategy is used in several papers : in \cite{art:PagesPanloup2022unajusted}, this idea is developed in a multiplicative setting with a so-called ``domino'' approach for the control of the $1$-Wasserstein and $TV$ distances between the process and its discretization, uniformly in time. For the control of ${\cal W}_1(\pi,\pi^\gamma)$ itself, our approach follows \cite{durmus2021asymptotic} which provides a series of bounds in many models and sets of assumptions which are mainly based on the following principle (see Lemma 1 of \cite{durmus2021asymptotic}). Taking advantage of the stationarity of $\pi^\gamma$, for any $p\ge1$, for any $t>0$,
\begin{equation*}
    {\cal W}_p(\pi,\pi^\gamma)\le {\cal W}_p(\pi,\pi^\gamma P_{\un{t}} )+{\cal W}_p(\pi^\gamma {P}_{\un{t}},\pi^\gamma \bar{P}_{\un{t}}),
\end{equation*} 
so that if we assume that 
\begin{equation*}
    {\cal W}_p(\pi,\pi^\gamma P_{t} )\le \varepsilon_1(t){\cal W}_p(\pi,\pi^\gamma)\quad\textnormal{and}\quad {\cal W}_p(\pi^\gamma {P}_{\un{t}},\pi^\gamma \bar{P}_{\un{t}})\le \varepsilon_2(\un{t}),
\end{equation*} 
then,
\begin{equation}\label{art2:eq:ebdur}
    {\cal W}_p(\pi,\pi^\gamma)\le \inf\left\{\frac{\varepsilon_2(\un{t})}{1-\varepsilon_1(\un{t})}, t>0\right\}.
\end{equation}
We thus propose to estimate $\varepsilon_1(t)$ and $\varepsilon_2(t)$ under Assumption $\Hrun$ (with or without $\Hrdeux$). This is the purpose of Lemmas \ref{art2:lem:varepsilonn} and 
\ref{art2:lem:psin} respectively. These two estimates lead to the following proposition
\begin{prop}\label{art2:prop:HypH3}
    Assume $\Hrun$ and let $\delta \in (0,1)$. Assume that $U$ is ${\cal C}^3$ with $\|\Delta (\nabla U)\|_{2,\infty}^2\lesssim_{r} \sigma^{-4} L^3 \Psideux$ (with $\|\Delta (\nabla U)\|_{2,\infty}$ defined in Lemma \ref{art2:lem:psin}). Then, a constant $c_{r,\delta}$ (depending only on $r$ and $\delta$) exists such that for all $\gamma \in (0,\gamzero]$,
    \begin{equation*}
        {\cal W}_2 (\pi,\pi^\gamma)\le c_{r,\delta} \left(  \un{c}^{-\frac{1}{1-\delta}} L^{\frac{3}{2}}\Psideux^{\frac{1}{2}+\frac{r}{1-\delta}}\right) \gamma.
    \end{equation*} 
\end{prop}
\begin{rem} $\rhd$ Note that this result is clearly in the spirit of \cite[Theorem 6]{durmus2021asymptotic}. However, there are several differences. First, we need here to adapt our proof to a setting where we only have polynomial convergence to equilibrium (instead of exponential convergence). Second, under our assumptions on $U$ which are more restrictive than the one of \cite[Theorem 6]{durmus2021asymptotic}, we can  improve the constants (and in particular avoid some exponential dependence in the Lipschitz constant $L$).\\

$\rhd$ Compared with Lemma \ref{art2:lem:firstresult} , this result improves the dependence in $\gamma$ but it is worth noting that the bound is also better with respect to $\Psideux$ (and thus to the dimension). 
\end{rem}
\begin{proof} With the notations of \cite[Theorem 6]{durmus2021asymptotic}, let $t_{\rm rel}=\inf\{t> 0, \varepsilon_1(\un{t})\le 1/\sqrt{2}\}$. Using Lemma \ref{art2:lem:varepsilonn},
one can upper-bound $t_{\rm rel}$ by asking  the two right-hand terms of \eqref{art2:eq:weppigamma} to be bounded by $1/4$. With $\phi(t)= t^\delta$ with $\delta \in(0,1)$, this leads to:
\begin{equation*}
    t_{\rm rel}\lesssim_{r} \max\{(\log 4)^\delta, (4{\un{c}^{-1}})^{\frac{1}{1-\delta}}\Psideux^{\frac{r}{1-\delta}}\}.
\end{equation*}
By \eqref{art2:eq:ebdur} and Lemma \ref{art2:lem:psin}\emph{(ii)}, we get for any $\lambda\in(0,1]$,
\begin{equation*}
    {\cal W}_2(\pi,\pi^\gamma)\le c_r \frac{\gamma}{\lambda} \sqrt{e^{\lambda t_{\rm rel}}L^3 \Psideux}.
\end{equation*}
Let $\lambda= t_{\rm rel}^{-1}$. In this case, we obtain
\begin{equation*}
    {\cal W}_2 (\pi,\pi^\gamma)\le c_r \gamma t_{\rm rel} \sqrt{L^3 \Psideux}\lesssim_{r,\delta} \gamma \un{c}^{-\frac{1}{1-\delta}} L^{\frac{3}{2}} \Psideux^{\frac{1}{2}+\frac{r}{1-\delta}}.
\end{equation*}
\end{proof}
\begin{lem}\label{art2:lem:varepsilonn} Assume $\Hrun$. Then, for any $\gamma\in(0,\gamzero]$ and for any positive function $\phi:\ER_+\rightarrow \ER$,
\begin{equation}\label{art2:eq:weppigamma}
{\cal W}_2^2(\pi,\pi^\gamma P_t)\lesssim_{r} {\cal W}_2^2(\pi,\pi^\gamma)\left(e^{-2\phi(t)} +  \frac{\phi(t)}{ \un{c} t} \Psideux^{ r} \right).
\end{equation}
\end{lem}
\begin{rem} $\rhd$ The proof of \emph{(i)} is mostly a continuous-time version of the one of Proposition \ref{art2:prop:discretconfl}.
\bigskip \newline $\rhd$ As mentioned in Remark \ref{art2:rem:exponenrates}, the proof can be adapted to provide exponential rates but unfortunately our method would lead to exponential dependence in the dimension. For this section, the lack of exponential rate does not have a serious impact on the bounds.
\newline Nevertheless,  if we needed to improve our bounds, an idea would be  to apply \cite[Theorem 5.6]{cattiaux:hal-02486264}. In this result, the authors provide exponential rates under assumptions which are similar to ours.  However the related constant depends on the density of the semi-group and it would be necessary to be able to  control it with respect to the parameters of the model.
\end{rem}

\begin{proof} Denoting by $T^x_t=\partial_x X_t^x$, the first variation process related to $(X_t^x)_{t\ge0}$, we have:
\begin{equation*}
    X_t^y-X_t^x=\int_0^1 T^{x+\lambda(y-x)}_t (y-x) d\lambda.
\end{equation*}
Thus,
\begin{equation*}
    |X_t^y-X_t^x|^2\le \int_0^1 \|T^{x+\lambda(y-x)}_t \|^2 |y-x|^2 d\lambda,
\end{equation*}
where $\|.\|$ stands for the operator norm associated with the Euclidean norm. 
Since $T^x$ is the solution to $dT_t^x=-D^2 U(X_t^x) T_t^x dt$ with $T_0^x=I_d$, 
one easily checks that
\begin{equation*}
    \|T^{x+\lambda(y-x)}_t \|^2\le e^{-2\int_0^t \underline{\lambda}_{U}(X_s^{x+\lambda(y-x)}) {\rm{d}}s}.
\end{equation*}
Thus,
\begin{equation*}
    |X_t^y-X_t^x|^2\le |x-y|^2 \int_0^1 e^{-2\int_0^t \underline{\lambda}_{U}(X_s^{x+\lambda(y-x)}) {\rm{d}}s} d\lambda.
\end{equation*}
Following the arguments of Proposition \ref{art2:prop:discretconfl}, we get for any positive function $\phi$,
\begin{equation*}
    \begin{split}
        \ES \left[ e^{-2\int_0^t \underline{\lambda}_{U}(X_s^{x+\lambda(y-x)}) {\rm{d}}s} \right] &\le  e^{-2 \phi(t)} + \frac{\phi(t)}{t } \ES  \left[  \frac{1}{t} \int_0^{t} \vpinf_{D^2 \pot\left(X_{\un{u}}^{x+\lambda(y-x)}\right)}^{-1} \mathrm{d}u \right]
        \\ & \le  e^{-2\phi(t)} + \frac{\phi(t)}{ t}  \sup_{s\in[0,t]\}} \ES  \left[ \vpinf_{D^2 \pot\left(X_{s}^{x+\lambda(y-x)}\right)}^{-1} \right].
    \end{split}
\end{equation*}
By $\Hrun$, one deduces that
\begin{align*}
    \ES[|X_t^y-X_t^x|^2]&\le |x-y|^2 \left(e^{-2\phi(t)} +\frac{\phi(t)}{ \un{c}t} \sup_{t\ge0]\}} \ES  \left[ U^{r}\left(X_{t}^{(1-\lambda) x +\lambda y}\right) \right]\right)\\
    &\lesssim_{uc}  |x-y|^2   \left(e^{-2\phi(t)} +  \frac{\phi(t)}{ \un{c} t} (U^{ r}(x)+U^{ r} (y)+\Psiun^{ r}) \right),
\end{align*}
where in the second line, we used Proposition \ref{art2:prop:moment} and the convexity of $U$. 
\newline Let now $\nu$ be a coupling of $\pi$ and $\pi^\gamma$. We have
\begin{equation*}
    {\cal W}_2^2(\pi,\pi^\gamma P_t)\le \int |x-y|^2(\nu(dx,dy)\left(e^{-2\phi(t)} + \frac{\phi(t)}{\underline{c} t} (\pi(U^{ r})+\pi^\gamma(U^{ r})+\Psiun^{r} )\right).
\end{equation*}
Taking the infimum over the set of couplings $\nu$ of $\pi$ and $\pi^\gamma$ and using again Proposition \ref{art2:prop:moment}, this yields
\begin{equation*}
    {\cal W}_2^2(\pi,\pi^\gamma P_t)\lesssim_{r} {\cal W}_2^2(\pi,\pi^\gamma)\left(e^{-2\phi(t)} +  \frac{\phi(t)}{ \underline{c} t} \Psideux^{ r} \right)
\end{equation*}
\end{proof}
\begin{lem}\label{art2:lem:psin}
$(i)$ Let $U$ be a ${\cal C}^3$-convex function. Then,
\begin{equation*}
    \ES[|X_\gamma^x-\bar{X}_\gamma^y|^2]\le |x-y|^2 e^{\lambda \gamma} +\mathfrak{c}_{\gamma,\lambda}(x,y)\gamma^3.
\end{equation*}
with 
\begin{equation*}
    \mathfrak{c}_{\gamma,\lambda}(x,y)= \frac{e^{\lambda\gamma}}{6\lambda} \left(\blip^2 |\nabla U(y)|^2+{\sigma^4}   \|\Delta (\nabla U)\|_{2,\infty}^2 + 2\lambda\sigma  \blip \|D^2 U\|_{2,\infty} \left(\sqrt{\gamma}\right)\right)
\end{equation*}
where 
\begin{equation*}
    \|D^2 U\|_{2,\infty}=\sup_{x\in\ER^d} \|D^2 U\|_F, \quad \|\Delta (\nabla U)\|_{2,\infty}=\sup_{x\in\ER^d} \sum_{i=1}^d |\Delta \partial_i U|^2,
    \end{equation*}
 and $S(x,\gamma )=\sup_{u\in[0,\gamma]} \ES[|b(\X_u^x)|^2] ^{\frac{1}{2}}$.
 \bigskip \newline $(ii)$ Let $\Hrun$ hold. Assume that $U$ is ${\cal C}^3$ with $\|\Delta (\nabla U)\|_{2,\infty}^2\lesssim_{r} \sigma^{-4} L^3 \Psideux$. Then, a constant $c_r$ exists such that for all $\gamma\in(0,\gamzero]$, for all $\lambda\in(0,1]$,
\begin{equation*}
    {\cal W}_2^2(\pi^\gamma P_{n\gamma},\pi^\gamma)\le c_r \frac{\gamma^2}{\lambda^2} e^{\lambda (n+1)\gamma}L^3 \Psideux.
\end{equation*}
\end{lem}
\begin{rem}\label{art2:rem:laplacienU}The assumption on $\Delta (\nabla U)$ is calibrated to control its contribution by $L^3 \Psideux$. That simplifies the purpose and we could keep its specific contribution at the price of technicalities. However, this assumption is not really restrictive: denoting by $A(x)$ the $d\times d$-matrix defined by
$A_{i,j}(x)=D^3_{i,j,j} U(x)$. One easily checks that
\begin{equation*}
    \|\Delta (\nabla U)\|_{2,\infty}^2=\sup_{x\in\ER^d} \|A(x)\|_F^2\le d\sup_{x\in\ER^d} \bar{\lambda}_{A(x)},
\end{equation*}
the second inequality coming from a classical inequality related to the Frobenius norm. Since   
$L\ge \sup_{x\in\ER^d} \bar{\lambda}_{D^2 U(x)}$, the assumption is for instance true if 
\begin{equation*}
    \sup_{x\in\ER^d} \bar{\lambda}_{A(x)}\le \sigma^{-2} \sup_{x\in\ER^d} (\bar{\lambda}_{D^2 U(x)})^4 \frac{\Psideux}{\sigma^2 d L}.
\end{equation*}
To conclude, note that $\frac{\Psideux}{\sigma^2 d L}$ is well controlled: for instance, under $\Hrun$ and $\Hrdeux$, 
\begin{equation*}
    \frac{\Psideux}{\sigma^2 d L}\lesssim_r \un{c}^{-1}\left(\frac{\bar{c}}{L}\vee 1 \right).
\end{equation*}
\end{rem}  
\begin{rem} The calibration of the parameter $\lambda$ is of first importance in the proof of Proposition \ref{art2:prop:HypH3} in order to avoid exponential dependence in the dimension.
\end{rem}
\begin{proof} The proof is an adaptation of Lemma 5.2 and Proposition 5.3 of \cite{art:egeapanloup2021multilevel} but with the viewpoint that $U$ is only convex. More precisely, we start with a one-step control of the error between the diffusion and its Euler scheme by setting 
\begin{equation*}
    F_{x,y}(t)=\frac{1}{2}\ES[|\X_t^x-\Xge_t^y|^2], \quad t\in[0,\gamma].
\end{equation*}
Then, setting $b=-\nabla U$, we have
\begin{equation} \label{art2:eq:basiceq}
    \begin{split}
        F'_{x,y}(t)&=\ES\left[\langle \X_t^x-\Xge_t^y, {{b}}(\X_t^x)-{{b}}(y)\rangle\right]\\
        &=\ES\left[\langle \X_t^x-\Xge_t^y, {{b}}(\X_t^x)-{{b}}(\Xge_t^y)\rangle\right]+\ES\left[\langle \X_t^x-\Xge_t^y, {{b}}(\Xge_t^y)-{{b}}(y)\rangle\right]\nonumber \\
        &\le \ES\left[\langle \X_t^x-\Xge_t^y, {{b}}(\Xge_t^y)-{{b}}(y)\rangle\right],  
    \end{split}
\end{equation}
where in the last line we used the convexity of $U$ which involves that 
\begin{equation*}
    \langle \nabla U(x)-\nabla U(y),x-y\rangle \ge 0.
\end{equation*} 
We then write
\begin{align}
    \ES\left[\langle \X_t^x-\Xge_t^y, {{b}}(\Xge_t^y)-{{b}}(y)\rangle\right]&= \ES\left[\langle \X_t^x-\Xge_t^y, {{b}}(\Xge_t^y)-b(y+\sigma \B_t)\rangle\right]\label{art2:eq:gghbis}\\
    &+\ES\left[\langle \X_t^x-\Xge_t^y,b(y+\sigma \B_t)-b(y)\rangle\right].\label{art2:eq:ggh2bis}
\end{align}
Let $\lambda>0$. For the right-hand side of \eqref{art2:eq:gghbis}, we use the elementary inequality, ${|uv|\le \frac{\lambda}{2}|u|^2+\frac{1}{2\lambda}|v|^2}$ to obtain
\begin{equation}\label{art2:eq:part111bis}
    \ES\left[\langle \X_t^x-\Xge_t^y, {{b}}(\Xge_t^y)-b(y+\sigma \B_t)\rangle\right]\le \frac{\lambda}{2} F_{x,y}(t)+\frac{t^2}{2\lambda} \blip^2 |b(y)|^2.
\end{equation}
For \eqref{art2:eq:ggh2bis}, the It\^o formula applied to $b_i=-\partial_i U$ leads to
\begin{equation*}
    b_i(y+\sigma \B_t)-b_i(y)=\sigma^2 \int_0^t \Delta b_i (y+\sigma B_s) ds+\sigma\int_0^t \langle\nabla b_i(y+\sigma B_s),dB_s\rangle.
\end{equation*}
On the one hand, setting $\Delta b=(\Delta b_i)_{i=1}^d$,
\begin{equation*}
    \ES\left[\langle \X_t^x-\Xge_t^y,\sigma^2 \int_0^t \Delta b (y+\sigma B_s) ds\rangle \right]\le \frac{\lambda}{2} F_{x,y}(t)+  \frac{\sigma^4}{{2\lambda}}  t^2 \|\Delta b\|_{2,\infty}^2,
\end{equation*}
where 
\begin{equation*}
    \|\Delta b\|_{2,\infty}^2=\sup_{x\in\ER^d} \sum_{i=1}^d |\Delta  b_i(x)|^2.
\end{equation*}
On the other hand, using the fact that ${\cal M}$ defined by ${\cal M}_t=\int_0^t \langle \nabla b(y+\sigma B_s), dB_s\rangle$   is a martingale (we refer to \cite{art:egeapanloup2021multilevel} for the details), we get
\begin{equation*}
    |\ES\left[\langle \X_t^x-\Xge_t^y,\sigma \int_0^t \langle\nabla b(y+\sigma B_s),dB_s\rangle\rangle\right]|\le \sigma \blip \|\nabla b\|_{2,\infty} \left(s^{\frac{3}{2}} S(x,\pas) +\sigma s \sqrt{ d}\right),
\end{equation*}
where $S(x,\pas )=\sup_{u\in[0,\pas]} \ES[|b(\X_u^x)|^2] ^{\frac{1}{2}}$ and 
\begin{equation*}
     \|\nabla b\|_{2,\infty} =\sup_{x\in\ER^d} \|\nabla b(x)\|_F.
 \end{equation*}
Finally, from what precedes, we deduce that
\begin{equation*}
    F'_{x,y}(t)\le   \lambda F_{x,y}(t)+\frac{t^2}{{2\lambda}}\left(\blip^2 |b(y)|^2+{\sigma^4}   \|\Delta b\|_{2,\infty}^2
    +   2\lambda\blip \|\nabla b\|_{2,\infty} \left(\sqrt{t} S(x,\pas) +\sigma  \sqrt{ d}\right)\right).
\end{equation*}
A standard Gronwall argument then leads to 
\begin{align*}
    \ES[|\X_\gamma^x-\Xge_\gamma^y|^2]&\le |x-y|^2 e^{\lambda \gamma} \\
    & +\int_0^{\gamma} s^2 e^{\lambda(\gamma -s)} \mathrm{d}s  \left(\blip^2 |b(y)|^2+{\sigma^4}   \|\Delta b\|_{2,\infty}^2
    + 2\lambda\sigma  \blip \|\nabla b\|_{2,\infty} \left(\sqrt{t} S(x,\pas) +\sigma  \sqrt{ d}\right)\right)\\
    &\le  |x-y|^2 e^{\lambda \gamma} +\gamma^3 \mathfrak{c}_{\gamma,\lambda}(x,y)
\end{align*}
with 
\begin{equation*}
    \mathfrak{c}_{\gamma,\lambda}(x,y)= \frac{e^{\lambda\gamma}}{6\lambda} \left(\blip^2 |b(y)|^2+{\sigma^4}   \|\Delta b\|_{2,\infty}^2
    + 2\lambda\sigma  \blip \|\nabla b\|_{2,\infty} \left(\sqrt{\gamma} S(x,\pas) +\sigma  \sqrt{ d}\right)\right).
\end{equation*}
$(ii)$ Iterating the above inequality, we obtain for each $n\ge 1$,
\begin{align*}
    \ES[|\X_{n\gamma}^x-\Xge_{n\gamma}^x|^2]\le \gamma^3  \sum_{k=0}^{n-1}\ES[\mathfrak{c}_{\pas,\lambda}(X_{k\gamma}^x,\bar{X}_{k\gamma}^x))] e^{\lambda (n-k)\gamma}
\end{align*}
Integrating the initial condition with respect to $\pi^\gamma$, we get
\begin{align*}
    &{\cal W}_2(\pi^\gamma P_{n\gamma},\pi^\gamma)\le \gamma^2 e^{\lambda n\gamma} \sup_{k\ge 0} \int\ES[\mathfrak{c}_{\pas,\lambda}(X_{k\gamma}^x,\bar{X}_{k\gamma}^x))]\pi^\gamma(dx)\\
    &\le \frac{\gamma^2}{\lambda^2} e^{\lambda (n+1)\gamma} \left(\blip^2 \pi^\gamma(|b|^2)+{\sigma^4}   \|\Delta b\|_{2,\infty}^2
    + 2\lambda\sigma  \blip \|\nabla b\|_{2,\infty} \left(\sqrt{\gamma} \sup_{n\ge 0} \int \ES[S(X_{n\gamma}^x,\pas)]\pi^\gamma(dx) +\sigma  \sqrt{ d}\right)\right),
\end{align*}
where in the second line, we used the stationarity property of $\pi^\gamma$. Now, under $\Hrun$, $|b|^2=|\nabla U|^2\le 2 L U$ (with the same idea than one which leads to \eqref{art2:eq:minogradU}) so that by Proposition \ref{art2:prop:moment}$(iii)$, $\pi^\gamma(|b|^2)\lesssim_{r} L\Psideux$.
On the other hand, by the It\^o formula and the fact that $\Delta U\le d L$,
\begin{equation*}
    \ES[U(X_t^x)]\le U(x)+\sigma^2 \int_0^t\ES[\Delta U(X_s^x)] ds \le U(x)+\sigma^2 d L,
\end{equation*}
so that 
\begin{equation*}
    S(x,\pas)\le \sqrt{2L} \sup_{u\in[0,\pas]} \ES[U(X_u^x)] ^{\frac{1}{2}}\le \sqrt{2L U(x)}+\sigma \sqrt{2d}L.
\end{equation*}
Again, with the help of Proposition \ref{art2:prop:moment}\emph{(iii)},
\begin{equation*}
    \sup_{n\ge 0} \int \ES[S(X_{n\gamma}^x,\pas)]\pi^\gamma(dx)\lesssim_{r} \sqrt{L}\Psideux^\frac{1}{2}+\sigma \sqrt{d}L.
\end{equation*}
Thus, using that  $\gamma\le L^{-1}$,
\begin{equation*}
    {\cal W}_2(\pi^\gamma P_{n\gamma},\pi^\gamma)\lesssim_{r}\frac{\gamma^2}{\lambda^2} e^{\lambda (n+1)\gamma} \left(L^3 \Psideux+{\sigma^4}   \|\Delta b\|_{2,\infty}^2+\sigma \lambda  \|\nabla b\|_{2,\infty}(\sqrt{\Psideux}+\sigma \sqrt{d L})\right).
\end{equation*}
Since for a symmetric $d\times d$-matrix $A$, $\|A\|_F\le \sqrt{d}\bar{\lambda}_A$, one deduces that 
$\|\nabla b\|_{2,\infty}=\|D^2 U\|_{2,\infty}\le \sqrt{d} L$. It easily follows that $\sigma\lambda  \blip \|\nabla b\|_{2,\infty}(\sqrt{\Psideux}+\sigma \sqrt{d L})\le L^3\Psideux$ (using that $L\ge1$ and $\Psideux\ge d$). The result follows.
\end{proof}
\section{Proof of Theorem \ref{art2:theo:maintheo2}}\label{art2:sec:maintheo2}
Following the bias-variance decomposition of the MSE: 
\begin{equation*}
     \| \mathcal{Y}(f) - \pi(f) \|_2^2  \le \left[\ES [\mathcal{Y}(f)] - \pi(f)\right]^2 + \mathrm{Var}( \mathcal{Y}(f)),
\end{equation*}
we successively study the bias and the variance contributions and end the section by the proof of 
Theorem \ref{art2:theo:maintheo2}.

\subsection{Step 1: Bias of the procedure}
{In the sequel, $\mathcal{Y}(\lev,\left(\pas_\ind\right)_\ind,\tau,\left(T_\ind\right)_\ind,f)$ is usually written $\mathcal{Y}$ for the sake of simplicity.}  We start with a telescopic-type decomposition:
\begin{equation}\label{art2:eq:telescop}
    \begin{split}
        \mathcal{Y}(f)-\pi(f) &= \frac{1}{T_0}\int_0^{T_0} f(\bar{X}_{\un{s}_{\pas_0}}^{\pas_0,x_0}) - \pi^{\pas_0} (f) \mathrm{d}s
        \\ & + \sum_{\ind=1}^\lev \left(\frac{1}{T_\ind}\int_0^{T_\ind}f(\bar{X}_{\un{s}_{\pas_{\ind-1}}}^{\pas_\ind,x_0})- \pi^{\pas_\ind} (f) \mathrm{d}s-\frac{1}{T_\ind}\int_0^{T_\ind}f(\bar{X}_{\un{s}_{\pas_{\ind-1}}}^{\pas_{\ind-1},x_0})-\pi^{\pas_{\ind-1}} (f)  \mathrm{d}s \right)
        \\ & + \pi^{\pas_\lev}(f)-\pi(f). 
\end{split}
\end{equation}
Let us now study the bias generated by the first and second terms of the right-hand side of \eqref{art2:eq:telescop}.
\begin{lem}\label{art2:couchcontr}
Assume $\Hrun$ and $\pas_0 \in (0,\gamzero]$. Let  $x \in \ER^d$ such that $U(x)\lesssim_r \Psideux$. Then, for any $r\in[0,1)$ and $\delta\in(0,\frac{1}{2}]$, there exists a constant $c_{r,\delta}$ (depending only on $r$ and $\delta$) such that for all $T\ge 1$, for all Lipschitz continuous function $f:\ER^d\rightarrow\ER$,
\begin{equation*}
    \left| \frac{1}{T}\int_0^{T} \mathbb{E}_{x}[f(\bar{X}_{\un{s}}^{\pas,x_0})] - \pi^{\pas} (f) \mathrm{d}s \right| ^2 \le c_{r,\delta} \frac{[f]_1^2 \mathfrak{C}^{(1)}_{\mathrm{bias}}}{T^2} 
    \quad\textnormal{with}\quad 
    \mathfrak{C}^{(1)}_{\mathrm{bias}}=\left(\un{c}^{-\frac{1}{2}}\vee\un{c}^{-3-\frac{4\delta}{1-\delta}}\right)\Psideux^{{1+3r}+\frac{4\delta r}{1-\delta}}.
\end{equation*}\end{lem}
\begin{proof}
Let us apply Proposition \ref{art2:prop:HypH1} with $\phi(n)=\frac{n\gamma}{(n\gamma+1)^{1-\delta}}$ and $\kappa=\frac{2(1+\delta)}{1-\delta}=2+\frac{4\delta}{1-\delta}$ for $\delta\in(0,1/2]$. Using that 
$\phi(n)\ge (n\gamma)^\delta$, we have
\begin{equation*}
    \begin{split}
       \left| \frac{1}{T}\int_0^{T} \mathbb{E}_{x}[f(\bar{X}_{\un{s}}^{\pas,x_0})] - \pi^{\pas} (f) \mathrm{d}s \right|  &\lesssim_{r,\kappa} \frac{[f]_1}{T} \int_0^T {h_{\phi,\kappa}(\un{s})} \mathrm{d}s\\
        \\ & \lesssim_{r} \frac{[f]_1 }{T}\left( \un{c}^{-\frac{1}{2}}\Psideux^{\frac{1+r}{2}} \int_0^T e^{- \un{t}^{\delta}}dt+\un{c}^{-\frac{3}{2}-\frac{2\delta}{1-\delta}}\Psideux^{\frac{1+3r}{2}+\frac{2\delta r}{1-\delta}}\int_0^T\left(1+\un{t} \right)^{-1-\delta} \mathrm{d}t\right)\\
        &\lesssim_r \frac{[f]_1 }{\delta T} \left(\Gamma\left(\frac{1}{\delta}\right)\un{c}^{-\frac{1}{2}}\Psideux^{\frac{1+r}{2}}+\un{c}^{-\frac{3}{2}-\frac{2\delta}{1-\delta}}\Psideux^{\frac{1+3r}{2}+\frac{2\delta r}{1-\delta}}\right),
        \end{split}
\end{equation*}
where in the last line, we used standard arguments of comparisons between series and integrals.
The result follows.
\end{proof}
We are now ready to state a proposition about the control of the bias of the procedure.
\begin{prop}\label{art2:prop:biascontr}
Assume $\Hrun$ and $\pas_0 \in (0,\gamzero]$. Let  $x \in \ER^d$ such that $U(x)\lesssim_r \Psideux$. Let $\delta \in (0,\frac{1}{2}]$ and let $f$ be a continuous Lipschitz function with $[f]_1= 1$. Let $\lev \in\mathbb{N}^*$. Then,
\newline (i) 
\begin{equation*} 
        \left| \mathbb{E}_{x}[\mathcal{Y}(f)]-\pi(f) \right|^2 \lesssim_{r,\delta} 
         \mathfrak{C}^{(1)}_{\mathrm{bias}}
         \sum_{\ind=0}^\lev \frac{1}{T_\ind^2}+  \mathfrak{C^{(2,1)}_{\mathrm{bias}}}   \pas_\lev \quad\textnormal{with}\quad\mathfrak{C^{(2,1)}_{\mathrm{bias}}}=\frac{L}{\un{c}^{\frac{2}{1-\delta}}\wedge\un{c}}\Psideux^{1+3r+\frac{2\delta r}{1-\delta}}.
\end{equation*}
\noindent (ii) If the assumptions of Proposition \ref{art2:prop:HypH3} are fulfilled,
\begin{equation*} 
    \left| \mathbb{E}_{x}[\mathcal{Y}(f)]-\pi(f) \right|^2 \lesssim_{r,\delta} 
    \mathfrak{C}^{(1)}_{\mathrm{bias}}
    \sum_{\ind=0}^\lev \frac{1}{T_\ind^2}+  \mathfrak{C^{(2,2)}_{\mathrm{bias}}}   \pas_\lev^2 \quad\textnormal{with}\quad\mathfrak{C^{(2,2)}_{\mathrm{bias}}}=\un{c}^{-\frac{2}{1-\delta}} L^{{3}}\Psideux^{1+\frac{2r}{1-\delta}}.
\end{equation*}
\end{prop}
\begin{proof}
\emph{(i)} Taking the expectation in \eqref{art2:eq:telescop}, we obtain:
\begin{equation*}
    \begin{split}
        \left| \mathbb{E}_x[\mathcal{Y}(f)-\pi(f)] \right| &\le \left|\frac{1}{T_0}\int_0^{T_0}\ES_{x_0}\left[ f(\bar{X}_{\un{s}_{\pas_0}}^{\pas_0,x_0})\right]- \pi^{\pas_0} (f) \mathrm{d}s \right| 
         + \sum_{\ind=1}^\lev  \frac{1}{T_\ind}\left|\int_0^{T_\ind}\ES\left[f(\bar{X}_{\un{s}_{\pas_{\ind-1}}}^{\pas_\ind,x_0})\right]- \pi^{\pas_\ind} (f) \mathrm{d}s\right|
        \\&+ \sum_{\ind=1}^\lev\frac{1}{T_\ind}\left|\int_0^{T_\ind}\ES\left[ f(\bar{X}_{\un{s}_{\pas_{\ind-1}}}^{\pas_{\ind-1},x_0})\right]-\pi^{\pas_{\ind-1}} (f)  \mathrm{d}s\right|
         + \left| \pi^{\pas_\lev}(f)-\pi(f) \right|.
    \end{split}
\end{equation*}
For the three first terms, we apply Lemma \ref{art2:couchcontr} and for the last one, Lemma \ref{art2:lem:firstresult}.  The result follows.
\bigskip \newline \emph{(ii)} It is the same proof using Proposition \ref{art2:prop:HypH3} to control the last term (instead of 
Lemma \ref{art2:lem:firstresult}).
\end{proof}

\subsection{Step 2 : Control of the variance} 
Now we have to control the variance of our estimator. Owing to the independency between the layers,
\begin{equation}\label{art2:eq:decompofvar}
    \mathrm{Var}(\mathcal{Y}(\lev,\left(\pas_\ind\right)_\ind,\left(T_\ind\right)_\ind,f)) = \mathrm{Var}\left(\frac{1}{T_0}\int_{0}^{T_0} f(\Xge_{\un{s}_{\pas_0}}^{\pas_0,x}) \mathrm{d}s\right) + \sum_{\ind=1}^\lev \mathrm{Var}\left( \frac{1}{T_\ind} \int_0^{T_\ind}G_{s}^{\pas_{\ind}}\mathrm{d}s\right),
\end{equation}
where for some given $\pas>0$ and $s>0$,
\begin{equation*}
    \begin{split}
        G_{s}^\pas &= f\left(\Xge_{\un{s}_\pas}^{\frac{\pas}{2},x}\right) - f\left(\Xge_{\un{s}_\pas}^{\pas,x}\right).
    \end{split}
 \end{equation*}
Before going further, let us recall that in order that the multilevel method be efficient, the correcting layers must have a small variance. In the long-time setting, this involves to be able to control the $L^2$-distance between couplings of Euler schemes with steps $\gamma$ and $\gamma/2$. By Proposition \ref{art2:prop:HypH2}, this is still possible under $\Hrun$, and such a property  allows to obtain the following result:
\begin{lem}\label{art2:lem:propvarlvl}
    Assume $\Hrun$ and $\pas_0 \in (0,\gamzero]$. Let  $x \in \ER^d$ such that $U(x)\lesssim_r \Psideux$. Let $\delta \in (0,\frac{1}{2}]$ and $\kappa>\frac{2}{1-\delta}$. Let $f$ be a continuous Lipschitz function.Then, for all $T>0$, 
    \begin{equation*} 
        \mathrm{Var} \left( \frac{1}{T} \int_0^{T} G_s^\pas \mathrm{d}s \right) \lesssim_{r,\delta,\kappa}  [f]_1^2 \frac{\mathfrak{C_{\rm var}^{(1)}}(\kappa,\delta) \gamma^{1-\frac{1}{\kappa(1-\delta)}}}{T},
    \end{equation*}
    where 
    \begin{equation*}
        \mathfrak{C_{\rm var}^{(1)}}(\kappa,\delta)=  L^{1-\frac{1}{\kappa(1-\delta)}} (\un{c}^{-\frac{3-\delta}{\kappa(1-\delta)^2}}\vee \un{c}^{-\frac{3+\kappa-(\kappa+1)\delta}{\kappa(1-\delta)^2}})\Psideux^{1+\left(4-\frac{2}{\kappa}\right)r+\varepsilon(\delta,\kappa) r}\;\textnormal{with}\;
        \varepsilon(\delta,\kappa)=\left(3-\frac{2}{\kappa}-\frac{2}{\kappa(1-\delta)}\right) \frac{\delta}{1-\delta}.
    \end{equation*} 
\end{lem}
\begin{rem} In the uniformly convex case, the variance is controlled by 
$\frac{\gamma \log(1/\gamma)}{T}$ whereas, here, we are only able to obtain $\frac{\gamma^{1-\frac{1}{\kappa(1-\delta)}}}{T}$. This difference is due to the lack of exponential convergence to equilibrium under our assumptions. Note that if we leave $\kappa$ go to $\infty$, we are moving ever closer to the uniformly convex bound. However, the constant depends on $\kappa$ and explodes when $\kappa\rightarrow+\infty$. The interesting point is that the exponent of $\Psideux$ remains bounded when $\kappa\rightarrow+\infty$, which means that the dependence in the dimension is slightly impacted by the choice of $\kappa$.
\end{rem}
\begin{proof}
A standard computation shows that
\begin{equation*} 
    \mathrm{Var} \left( \frac{1}{T} \int_0^{T} G_s^\pas \mathrm{d}s \right)  \le  \frac{2}{T^2} \int_0^{T} \int_u^{T} \mathrm{Cov} \left(G_s^\pas,G_u^\pas\right) \mathrm{d}s \mathrm{d}u.
\end{equation*}
First, at the price of replacing $f$ by $f/[f]_1$,we can assume in the sequel that $[f]_1\le 1$. Then,
\begin{equation*} 
    |G_{s}^\pas|\le |\Xge_{\un{s}_\pas}^{\frac{\pas}{2},x}-X_{\un{s}_\pas}|+|\Xge_{\un{s}_\pas}^{\pas,x}-X_{\un{s}_\pas}|.
\end{equation*} 
By Proposition \ref{art2:prop:HypH2} and the fact that  $\pot(x)\lesssim_{uc} \Psideux$, we deduce that for every $\delta\in(0,1]$,
\begin{equation}\label{art2:controlgupas}
\ES[ |G_u^\pas|^2]\lesssim_{r,\delta} \frac{L\gamma}{\un{c}^{\frac{2}{1-\delta}}}\Psideux^{1+3r+\frac{2\delta r}{1-\delta}}.
\end{equation}
This yields a first bound for $\mathrm{Cov} \left(G_s^\pas,G_u^\pas\right)$:
\begin{equation*} 
    \mathrm{Cov} \left(G_s^\pas,G_u^\pas\right)\le \ES[ |G_s^\pas|^2]^{\frac{1}{2}}\ES[ |G_u^\pas|^2]^{\frac{1}{2}}\lesssim_{r,\delta} \frac{L\gamma}{\un{c}^{\frac{2}{1-\delta}}}\Psideux^{1+3r+\frac{2\delta r}{1-\delta}}.
\end{equation*} 
Hence, for any $t_0>0$,
\begin{equation}\label{art2:eq:covgsgtbound1}
    \int_u^{u+t_0} \mathrm{Cov} \left(G_s^\pas,G_u^\pas\right) \mathrm{d}s\lesssim_{r,\delta} \frac{L\gamma t_0}{\un{c}^{\frac{2}{1-\delta}}}\Psideux^{1+3r+\frac{2\delta r}{1-\delta}}.
 \end{equation}
We now want to take advantage of the convergence to equilibrium to get a second bound when $s-u\ge t_0$: since  $G_u^{\pas}$ is ${\cal F}_{\un{u}_{\pas}}$-measurable, we have  for any $s\ge \un{u}_{\pas}$,
\begin{equation*}
     \ES\left[G_s^{\pas} G_u^{\pas}\right]=\ES[\ES[G_s^\pas|\mathcal{F}_{\un{u}_{\pas}}]G_u^\pas].
\end{equation*}
Setting  $ \FFF(\pas,t,x)=\ES[f(\Xge_t^{\pas,x})]-\pi^\gamma(f)$, we deduce from the Markov property that
\begin{equation*}
    \mathbb{E}[G_s^\pas|\mathcal{F}_{\un{u}_{\pas}}]=  \FFF\left(\frac{\pas}{2}, \un{s}_\pas-\un{u}_\pas, \Xge_{\un{u}_{\pas}}^{\frac{\pas}{2},x}\right)-\FFF\left({\pas}, \un{s}_\pas-\un{u}_\pas, \bar{X}_{\un{u}_{\pas}}^{\pas,x}\right) +(\pi^\gamma-\pi^\frac{\gamma}{2})(f),
\end{equation*} 
and hence,
\begin{align*}
    \ES\left[G_s^{\pas} G_u^{\pas}\right]=  \ES \left[\left(\FFF\left(\frac{\pas}{2}, \un{s}_\pas-\un{u}_\pas, \Xge_{\un{u}_{\pas}}^{\frac{\pas}{2},x}\right)-\FFF\left({\pas}, \un{s}_\pas-\un{u}_\pas, \bar{X}_{\un{u}_{\pas}}^{\pas,x}\right)\right)G_u^\pas \right]+(\pi^\gamma-\pi^\frac{\gamma}{2})(f)\ES[G_u^\pas].
\end{align*}
On the other hand,
\begin{align*}
\ES[G_s^{\pas}] \ES[G_u^{\pas}]= \ES \left[\left(\FFF\left(\frac{\pas}{2}, \un{s}_\pas, \Xge_{\un{u}_{\pas}}^{\frac{\pas}{2},x}\right)-\FFF\left({\pas}, \un{s}_\pas, \bar{X}_{\un{u}_{\pas}}^{\pas,x}\right)\right)\right]\ES\left[G_u^\pas \right]+(\pi^\gamma-\pi^\frac{\gamma}{2})(f)\ES[G_u^\pas].
\end{align*}
As a consequence,
\begin{align}
    \mathrm{Cov} \left(G_s^\pas,G_u^\pas\right)&= \ES \left[\left(\FFF\left(\frac{\pas}{2}, \un{s}_\pas-\un{u}_\pas, \Xge_{\un{u}_{\pas}}^{\frac{\pas}{2},x}\right)-\FFF\left({\pas}, \un{s}_\pas-\un{u}_\pas, \bar{X}_{\un{u}_{\pas}}^{\pas,x}\right)\right)G_u^\pas \right]\label{art2:eq:gupasun}\\
    &-\ES \left[\left(\FFF\left(\frac{\pas}{2}, \un{s}_\pas, x\right)-\FFF\left({\pas}, \un{s}_\pas, x\right)\right)\right]\ES\left[G_u^\pas \right].\label{art2:eq:gupasdeux}
\end{align}
Let us study the two right-hand members successively. 
For \eqref{art2:eq:gupasun}, the Cauchy-Schwarz inequality and \eqref{art2:controlgupas} yield:
\begin{align*}
    \ES \Big[\Big(\FFF &\left(\frac{\pas}{2}, \un{s}_\pas -\un{u}_\pas , \Xge_{\un{u}_{\pas}}^{\frac{\pas}{2},x}\right)-\FFF\left({\pas}, \un{s}_\pas-\un{u}_\pas, \bar{X}_{\un{u}_{\pas}}^{\pas,x}\right)\Big)G_u^\pas \Big]\\
    &\lesssim_{r,\delta}  \left(\left\|\FFF\left(\frac{\pas}{2}, \un{s}_\pas-\un{u}_\pas, \Xge_{\un{u}_{\pas}}^{\frac{\pas}{2},x}\right)\right\|_2+\left\|
    \FFF\left({\pas}, \un{s}_\pas-\un{u}_\pas, \bar{X}_{\un{u}_{\pas}}^{\pas,x}\right)\right\|_2\right)\sqrt{\frac{L\gamma}{\un{c}^{\frac{2}{1-\delta}}}\Psideux^{1+3r+\frac{2\delta r}{1-\delta}}}.
\end{align*}
By Proposition \ref{art2:prop:HypH1} or more precisely by \eqref{art2:eq:vitessexxx} combined with Proposition \ref{art2:prop:moment}\emph{(iii)}\footnote{\label{art2:footnoteone} In fact, Proposition \ref{art2:prop:HypH1} is written under the assumption $U(x)\lesssim \Psideux$ but here, we need to integrate with respect to the initial condition. To extend to this setting, the idea is to start from   \eqref{art2:eq:vitessexxx} and to use the bounds of Proposition \ref{art2:prop:moment}\emph{(iii)}. This allows us to retrieve controls which are similar to Proposition \ref{art2:prop:HypH1}.} applied with $\phi(n)=(n\gamma)^{\delta}$, 
\begin{equation*}
    \begin{split}
        \left\|\FFF\left(\frac{\pas}{2}, \un{s}_\pas-\un{u}_\pas, \Xge_{\un{u}_{\pas}}^{\frac{\pas}{2},x}\right)\right\|_2+\left\|
        \FFF\left({\pas}, \un{s}_\pas-\un{u}_\pas, \bar{X}_{\un{u}_{\pas}}^{\pas,x}\right)\right\|_2  \lesssim_{r,\delta,\kappa} \un{c}^{-\frac{1}{2}} \Psideux^{\frac{1+r}{2}} e^{-(\un{s}_\pas-\un{u}_\pas)^\delta} &
        \\ +\un{c}^{-\frac{\kappa+1}{2}}\Psideux^{\frac{1+r(1+\kappa)}{2}} (\un{s}_\pas-\un{u}_\pas)^{-\frac{\kappa(1-\delta)}{2}}.  
    \end{split}
\end{equation*}
Now, let us remark that if $s-u\ge t_0$, $t_0\ge 2$ and $\gamma\in(0,\gamzero]$, then $\un{s}_\pas-\un{u}_\pas\ge 1$ (since $\gamzero\le 1$). Noting that for any $t\ge 1$, $e^{-t^\delta}\lesssim_{\delta} t^{-\frac{\kappa(1-\delta)}{2}}$ and that for any $\kappa>\frac{2}{1-\delta}$,
\begin{equation*}
    \int_{(u+t_0)\wedge T}^T(\un{s}_\pas-\un{u}_\pas)^{-\frac{\kappa(1-\delta)}{2}} \mathrm{d}s\le \int_{u+t_0}^{+\infty} (s-u-\gamma)^{-\frac{\kappa(1-\delta)}{2}} \mathrm{d}s\le \frac{(t_0-\gamma)^{1-\frac{\kappa(1-\delta)}{2}}}{\frac{\kappa(1-\delta)}{2}-1}\lesssim_{\kappa,\delta} \left(\frac{t_0}{2}\right)^{1-\frac{\kappa(1-\delta)}{2}}
\end{equation*}
we deduce that  for any $t_0\ge 2$ and for any $\gamma\in(0,\gamzero]$
\begin{align*}
    \int_{(u+t_0)\wedge T}^T \ES \Big[\Big(\FFF\left(\frac{\pas}{2}, \un{s}_\pas-\un{u}_\pas, \Xge_{\un{u}_{\pas}}^{\frac{\pas}{2},x}\right)-\FFF\left({\pas}, \un{s}_\pas-\un{u}_\pas, \bar{X}_{\un{u}_{\pas}}^{\pas,x}\right)\Big)G_u^\pas \Big] \mathrm{d}s\lesssim_{r,\delta,\kappa}  \frac{(\gamma L)^{\frac{1}{2}} \Psideux^{1+2r+\frac{\delta r}{1-\delta}+\frac{\kappa r}{2}}}{\un{c}^{\frac{3}{2}+\frac{\delta}{1-\delta}}\wedge \un{c}^{\frac{3+\kappa}{2}+\frac{\delta}{1-\delta}}} \left(\frac{t_0}{2}\right)^{1-\frac{\kappa(1-\delta)}{2}}.
\end{align*}
For \eqref{art2:eq:gupasdeux}, using that $\un{s}_\gamma\ge \un{s}_\pas-\un{u}_\pas$, we remark that we can obtain the same bound so that:
\begin{equation*}
    \int_{(u+t_0)\wedge T}^T \mathrm{Cov} \left(G_s^\pas,G_u^\pas\right)\mathrm{d}s \lesssim_{r,\delta,\kappa}  \frac{(\gamma L)^{\frac{1}{2}} }{\un{c}^{\frac{3}{2}+\frac{\delta}{1-\delta}}\wedge \un{c}^{\frac{3+\kappa}{2}+\frac{\delta}{1-\delta}}}\Psideux^{1+2r+\frac{\delta r}{1-\delta}+\frac{\kappa r}{2}}\left(\frac{t_0}{2}\right)^{1-\frac{\kappa(1-\delta)}{2}}.
\end{equation*}
In view of the above bound and of the one obtained in \eqref{art2:eq:covgsgtbound1}, we now optimize the choice of $t_0$ by taking $t_0$ solution to:
\begin{equation*}
    \frac{(\gamma L)^{\frac{1}{2}} }{\un{c}^{\frac{3}{2}+\frac{\delta}{1-\delta}}\wedge \un{c}^{\frac{3+\kappa}{2}+\frac{\delta}{1-\delta}}} \Psideux^{1+2r+\frac{\delta r}{1-\delta}+\frac{\kappa r}{2}}\left(\frac{t_0}{2}\right)^{1-\frac{\kappa(1-\delta)}{2}}=
    \frac{L\gamma t_0}{\un{c}^{\frac{2}{1-\delta}}}\Psideux^{1+3r+\frac{2\delta r}{1-\delta}}.
\end{equation*}
\emph{i.e.}, 
\begin{equation*}
    t_0=2^{1-\frac{2}{\kappa(1-\delta)}}(\un{c}^{\frac{1}{2}+\frac{\delta}{1-\delta}}\vee \un{c}^{\frac{1-\kappa}{2}+\frac{\delta}{1-\delta}})^{\frac{2}{\kappa(1-\delta)}} (L\gamma)^{-\frac{1}{\kappa(1-\delta)}}\Psideux^{(1-\frac{2}{\kappa}) \frac{r}{1-\delta}-\frac{2\delta}{\kappa(1-\delta)^2}r}.
\end{equation*}
Plugging this value of $t_0$ into \eqref{art2:eq:covgsgtbound1}, this leads to: for any $\kappa>\frac{2}{1-\delta}$,
\begin{equation*}
    \int_u^T \mathrm{Cov} \left(G_s^\pas,G_u^\pas\right) \mathrm{d}s\lesssim_{r,\delta,\kappa} (\gamma L)^{1-\frac{1}{\kappa(1-\delta)}} (\un{c}^{-\frac{3}{2}-\frac{\delta}{1-\delta}}\vee \un{c}^{-\frac{3+\kappa}{2}-\frac{\delta}{1-\delta}})^{\frac{2}{\kappa(1-\delta)}} 
    \Psideux^{1+\left(4-\frac{2}{\kappa}\right)r+\varepsilon(\delta,\kappa) r}, 
\end{equation*}
where $\varepsilon(\delta,\kappa)=\left(3-\frac{2}{\kappa}-\frac{2}{\kappa(1-\delta)}\right) \frac{\delta}{1-\delta}$. The result follows.
\end{proof}


In the next proposition, we are now able to work on the variance of the multilevel procedure.
\begin{prop}\label{art2:varcontr}
Assume $\Hrun$ and $\pas_0 \in (0,\gamzero]$. Let  $x \in \ER^d$ such that $U(x)\lesssim_r \Psideux$. Let $\delta \in (0,\frac{1}{2}]$ and $\kappa>\frac{2}{1-\delta}$. Let $f$ be a continuous Lipschitz function.Then, 
\begin{equation*}
        \mathrm{Var}(\mathcal{Y}(\lev,\left(\pas_\ind\right)_\ind,\left(T_\ind\right)_\ind,f)) \lesssim_{r,\delta,\kappa}
         [f]_1^2 \left(\mathfrak{C_{\rm var}^{(1)}}(\kappa,\delta)\sum_{\ind=1}^\lev  \frac{\gamma^{1-\frac{1}{\kappa(1-\delta)}}_\ind}{T_\ind}+
       \frac{\mathfrak{C_{\rm var}^{(2)}}}{T_0}\right),                
\end{equation*}
where $\mathfrak{C_{\rm var}^{(1)}}(\kappa,\delta)$ is defined in Lemma \ref{art2:lem:propvarlvl} and $\mathfrak{C_{\rm var}^{(2)}}= \left(\un{c}^{-\frac{3}{2}}\vee  \un{c}^{-\frac{5}{2}-\frac{2\delta}{1-\delta}}\right) \Psideux^{1+2r+\frac{2\delta r}{1-\delta}}$.
\end{prop}


\begin{proof}
We assume (without loss of generality) that $[f]_1=1$ and $f(x^\star)=0$. In view of the decomposition   \eqref{art2:eq:decompofvar}, we apply Lemma \ref{art2:lem:propvarlvl} for each level $\ind\in\{1,\ldots,\lev\}$ with $T=T_\ind$ and $\pas=\pas_{\ind-1}$. We obtain for any $\delta\in(0,1/2]$ and $\kappa>2/(1-\delta)$,
\begin{equation*}
        \mathrm{Var}(\mathcal{Y}(\lev,\left(\pas_\ind\right)_\ind,\left(T_\ind\right)_\ind,f))  \le \mathrm{Var}\left(\frac{1}{T_0}\int_{0}^{T_0} f(\Xge_{\un{s}_{\pas_0}}^{\pas_0,x_0}) \mathrm{d}s\right) +\mathfrak{C_{\rm var}^{(1)}}(\kappa,\delta)\sum_{\ind=1}^\lev  \frac{\gamma^{1-\frac{1}{\kappa(1-\delta)}}_\ind}{T_\ind}.
\end{equation*}
It remains to control the first term, \emph{i.e.} the variance related to the first level. We use similar arguments as in the proof of Lemma
\ref{art2:lem:propvarlvl} (see in particular \eqref{art2:eq:gupasun} and what follows). First, one can check that for every $0\le u\le s\le T$,
\begin{align*}
    \mathrm{Cov} \left(f(\Xge_{\un{u}_{\pas_0}}^{\pas_0,x_0}),f(\Xge_{\un{s}_{\pas_0}}^{\pas_0,x_0})\right)&= \ES \left[\FFF\left(\pas_0, \un{s}_{\pas_0}-\un{u}_{\pas_0}, \Xge_{\un{u}_{\pas_0}}^{\pas_0,x}\right)f(\Xge_{\un{u}_{\pas_0}}^{\pas_0,x_0}) \right]-\ES \left[\FFF\left(\pas_0, \un{s}_{\pas_0}, x\right)\right]\ES\left[f(\Xge_{\un{u}_{\pas_0}}^{\pas_0,x_0}) \right]\\
    &\le \left(\left\|\FFF\left(\pas_0, \un{s}_{\pas_0}-\un{u}_{\pas_0}, \Xge_{\un{u}_{\pas_0}}^{\pas_0,x}\right)\right\|_2+\left\|\FFF\left(\pas_0, \un{s}_{\pas_0}, x\right)\right\|_2\right) \left\|f\left(\Xge_{\un{u}_{\pas_0}}^{\pas_0,x_0}\right)\right\|_2,
\end{align*}
by Cauchy-Schwarz inequality. Then, by Proposition \ref{art2:prop:HypH1} (and footnote \ref{art2:footnoteone}) applied with $\phi(n)=n\gamma$ and $\kappa=\frac{2(1+\delta)}{1-\delta}$, one deduces that (we leave the details to the reader),
\begin{align*}
    \int_u^T \mathrm{Cov} \left(f(\Xge_{\un{u}_{\pas_0}}^{\pas_0,x_0}),f(\Xge_{\un{s}_{\pas_0}}^{\pas_0,x_0})\right) ds&\lesssim_{r,\delta}
    \left(\un{c}^{-\frac{1}{2}}\vee \un{c}^{-\frac{3}{2}-\frac{2\delta}{1-\delta}} \right)\Psideux^{\frac{1+3r}{2}+\frac{2\delta r}{1-\delta}} \times \sup_{u\ge0} \left\|f\left(\Xge_{\un{u}_{\pas_0}}^{\pas_0,x_0}\right)\right\|_2 \\
    &\lesssim_{r,\delta} \left(\un{c}^{-\frac{3}{2}}\vee  \un{c}^{-\frac{5}{2}-\frac{2\delta}{1-\delta}}\right)  \Psideux^{1+2r+\frac{2\delta r}{1-\delta}},
\end{align*}
where in the second line, we used Proposition \ref{art2:prop:moment}\emph{(iii)} and the fact that (by Lemma \ref{art2:lem:momenttoU})
\begin{equation*}
    |f(x)|^2=|f(x)-f(x^\star)|^2\le |x-x^\star|^2\le \frac{\pot^{1+r}(x)}{\un{c}}.    
\end{equation*}
\end{proof}
\subsection{Step 3 : Proof of Theorem \ref{art2:theo:maintheo2}}
Back to the bias-variance decomposition, we deduce from Proposition \ref{art2:prop:biascontr} and Proposition \ref{art2:varcontr} that, up to a constant depending on $\kappa$, $\delta$ and $r$, the MSE is lower than $\varepsilon^2$ if the following conditions are satisfied (with $\delta\in(0,1/2]$ and $\kappa>2/(1-\delta)$):
\begin{equation}\label{art2:eq:troisconditions}
    \begin{cases}
        \displaystyle{\emph{(i.a)}: \quad   \mathfrak{C}^{(1)}_{\mathrm{bias}}
         \sum_{\ind=0}^\lev \frac{1}{T_\ind^2}\le \varepsilon^2,}&\displaystyle{\emph{(i.b)}: \quad \mathfrak{C^{(2,1)}_{\mathrm{bias}}}  \pas_\lev\le \varepsilon^2   \quad \textnormal{or} \quad \displaystyle{(i.b)'}:\quad \mathfrak{C^{(2,2)}_{\mathrm{bias}}}   \pas_\lev^2\le \varepsilon^2, }\\
        \displaystyle{\emph{(ii.a)}:\quad  \frac{ \mathfrak{C_{\rm var}^{(2)}}}{T_0}\le \varepsilon^2,}&
        \displaystyle{ \emph{(ii.b)}: \quad \mathfrak{C_{\rm var}^{(1)}}(\kappa,\delta)\sum_{\ind=1}^\lev  \frac{\gamma^{1-\frac{1}{\kappa(1-\delta)}}_\ind}{T_\ind}\le \varepsilon^2.}
        \end{cases}
\end{equation}
Note that \emph{(i.b)} corresponds to Proposition \ref{art2:prop:biascontr}\emph{(i)} whereas  \emph{(i.b)}' corresponds to Proposition \ref{art2:prop:biascontr}\emph{(ii)}. Let us assume that 
\begin{equation*}
    \gamma_j=\gamma_0 2^{-j}\quad \textnormal{and}\quad T_j=T_0 2^{-(1-\rho) j},    
\end{equation*}
where $\rho\ge 1/2$, $\gamma_0$ and $T_0$ are positive numbers which will calibrated further. With these choices, the above conditions read (up to universal constants):
\begin{equation}\label{art2:eq:troisconditionsbis}
    \begin{cases}
        \displaystyle{\emph{(i.a)}: \quad  T_0 \ge 2^{(1-\rho) J} \sqrt{\mathfrak{C^{(1)}_{\mathrm{bias}}}}\varepsilon^{-1}} &\displaystyle{\emph{(i.b)}: 2^J\ge \mathfrak{C^{(2,1)}_{\mathrm{bias}}} \pas_0 \varepsilon^{-2} \quad \textnormal{or} \quad \displaystyle{(i.b)'}:\quad 2^J\ge \sqrt{\mathfrak{C^{(2,2)}_{\mathrm{bias}}}} \pas_0 \varepsilon^{-1},}\\
          \displaystyle{\emph{(ii.a)}:\quad  {T_0}\ge  \mathfrak{C_{\rm var}^{(2)}}\varepsilon^{-2},}& \displaystyle{ \emph{(ii.b)}: \quad T_0\ge \mathfrak{C_{\rm var}^{(1)}}(\kappa,\delta){\gamma_0^{1-\frac{1}{\kappa(1-\delta)}}\sum_{\ind=1}^\lev  2^{(-\rho+\frac{1}{\kappa(1-\delta)})\ind} \varepsilon^{-2}.}}
        \end{cases}
\end{equation}

\noindent \textit{Proof of Theorem \ref{art2:theo:maintheo2}} \emph{(i)} In this case, we have to calibrate the parameters according to \emph{(i.a)}, \emph{(i.b)}, \emph{(ii.a)} and \emph{(ii.b)} are satisfied. First, for \emph{(i.b)}, we need $ 2^{ J}\ge \mathfrak{C^{(2,1)}_{\mathrm{bias}}} \pas_0 \varepsilon^{-2}$ so we can set:
\begin{equation*}
    J=\left\lceil  \log_2\left(\mathfrak{C^{(2,1)}_{\mathrm{bias}}} \pas_0 \varepsilon^{-2}\right)\right\rceil.
\end{equation*}
Then, set $\rho=1/2$. With the above value of $J$ and the condition $\kappa>2/(1-\delta)$, 
\begin{equation*}
    2^{\frac{J}{2}}\le\sqrt{2\gamma_0\mathfrak{C^{(2,1)}_{\mathrm{bias}}}}\varepsilon^{-1}\quad\textnormal{and}\quad \sum_{\ind=1}^\lev  2^{(-\frac{1}{2}+\frac{1}{\kappa(1-\delta)})\ind}\lesssim_{\kappa,\delta} 1.    
\end{equation*}
Hence, \emph{(i.a)}, \emph{(ii.a)}, \emph{(ii.b)} are satisfied (up to a constant depending on $\kappa$, $\delta$ and $r$ only) if 
\begin{equation*}
    T_0\ge \mathfrak{C} \varepsilon^{-2} \quad\textnormal{with}\quad\mathfrak{C}=\max\left(\sqrt{\gamma_0\mathfrak{C^{(1)}_{\mathrm{bias}}}\mathfrak{C^{(2,1)}_{\mathrm{bias}}}},\mathfrak{C_{\rm var}^{(2)}},\mathfrak{C_{\rm var}^{(1)}}(\kappa,\delta)\gamma_0^{1-\frac{1}{\kappa(1-\delta)}}\right).    
\end{equation*} 
The complexity of the procedure is then:
\begin{equation*}
    {\cal C}({\cal Y})=2\left(\frac{T_0}{\gamma_0}+\sum_{j=1}^J \frac{T_j}{\gamma_j}\right)\lesssim \frac{T_0}{\gamma_0} 2^{\frac{J}{2}}\lesssim_{r,\delta,\kappa} \frac{\mathfrak{C} \sqrt{\mathfrak{C^{(2,1)}_{\mathrm{bias}}}}\varepsilon^{-3}}{\sqrt{\gamma_0}}.    
\end{equation*}
To deduce the result, it is now enough to remark that with $\kappa=\frac{2(1+\delta)}{1-\delta}$, $\gamma_0\le 1/(4L)$,
\begin{equation*}
    \mathfrak{C}\le(\un{c}^{-\frac{3}{4}}\vee \un{c}^{-\frac{5}{2}-\frac{2\delta}{1-\delta}})\Psideux^{1+3r+\frac{3\delta r}{1-\delta}}\quad \textnormal{so that}\quad  \mathfrak{C}\sqrt{\mathfrak{C^{(2,1)}_{\mathrm{bias}}}}\le \sqrt{L}(\un{c}^{-\frac{5}{4}}\wedge \un{c}^{-{\frac{7}{2}}-\frac{3\delta}{1-\delta}})\Psideux^{\frac{3}{2}+\frac{9}{2}r+\frac{4\delta r}{1-\delta}},    
\end{equation*}
as soon as $\delta <1/3$. The main result then follows from a change of variable replacing $\delta$ by $\tilde{\delta}= a\delta$ with $a$ small enough.
\bigskip \newline \emph{(ii)} In this case, we have to calibrate the parameters according to \emph{(i.a)}, \emph{(i.b)}', \emph{(ii.a)} and \emph{(ii.b)}. First, for \emph{(i.b)}', we need $ 2^{ J}\ge \sqrt{\mathfrak{C^{(2,2)}_{\mathrm{bias}}}} \pas_0 \varepsilon^{-1}$ so we can set:
\begin{equation*}
    J=\left \lceil  \log_2\left(\sqrt{\mathfrak{C^{(2,2)}_{\mathrm{bias}}}} \pas_0 \varepsilon^{-1}\right)\right \rceil.    
\end{equation*}
Then, if $\rho>\frac{1}{\kappa(1-\delta)}$, \emph{(i.a)}, \emph{(ii.a)}, \emph{(ii.b)} are satisfied (up to a constant depending on $\kappa$, $\delta$ and $r$ only) if 
\begin{equation*}
    T_0\ge \widetilde{\mathfrak{C}} \varepsilon^{-2} \quad\textnormal{with}\quad \widetilde{\mathfrak{C}}=\max\left(\sqrt{\mathfrak{C^{(1)}_{\mathrm{bias}}}}\left(\gamma_0^2\mathfrak{C^{(2,2)}_{\mathrm{bias}}}\right)^{\frac{1-\rho}{2}},\mathfrak{C_{\rm var}^{(2)}},\mathfrak{C_{\rm var}^{(1)}}(\kappa,\delta)\gamma_0^{1-\frac{1}{\kappa(1-\delta)}}\right),  
\end{equation*}
and the complexity of the procedure satisfies:
\begin{equation*}
    {\cal C}({\cal Y})\lesssim \frac{T_0}{\gamma_0} 2^{\rho{J}} \lesssim_{r,\delta,\kappa,\rho} \frac{\widetilde{\mathfrak{C}} (\mathfrak{C^{(2,2)}_{\mathrm{bias}}})^{\frac{\rho}{2}}\varepsilon^{-2-\rho}}{\gamma_0^{\rho}}.    
\end{equation*}
Set $\kappa=\frac{1+\delta}{\rho(1-\delta)}=\frac{1}{\rho}+\frac{2\delta}{\rho(1-\delta)}$. For $\delta$ small enough and $\gamma_0\le 1/L$,
\begin{equation*}
    \gamma_0^{1-\frac{1}{\kappa(1-\delta)}}\mathfrak{C_{\rm var}^{(1)}}(\kappa,\delta)\le (\un{c}^{-3\rho(1-\frac{\delta}{1-\delta})}\vee \un{c}^{-(1+3\rho)-5\delta})\Psideux^{1+ (4-2\rho) r+\frac{4\delta r}{1-\delta}}.    
\end{equation*}
Using that $\gamma_0\le 1/(4L)$,
\begin{equation*}
    \sqrt{\mathfrak{C^{(1)}_{\mathrm{bias}}}}\left(\gamma_0^2\mathfrak{C^{(2,2)}_{\mathrm{bias}}}\right)^{\frac{1-\rho}{2}}\le L^{\frac{\rho}{2}}(\un{c}^{-\frac{1}{4}-\frac{1-\rho}{1-\delta}}\vee \un{c}^{-\frac{3}{2}-\frac{1-\rho}{1-\delta}-\frac{2\delta}{1-\delta}}) \Psideux^{1-\frac{\rho}{2}+\frac{5-2\rho}{2} r+(3-\rho)\frac{\delta r}{1-\delta}}.    
\end{equation*}
Using that $\rho\le 1/2$, this yields (at the price of replacing $\delta$ by $a\delta$ for $a$ small enough)
\begin{equation*}
    \widetilde{\mathfrak{C}}\le L^{\frac{\rho}{2}} \left(\un{c}^{-(\frac{5}{4}-\rho)\wedge(3\rho)+\delta}\vee \un{c}^{-\frac{5}{2}-\delta}\right)\Psideux^{1+(4-2\rho+\delta) r}.    
\end{equation*}
The result follows.
\section*{Acknowledgements} The present author is deeply grateful to F.Panloup for his numerous suggestions and discussions that improved the quality of this paper. The author also thanks the SIRIC ILIAD Nantes-Angers program supported by the French National Cancer Institute (INCA-DGOS-Inserm 12558 grant) for funding  M. Eg\'ea's Ph.D. thesis.
\bibliographystyle{alpha}
\bibliography{bibliomax.bib}
\appendix
\section{Proof of Proposition \ref{art2:prop:Eulexponentcontrol}}\label{art2:appen:appendixA}
Consider the process $\left(\Xge_t\right)_{t >0}$ defined by $\Xge_t=x - t\nabla \pot(x) + \sigma \sqrt{t}Z$ where 
$Z=(Z_1,\ldots,Z_d) \sim \mathcal{N}(0,\mathrm{Id}_{\ER^d})$. By a Taylor expansion with integral remainder of the function $\pot$ we have
\begin{align}
    \pot(\Xge_t)&  = \pot(x) + \langle \nabla \pot(x),\Xge_t-x \rangle + \int_0^1\langle  D^2 \pot(\xi_\lambda)(\Xge_t-x),\Xge_t-x \rangle \mathrm{d}\lambda
    \nonumber \\ 
    & \le  \pot(x) + \langle \nabla \pot(x),\Xge_t-x \rangle + \left| \Xge_t-x\right|^2  \int_0^1 \bar{\lambda}_{D^2 \pot(\xi_\lambda)}\mathrm{d}\lambda,\label{art2:eq:startingtaylor}
\end{align}
where $\xi_\lambda=\lambda \Xge_t + (1-\lambda)x$.
\bigskip \newline \emph{(i)} In this first part, we only assume that $\nabla \pot$ is $L$-Lipschitz so that $\bar{\lambda}_{D^2 \pot(\xi_\lambda)}\le L$.
Thus, since 
\begin{equation}\label{art2:eq:adeuxplusbdeux}
    | \Xge_t-x|^2\le 2 t^2 |\nabla U(x)|^2+2 t\sigma^2 |Z|^2,
\end{equation}
we get
\begin{equation*}
    \pot(\Xge_t)\le \pot(x) - t |\nabla \pot (x) |^2(1-2Lt)+ \sigma \sqrt{t} \langle \nabla \pot (x) , Z \rangle+2 L \sigma^2 t |Z|^2.    
\end{equation*} 
Let $\theta>0$ and define, 
\begin{equation} \label{art2:eq:LyapFunctionEul}
    f_{\theta}: x \mapsto e^{\theta \pot(x)}
\end{equation}
We deduce from the previous inequality that 
\begin{align*}
    \ES[f_{\theta}\left(\Xge_t \right)] &\le f_\theta (x) e^{- \theta t |\nabla \pot (x) |^2(1-2Lt)}\ES\left[e^{\theta \sigma \sqrt{t} \langle \nabla \pot (x) , Z \rangle  + 2 \theta L \sigma^2 t |Z|^2}\right]\\
    &\le f_\theta (x) e^{ -\theta t |\nabla \pot (x) |^2(1-2Lt)} \prod_{i=1}^d \ES[e^{\alpha_i Z_i+\beta_i |Z_i|^2}],
\end{align*}
where $\alpha_i=\theta \sigma \sqrt{t} \partial_i \pot(x)$ and $\beta_i=2 \theta L \sigma^2 t$. A standard computation shows that  for any $u\in\mathbb{R}$ and any $v<1/2$,
\begin{equation}\label{art2:eq:exactformula}
    \mathbb{E}_{Z_1\sim{\cal N}(0,1)} [e^{u Z_1+v Z_1^2}]=\frac{1}{\sqrt{1-2v}} e^{\frac{u^2}{2(1-2v)}}.
\end{equation}
Thus, 
\begin{equation*}
    \prod_{i=1}^d \ES[e^{\alpha_i Z_i+\beta_i |Z_i|^2}]= \left(\frac{1}{1-2\theta L \sigma^2 t}\right)^\frac{d}{2} e^{\frac{\theta^2 \sigma^2 t |\nabla \pot(x)|^2}{2(1-2 \theta L \sigma^2 t)}}
\end{equation*}
This yields 
\begin{equation*}
    \ES[f_{\theta}\left(\Xge_t \right)]\le  f_\theta (x) e^{ -\theta t |\nabla \pot (x) |^2(1-2Lt)}e^{-\frac{d}{2}\log\left(1-2\theta L \sigma^2 t\right)} \exp \left(\frac{\theta^2 \sigma^2 t |\nabla \pot(x)|^2}{2(1-2 \theta L \sigma^2 t)}\right).    
\end{equation*}
Since, $\gamma L\le 1/4$ and $\theta \sigma^2\le 1/8$, $2\theta L \sigma^2 t\in[0,1/2]$ and we can use the elementary inequality 
$\log(1-x) \ge -2x$ for all $x \in [0,1/2]$ to obtain:
\begin{equation*}
    \ES[f_{\theta}\left(\Xge_t \right)]\le f_\theta (x) e^{ -c_{t,x,\theta}|\nabla \pot (x) |^2+2 d\theta L \sigma^2 t} \quad\textnormal{with}\quad
    c_{t,x}=\theta t\left(1-2 L t-\theta\sigma^2\right)\ge \frac{\theta t}{4}.    
\end{equation*}
For $t=\gamma$, this yields
\begin{equation*}
    \ES[f_{\theta}\left(\Xge_t \right)]\le f_\theta (x) e^{-\frac{\theta\gamma}{4} \left(|\nabla \pot (x) |^2-8 d L\right)}.
\end{equation*}
Let 
\begin{equation*}
    {\cal C}_M:=\{x\in\ER^d, |\nabla \pot (x) |^2-8 d L\le M\}.    
\end{equation*}
We get
\begin{align*}
    \ES[f_{\theta}\left(\Xge_t \right)]&\le f_\theta (x) e^{2\theta\gamma d L} 1_{\{x\in {\cal C}_M\}}+f_\theta (x)e^{-\frac{\theta\gamma M}{4}} 1_{\{x\in {\cal C}_M^c\}}\\
    &\le f_\theta (x)e^{-\frac{\theta\gamma M}{4}}+ \sup_{x \in {\cal C}_M} f_\theta (x) (e^{2\theta\gamma d L}-e^{-\frac{\theta\gamma M}{4}})
\end{align*}
In order to control $\sup_{x \in {\cal C}_M} f_\theta (x)$, one needs to include ${\cal C}_M$ in a level set of $U$. By \eqref{art2:eq:minogradU} and the fact that $\pot(x^\star)=1$, one checks that
\begin{equation*}
    {\cal C}_M\subset \left\{x\in\ER^d, U(x)\le \left(1+\frac{M+8dL}{2\un{c}}\right)^{\frac{1}{1-r}}\right\}    
\end{equation*}
so that with $t=\gamma$,
\begin{equation*}
    \ES[f_{\theta}\left(\Xge_\gamma \right)]\le  e^{-\frac{\theta\gamma M}{4}} f_\theta (x)+ \mathfrak{c}(M,\gamma,\theta)\quad \textnormal{with}\quad 
    \mathfrak{c}(M,\gamma,\theta)=e^{\theta \left(1+\frac{M+8dL}{2\un{c}}\right)^{\frac{1}{1-r}}} (e^{2\theta\gamma d L}-e^{-\frac{\theta\gamma M}{4}}).
\end{equation*}
An induction leads to
\begin{equation*}
    \sup_{n\ge0} \ES[f_{\theta}\left(\Xge_{n\gamma} \right)]\le f_\theta(x) +\frac{\mathfrak{c}(M,\gamma,\theta)}{1-e^{-\frac{\theta\gamma M}{4}}}.    
\end{equation*}
Setting $M=8dL$,  
\begin{equation*}
    \frac{\mathfrak{c}(M,\gamma,\theta)}{1-e^{-\frac{\theta\gamma M}{4}}}= \frac{e^{\theta\left(1+\frac{8dL}{\un{c}}\right)^{\frac{1}{1-r}}} \sinh(2\theta\gamma d L)}{e^{-\theta\gamma dL}{\sinh}(\theta\gamma d L)}=e^{\theta\left(1+\frac{8dL}{\un{c}}\right)^{\frac{1}{1-r}} +\theta dL}{\cosh}(\theta\gamma d L).    
\end{equation*} 
\emph{(ii)}  We now deal with the additional assumption $\Hrdeux$. The idea is now to refine the controls by taking into account that the largest eigenvalue also decreases at infinity. Unfortunately, this refinement will require additional technicalities and concentration arguments. By \eqref{art2:eq:startingtaylor}, \eqref{art2:eq:adeuxplusbdeux} and $\Hrdeux$,
\begin{equation*}
    \begin{split}
    \pot(\Xge_t)\le  \pot(x) - t |\nabla \pot (x) |^2 (1-2tL)  +\sigma \sqrt{t} \langle \nabla \pot (x) , Z \rangle  + 2 \bar{c} \sigma^2 t |Z|^2 \int_0^1  \pot^{-r} (\xi_\lambda) \mathrm{d}\lambda.
    \end{split}
\end{equation*}
Keeping in mind that $f_\theta(x)=e^{\theta U(x)}$, this leads to
\begin{equation}\label{art2:eq:backtoftheta}
    \ES \left[f_{\theta}\left(\Xge_t \right)\right] \le f_\theta (x) \exp \left(-\theta t |\nabla \pot (x) |^2 (1-2tL)  \right)  \underbrace{\ES\left[\exp \left(\theta \sigma \sqrt{t} \langle \nabla \pot (x) , Z \rangle  + 2 \theta \bar{c} \sigma^2 t |Z|^2 \int_0^1  \pot^{-r} (\xi_\lambda) \mathrm{d}\lambda \right) \right]}_{\eqqcolon E}.
\end{equation}
In order to get a sharp bound of the expectation $E$, we choose to divide it into two parts depending on a parameter $K$ which will be calibrated below:
\begin{equation*}
    \begin{split}
       E & = \underbrace{  \ES\left[\exp \left(\theta \sigma \sqrt{t} \langle \nabla \pot (x) , Z \rangle  + 2 \theta \bar{c} \sigma^2 t |Z|^2 \right)1_{\{|Z|^2 \ge K\}}\right]}_{\eqqcolon E_1}
        \\ & +  \underbrace{ \ES\left[\exp \left(\theta \sigma \sqrt{t} \langle \nabla \pot (x) , Z \rangle  + 2 \theta \bar{c} \sigma^2 t |Z|^2 \int_0^1 \pot^{-r} (\xi_\lambda) \mathrm{d}\lambda \right)1_{\{|Z|^2 < K\}}\right]}_{\eqqcolon E_2},
    \end{split}
\end{equation*}
\textit{Bound for $E_1$ :} By Cauchy-Schwarz inequality we have 
\begin{equation*}
    \begin{split}
        E_1 & \le  \ES\left[\exp \left( 2\theta \sigma \sqrt{t} \langle \nabla \pot (x) , Z \rangle  + 4 \theta \bar{c}  \sigma^2 t |Z|^2 \right)\right]^{1/2} \Pr\left(|Z|^2 \ge K \right)^{1/2}
        \\ & \le \prod_{i=1}^d \ES_{Z_i \sim \mathcal{N}(0,1)} \left[\exp\left(2\theta \sigma \sqrt{t}\partial_i \pot(x) Z_i+4 \theta \bar{c} \sigma^2 t Z_i^2\right)\right]^{1/2}\Pr\left(|Z|^2 \ge K \right)^{1/2}.
    \end{split}
\end{equation*}   
By \eqref{art2:eq:exactformula}, this yields
\begin{equation*}
    \begin{split}
        E_1 & \le \left( \left(\frac{1}{1-8\theta \bar{c}  \sigma^2 t}\right)^{\frac{d}{2}} \exp \left(\frac{2\theta^2  \sigma^2 t |\nabla \pot(x)|^2}{1-8 \theta \bar{c} \sigma^2 t  } \right)\right)^{1/2}\Pr\left(|Z|^2 \ge K \right)^{1/2}
        \\ &  \le  \exp \left(-\frac{d}{4}\log\left(1-8\theta \bar{c} \sigma^2 t\right) + \frac{ 2\theta^2 t \sigma^2 |\nabla \pot(x)|^2}{1-8 \theta \bar{c}  t \sigma^2  } \right)\Pr\left(|Z|^2 \ge K \right)^{1/2}.
    \end{split}
\end{equation*}
By exponential Markov inequality and \eqref{art2:eq:exactformula} applied with $u=0$ and $v=1/4$,
\begin{equation*}
    \begin{split}
    E_1 & \le \exp \left(-\frac{d}{4}\log\left(1-8\theta \bar{c} \sigma^2 t\right) + \frac{ 2\theta^2 t \sigma^2 |\nabla \pot(x)|^2}{1-8 \theta \bar{c} t \sigma^2  } - \frac{K}{8}\right) \ES \left[e^{\frac{|Z|^2}{4}}\right]^{1/2}
    \\ & \le \exp \left(-\frac{d}{4}\log\left(1-8\theta \bar{c} \sigma^2 t\right) + \frac{ 2\theta^2 t \sigma^2 |\nabla \pot(x)|^2}{1-8 \theta \bar{c} t \sigma^2}   - \frac{K}{8}+\frac{d\log 2}{4}\right).
      \end{split}
\end{equation*}
Under the assumptions on $\theta$ and $\bar{c}$, $8\theta \bar{c} \sigma^2 t\le \frac{1}{4}$ for any $t\in[0,\gamma]$. Using the elementary inequality $\log(1-x) \ge -2x$ for all $x \in [0,1/2]$,
\begin{equation*}
    \begin{split}
    E_1 & \le \exp \left(4 d\theta \bar{c} \sigma^2 t + \frac{ 2\theta^2 t \sigma^2 |\nabla \pot(x)|^2}{1-8 \theta \bar{c} t \sigma^2  } - \frac{K}{8}+\frac{d\log 2}{4}\right)
    \\ & \le  \exp \left( \frac{ 2\theta^2 t \sigma^2 |\nabla \pot(x)|^2}{1-8 \theta \bar{c} t \sigma^2  } - \frac{K}{8}+ \frac{d}{8}+\frac{d\log 2}{4}\right).
    \end{split}
\end{equation*}
Using again that $\theta\sigma^2\le 1/8$ and $8\theta \bar{c} \sigma^2 t\le \frac{1}{4}$, we deduce  that
\begin{equation}\label{art2:eq:E1eulexp}
    E_1  \le \exp \left(\frac{ \theta t}{3}|\nabla \pot(x)|^2 - {c}_K\right),\quad\textnormal{with}\quad c_K=\frac{1}{8}\left(K- {d}\left(1+2\log 2\right)\right).
\end{equation}\noindent  \textit{ Bound for $E_2$ :} First, we have
\begin{equation*}
    \sup_{\xi \in [x,\Xge_t]}1_{\{|Z|^2 \le K \}} \pot^{-r} (\xi)  \le  \sup_{\xi \in B\left(x,t|\nabla\pot(x)|+\sqrt{t K \sigma^2}\right)} \pot^{-r} (\xi)\eqqcolon C_{\pot,K}(x),
\end{equation*}
 that gives
\begin{equation*}
    \begin{split}
        E_2 & \le \ES\left[\exp \left(\sigma \theta \sqrt{t}  \langle \nabla \pot (x) , Z \rangle + 2\sigma^2 \theta \bar{c} t |Z|^2 C_{\pot,K}(x) \right) \right]. 
          \end{split}
\end{equation*}
Using again  \eqref{art2:eq:exactformula}, this involves that
\begin{equation*}
    \begin{split}   
       E_2& \le \exp \left(-\frac{d}{2}\log\left(1-4\theta \bar{c} \sigma^2 t C_{\pot,K}(x) \right) + \frac{ \sigma^2 \theta^2 t |\nabla \pot(x)|^2}{2(1-4 \theta \bar{c} t \sigma^2 ) } \right)
        \\ & \le \exp \left(4 d\theta \bar{c} \sigma^2 t C_{\pot,K}(x)  + \frac{\sigma^2 \theta^2 t  |\nabla \pot(x)|^2}{2(1-4 \theta \bar{c} t \sigma^2)}  \right),
    \end{split}
\end{equation*}
where in the last line we used the inequality $\log(1-x) \ge -2x$ for all $x \in [0,1/2]$. Once again, under $\theta\sigma^2\le 1/8$ and $t\bar{c}\le 1/4$, one checks that $\sigma^2\theta(2(1-4 \theta \bar{c} t \sigma^2))^{-1}\le 2/7 \le 1/3$. Thus,
\begin{equation}\label{art2:eq:E2eulexp}
    E_2 \le \exp \left( 4 d \theta \bar{c} \sigma^2 t  C_{\pot,K}(x)+ \frac{t \theta}{3} |\nabla \pot(x)|^2\right).
\end{equation}
Plugging \eqref{art2:eq:E1eulexp} and \eqref{art2:eq:E2eulexp} into \eqref{art2:eq:backtoftheta}, we obtain
\begin{equation*}
    \begin{split}
        \ES \left[f_{\theta}\left(\Xge_t \right)\right]  & \le f_\theta (x) \exp \left(-\theta t |\nabla \pot (x) |^2 \left((1-2t L-\frac{1}{3} \right)  \right)\left( e^{-c_K}+ e^{4 d \sigma^2 \theta \bar{c} t  C_{\pot,K}(x)} \right)
        \\ & \le f_\theta (x)e^{-c_K} + f_\theta (x) \exp \left(- t \left( \frac{\theta}{6} |\nabla \pot (x) |^2  - 4d\sigma^2\theta \bar{c}  C_{\pot,K}(x) \right) \right),
    \end{split}
\end{equation*}
where in the second line, we used that $2tL\le 1/2$.
Now, let us follow the same strategy as in the first case by setting \begin{equation*}
    \bar{{\cal C}}_M=\left\{x \in \ER^d ; \frac{\theta}{6} |\nabla \pot (x) |^2 -4d\sigma^2\theta \bar{c}  C_{\pot,K}(x) \le M\right\}.    
\end{equation*}
Then
\begin{equation}\label{art2:eq:avantiteration}
    \begin{split}
        \ES \left[ f_\theta \left(\Xge_t \right)\right] & \le f_\theta (x)e^{-c_K} + f_\theta(x) e^{-tM} 1_{\bar{\cal C}_M^c}(x) + f_\theta(x) \exp \left(4t d\sigma^2\theta \bar{c}  C_{\pot,K}(x) - \frac{\theta t}{6}  \left| \nabla \pot(x)\right|^2 \right) 1_{\bar{\cal C}_M}(x)
        \\ & \le  f_\theta (x)\left( e^{-c_K} + e^{-tM}\right)  + f_\theta(x) \left(\exp \left(4t d\sigma^2\theta \bar{c}  C_{\pot,K}(x) - \frac{\theta t}{6} \left| \nabla \pot(x)\right|^2 \right) -   e^{-tM} \right) 1_{\bar{\cal C}_M},
    \end{split}
\end{equation}
Now, let us show that $\bar{\cal C}_M$ is included in a level set of $\pot$. First, following the arguments which lead to \eqref{art2:eq:minogradU}, one is also able to show that
\begin{equation}\label{art2:eq:majogradU}
    |\nabla U(x)|^2\le \frac{2\bar{c}}{1-r} U^{1-r}(x),
\end{equation}
which in turn implies that $|\nabla \sqrt{\pot}|$ is bounded by $ \sqrt{\frac{\bar{c}}{2(1-r)}}$.  Thus, $\sqrt{U}$ is $ \sqrt{\frac{\bar{c}}{2(1-r)}}$-Lipschitz\footnote{$\sqrt{U}$ is also $\sqrt{L/2}$-Lipschitz. This could be alternatively used in this proof.} and 
\begin{equation*}
    \begin{split}
        \forall \xi \in B\left(x,t|\nabla\pot(x)|+\sqrt{t K \sigma^2}\right),    \quad     \pot^{1/2}(\xi)   \ge \pot^{1/2}(x) - \sqrt{\frac{\bar{c}}{2(1-r)}} |\xi-x|  .
    \end{split}
\end{equation*}
By the  triangular inequality we get 
\begin{align*}
    C_{\pot,K}(x) &\le \pot^{-r}(x) \left( 1- t  \frac{|\nabla \pot (x)|}{\pot^{1/2}(x)} \sqrt{\frac{\bar{c}}{2(1-r)}} -  \pot^{-1/2}(x) \sqrt{\frac{t K\sigma^2 \bar{c}}{2(1-r)}}\right)^{-2r},\\
    &\le \pot^{-r}(x) \left( 1-  \frac{t\bar{c}}{(1-r)} -  \pot^{-1/2}(x) \sqrt{\frac{\gamma K\sigma^2 \bar{c}}{2(1-r)}}\right)^{-2r},
\end{align*}
where in the second line, we used \eqref{art2:eq:majogradU} and the fact that $t\in[0,\gamma]$. Since $\frac{t\bar{c}}{1-r}\le \frac{1}{4}$, one can check here that if $\pot (x) \ge \frac{ 8 \gamma K\sigma^2 \bar{c}}{1-r}$, we have for every $t\in[0,\gamma]$,
\begin{equation*}
    1- t \frac{\bar{c}}{(1-r)} -  \pot^{-1/2}(x) \sqrt{\frac{\gamma K\sigma^2 \bar{c}}{2(1-r)}}\ge \frac{1}{2} \Longrightarrow C_{\pot,K}(x) \le 2 \pot^{-r}(x). 
\end{equation*}
With the help of \eqref{art2:eq:minogradU}, this implies that 
\begin{equation}\label{art2:eq:inclusion}
    \left\{x \in \ER^d ; \frac{\theta \un{c}}{3(1-r)} (\pot^{1-r}(x)-\pot^{1-r}(x^\star)) -8d\sigma^2\theta \bar{c}  \pot^{-r}(x)> M\right\}\bigcap\left\{x\in\ER^d;\pot (x) \ge \frac{ 8 \gamma K\sigma^2 \bar{c}}{1-r}\right\}\subset  {\cal C}_M^c.
\end{equation}
Keeping in mind that $\pot(x^\star)=1$, one can write
\begin{equation*}
    \frac{\theta \un{c}}{3(1-r)} (\pot^{1-r}(x)-\pot^{1-r}(x^\star)) -8d\sigma^2\theta \bar{c}  \pot^{-r}(x)=\frac{\theta \un{c}}{3(1-r)} \pot^{1-r}(x)\left(1- U^{r-1}(x)-\frac{24(1-r)d\sigma^2 \bar{c}}{\un{c}\pot(x) } \right),    
\end{equation*}
and one can deduce that this term is greater than $M$ if 
\begin{equation*}
    \frac{\theta \un{c}}{3(1-r)} \pot^{1-r}(x)>2M, \quad U^{r-1}(x)\le \frac{1}{4}\quad\textnormal{and}\quad  \frac{24(1-r)d\sigma^2 \bar{c}}{\un{c}\pot(x) }\le \frac{1}{4}.    
\end{equation*}
From these conditions and \eqref{art2:eq:inclusion}, we finally get
\begin{equation*}
    {\cal C}_M \subset \left\{x \in \ER^d ; \pot(x) \le \underbrace{\max \left(\frac{96d\sigma^2\bar{c}}{\un{c}}, \left(\frac{2M(1-r)}{\theta \bar{c}}\vee 4\right)^{\frac{1}{1-r}},  \frac{ 8\gamma K\sigma^2 \bar{c}}{1-r}\right)}_{\mathfrak{m}_K}\right\}.
\end{equation*}
Going back to \eqref{art2:eq:avantiteration} (and using that $C_{\pot,K}(x) \le 1$ for all $x$), we obtain
\begin{equation*}   
    \forall t\in[0,\gamma],\quad \ES \left[ f_\theta \left(\Xge_t \right)\right]  \le f_\theta (x)\left( e^{-c_K} + e^{-tM}\right) + e^{\theta{\mathfrak{m}_K}} \left( e^{4t \bar{c}\sigma^2\theta d}-e^{-tM}\right).
\end{equation*}
If we now assume that the parameters are chosen in such a way that 
\begin{equation}\label{art2:eq:ckcm}
    e^{-c_K} + e^{-\gamma M}<1,
\end{equation}
an induction leads to
\begin{equation}\label{art2:eq:fthetagen}
  \sup_{n\ge0} \ES_x \left[ f_\theta \left(\Xge_{n\gamma}\right) \right]  \le  f_\theta(x) + e^{\theta\bar{\mathfrak{m}}_K} \frac{e^{{4\gamma \bar{c}\sigma^2\theta d}}-e^{-\pas M}}{1-e^{-c_K} - e^{-\gamma M}}.
\end{equation}
From now on, assume that
\begin{equation*}
    K=d(1+2\log 2)+ \frac{ 4(1-r)}{\theta\gamma\bar{c}}\quad \textnormal{and}\quad M= \frac{4\theta \bar{c}}{1-r}.    
\end{equation*}
In this case, denoting by $c_r$ a positive constant which only depends on $r$ (and which may change from line to line), we have
\begin{equation*}
    \mathfrak{m}_K\le c_r\frac{d(1+\sigma^2)\bar{c}}{\un{c}}\quad\textnormal{and}\quad e^{-c_K}\le \frac{\theta\gamma\bar{c}}{1-r},    
\end{equation*}
where for  the second inequality, we used  that $e^{-x}\le 1/(2x)$ for $x\ge 1$ and $\frac{ 1-r}{2\theta\gamma\un{c}}\ge 1$ ($\theta\le 1$). One also checks that 
$\gamma M\le 1$ since $\theta\le 1$. 
By the inequality  $e^{-x}\le 1-\frac{1}{2}x$ for $x\in[0,1]$, this implies that
\begin{equation*}
    e^{-\gamma M}\le 1-\frac{2\gamma \theta \bar{c}}{1-r}\quad \Longrightarrow \quad e^{-c_K} + e^{-\gamma M}\le 1-\frac{\gamma \theta \bar{c}}{1-r}<1.    
\end{equation*}
Plugging into \eqref{art2:eq:fthetagen} and using the inequality $e^{-x}\ge 1-x$ for $x\ge0$, this yields
\begin{equation*}
    \sup_{n\ge0} \ES_x \left[ f_\theta \left(\Xge_{n\gamma}\right) \right]  \le  f_\theta(x) + e^{c_r\frac{\theta d(1+\sigma^2)\bar{c}}{\un{c}}} \frac{(1-r)(e^{{4\gamma \bar{c}\sigma^2\theta d}}-1)+2\gamma \theta \bar{c}}{{\gamma \theta \bar{c}}}.
\end{equation*}
To conclude, we need to separate two situations. If $\gamma\theta \bar{c} d\le 1$, then the inequality $e^x\le 1+ 4x$ for $x\in[0,1]$ leads to
\begin{equation*}
     \sup_{n\ge0} \ES_x \left[ f_\theta \left(\Xge_{n\gamma}\right) \right] \le  f_\theta(x) + e^{c_r\frac{\theta d(1+\sigma^2)\bar{c}}{\un{c}}} (16 \sigma^2 d+2).       
\end{equation*}
If $\theta\gamma \bar{c} d\ge 1 $ (so that $\theta\gamma \bar{c}\ge 1/d$), we obtain the following bound:
\begin{equation*}
    \sup_{n\ge0} \ES_x \left[ f_\theta \left(\Xge_{n\gamma}\right) \right]  \le f_\theta(x) + e^{c_r\frac{\theta d(1+\sigma^2)\bar{c}}{\un{c}}}  ({d e^{{4\gamma \bar{c}\sigma^2\theta d}}} +2)\le  f_\theta(x) + 3 d e^{c_r\frac{\theta  d(1+\sigma^2)\bar{c}}{\un{c}}} 
\end{equation*}
where in the last inequality, we used that $4\gamma \bar{c}\le 1\le  \frac{\bar{c}}{\un{c}}$. This concludes the proof.
\bigskip \newline \emph{(iii)} Noting that $\Psiun\lesssim_r \Psideux$, the first bound is obvious. For the second one, it is enough to note that under $\Hrun$, $(X_t)_{t\ge0}$ and $(\bar{X}_{n\gamma})_{n\ge0}$ converge in distribution to $\pi$ and $\pi^\gamma$ respectively so that with a uniform integrability argument combined with the first bound of \emph{(iii)}, the convergence holds for along functions $U^p$ for any $p>0$. 
\section{Proof of Proposition \ref{art2:prop:moment}}\label{art2:appen:appendixB}
The idea is to use Jensen inequality to derive controls of the polynomial moments from exponential moments. To this end, we begin with the following lemma:
\begin{lem}\label{art2:lem:genv}  Let $V$ denote a non negative random variable which satisfies
\begin{equation*}
    \ES[e^{\theta V}]<e^{a}+\rho e^b\quad \textnormal{for  positive $\theta$, $a$, $\rho$ and $b$.}    
\end{equation*}
Then, for any $p\ge 1$,
\begin{equation}\label{art2:eq:generalcontrol}
    \ES[V^p]\le \theta^{-p} \left(p-1+a+b+\log(2\rho)\right)^p.
\end{equation}
\end{lem}
\begin{proof}
Let $p\ge1$ and remark that
\begin{equation*}
    \ES[V^p]=\theta^{-p} \ES\left[ \log^p(e^{\theta V})\right]\le \theta^{-p} \ES\left[ \log^p(e^{p-1+\theta V})\right].    
\end{equation*}
The function $x \mapsto  \log^p(x)$ being concave on $[e^{p-1}, +\infty)$, we  deduce from the Jensen inequality that
\begin{equation*}
    \ES[V^p]\le \theta^{-p} (p-1+\log(\ES[e^{\theta V}]))^p\le \theta^{-p} \left(p-1+\log\left(e^{a}+\rho e^b\right)\right)^p.  
\end{equation*}
The lemma then follows from the following inequality: for $a,b>0$ and $\rho \ge 1$, $\log(e^a+\rho e^b ) \le a + b+ \log(2\rho)$ (since $e^a+\rho e^b\le 2\rho e^{a+b}$).
\end{proof}
\noindent \textit{Proof of Proposition \ref{art2:prop:moment}.} Let us consider separately the continuous and discrete cases. Let us also remark that it is enough to prove the result for $p\ge1$ (when $p\le 1$, one can use the bound obtained for $p=1$ combined with the Jensen inequality).
\bigskip \newline \emph{(i)} Owing to Proposition \ref{art2:prop:exponentcontrol}, we can apply Lemma \ref{art2:lem:genv} with $\theta=\sigma^{-2}$, $V= U(X_t)$, $a=\frac{U(x)}{\sigma^2}$ and 
\begin{equation*}
    (\rho,b)=\begin{cases} \left(1,\frac{1}{\sigma^2}(1+\frac{dL}{2\un{c}})^{\frac{1}{1-r}}\right)&\textnormal{under $\Hrun$}\\
    \left(\frac{d L }{2\un{c}},\frac{1}{\sigma^2} (4^{\frac{1}{1-r}}\vee \frac{d \bar{c}}{\un{c}})\right)&\textnormal{under $\Hrun$ and $\Hrdeux$.}
    \end{cases}
\end{equation*}
When only $\Hrun$ holds, this yields
\begin{equation*}
    \ES_x[\pot^p(X_t)]\le \left(\sigma^{2}(p-1)+U(x)+\left(1+\frac{dL}{2\un{c}}\right)^{\frac{1}{1-r}}\right)^p\le c_p\left(U^p(x)+\left((1+\sigma^2)\left(1+\frac{dL}{2\un{c}}\right)^{\frac{1}{1-r}}\right)^p\right).
\end{equation*}
Under $\Hrun$ and $\Hrdeux$, we get
\begin{equation*}
    \ES_x[\pot^p(X_t)]\le \left(\sigma^{2}(p-1)+\sigma^2\log\left(\frac{d L }{\un{c}}\right)+U(x)+4^{\frac{1}{1-r}}\vee \frac{d \bar{c}}{\un{c}}\right)^p.   
\end{equation*}
Using that $\log(x)\le x$ and that $L\le \bar{c}$ (since $U(x)\ge 1$), the second bound follows.
\bigskip \newline \emph{(ii)} Assume $\Hrun$. Owing to Proposition \ref{art2:prop:Eulexponentcontrol}\emph{(i)} applied with $\theta=1/(8\sigma^2)$, we can use Lemma \ref{art2:lem:genv} with $V= U(\bar{X}_{n\gamma})$, $a=e^{\theta U(x)}$,
$\rho= 1$ and $b=\theta\left(1+\frac{8dL}{\un{c}}\right)^{\frac{1}{1-r}} +2\theta dL$ (we used that $\cosh(u)\le e^u$ for $u\ge0$).
This yields
\begin{align*}
    \ES_x[U^p(\bar{X}_{n\gamma})] &\le \theta^{-p}\left(p-1+\theta U(x)+\theta\left(1+\frac{8dL}{\un{c}}\right)^{\frac{1}{1-r}}+2\theta dL\right)^p\\
    &\le \left(\frac{p-1}{\theta}+ U(x)+\left(1+\frac{8dL}{\un{c}}\right)^{\frac{1}{1-r}}+2 dL\right)^p\\ 
    & \le  \left(U(x)+c_p (1+\sigma^2)\left(dL+\left(1+\frac{dL}{\un{c}}\right)^{\frac{1}{1-r}}\right)\right)^p.
\end{align*}
This yields the bound under $\Hrun$ only. For the bound under $\Hrun$ and $\Hrdeux$, Proposition \ref{art2:prop:Eulexponentcontrol}\emph{(ii)} applied with $\theta=1\wedge 1/(8\sigma^2)$ allows us to use Lemma \ref{art2:lem:genv} with $V= U(\bar{X}_{n\gamma})$, $a=e^{\theta U(x)}$,
$\rho= c$ ($c$ denoting a universal constant) and $b=c_r\frac{\theta d(1+\sigma^2)\bar{c}}{\un{c}}$. This yields
\begin{align*}
    \ES_x[U^p(\bar{X}_{n\gamma})]\le 
    \left(\frac{p-1+ \log (2c)}{\theta}+ U(x)+c_r\frac{ d(1+\sigma^2)\bar{c}}{\un{c}}\right)^p,
\end{align*}
and the last bound easily follows.
\end{document}